\documentclass[12pt]{article}
\usepackage{amsmath,amsthm,amscd,amssymb}
\usepackage[margin=.75 in]{geometry}
\usepackage{subfig}
\usepackage{graphicx}
\usepackage[bookmarks=false]{hyperref}
\usepackage[labelfont=bf]{caption}
\allowdisplaybreaks
\usepackage{color}
\usepackage{cancel}
\usepackage{bbm}
\usepackage{dsfont}
\usepackage{hyperref}

\allowdisplaybreaks
\usepackage{mark_style}

\theoremstyle{plain}

\newtheorem{thm}{Theorem}[section]

\newtheorem{cor}[thm]{Corollary}
\newtheorem{prop}[thm]{Proposition}
\newtheorem{lem}[thm]{Lemma}
\theoremstyle{definition}
\newtheorem{defn}[thm]{Definition}
\newtheorem{rem}[thm]{Remark}
\newtheorem{ex}[thm]{Example}

\numberwithin{equation}{section}

\newcommand{\mc}{\mathcal}
\newcommand{\mb}{\mathbb}

\newcommand{\R}{\mathbb{R}}

\newcommand{\x}{\mathbf{x}}
\newcommand{\X}{\mathbf{X}}
\newcommand{\z}{\mathbf{z}}

\newcommand{\ba}{\begin{aligned}}
\newcommand{\ea}{\end{aligned}}
\newcommand{\bt}{\begin{thm}}
\newcommand{\et}{\end{thm}}
\newcommand{\bc}{\begin{cor}}
\newcommand{\ec}{\end{cor}}
\newcommand{\bl}{\begin{lem}}
\newcommand{\el}{\end{lem}}
\newcommand{\bpf}{\begin{proof}}
\newcommand{\epf}{\end{proof}}
\newcommand{\bpb}{\begin{problem}}
\newcommand{\epb}{\end{problem}}
\newcommand{\bd}{\begin{defn}}
\newcommand{\ed}{\end{defn}}
\newcommand{\bn}{\begin{note}}
\newcommand{\en}{\end{note}}
\newcommand{\bp}{\begin{prop}}
\newcommand{\ep}{\end{prop}}
\newcommand{\bex}{\begin{exercise}}
\newcommand{\eex}{\end{exercise}}
\newcommand{\ben}{\begin{enumerate}}
\newcommand{\een}{\end{enumerate}}
\theoremstyle{plain}

\title{The Dyson and Coulomb games}
\author{Ren\'{e} Carmona, Mark Cerenzia\thanks{Corresponding author}, Aaron Zeff Palmer\footnote{Electronic addresses: \texttt{rcarmona@princeton.edu, cerenzia@princeton.edu, azp@math.ubc.ca}}}
\date{}

\begin{document}

\maketitle

\begin{abstract}
We introduce and investigate certain $N$ player dynamic games on the line and in the plane that admit Coulomb gas dynamics as a Nash equilibrium. 
Most significantly, we find that the universal local limit of the equilibrium
is sensitive to the chosen model of player information in one dimension but not in two dimensions.
We also find that players can achieve game theoretic symmetry through selfish behavior despite
non-exchangeability of states, which allows us to establish strong localized convergence 
of the $N$--Nash systems to the expected mean field equations against locally optimal player ensembles, i.e., those exhibiting the same local limit as the Nash--optimal ensemble. In one dimension, this convergence notably features a nonlocal--to--local transition in the population dependence of the $N$--Nash system.
\end{abstract}

\tableofcontents

\section{Introduction}

In the random matrix theory (RMT) community, there is a well known study \cite{buses1} by physicists Krb\'{a}lek and \v{S}eba arguing that 
the spacing and arrival statistics of buses on a route in Cuernavaca, Mexico are well-described by the local statistics of eigenvalues of a random matrix belonging to
the Gaussian unitary ensemble (GUE), which have a \emph{repulsive} density on $\mb{R}^N$, $N \geq 2$, proportional to
\begin{equation} \label{density}
	\prod_{1 \leq k < \ell < N}|x^\ell - x^k|^\beta \cdot \exp \left \{ -\frac{N-1}{2}\sum_{i=1}^{N} (x^i)^2 \right \},  \ \ \beta = 2.
\end{equation} 
Although the repulsion parameter choice $\beta=2$ is related to the very special algebraic structure 
of determinantal correlations and to ``symmetry class" (see Chapter 1 of Forrester \cite{pf1}),
there have been many studies identifying the emergence of such statistics for general $\beta>0$, often where repulsive dynamics are natural: 
among parked cars \cite{parkedcars, umich, parkedcars2, goegue}, pedestrians \cite{pedestrian}, perched birds \cite{perchedbirds},
and even the New York City subway system \cite{nycsubway} (such statistics have also appeared in a geographical study \cite{france} of France, in genetics \cite{genes}, 
and notably among gaps between zeros of the Riemann zeta function \cite{primes}, which has already generated much research in number theory).

There have been some direct attempts to help explain such observational studies through rigorous mathematics,
such as Baik--Borodin--Deift--Suidan \cite{buses2} and Baik \cite{umich2}, 
but a common theme among these real-world systems has largely been ignored: 
they are all \emph{decentralized}.
Indeed, a striking aspect of the NYC subway study \cite{nycsubway} is that the MTA imposes a schedule on subway cars,
quite in contrast to the Mexican bus system, 
and yet Jagannath-Trogdon \cite{nycsubway} observe that even modest elements of \emph{individual} control can still produce RMT statistics.
The numerical physics paper \cite{pw1} of Warcho{\l} appears to be the only study of an agent-based model prior to our work here.

Motivated to prove rigorous theorems on this implicit link, 
we introduce a dynamic $N$ player game whose closed and open loop models
are explicitly solvable with Nash--optimal trajectories given by Dyson Brownian motion \eqref{dbm}, first introduced in \cite{theOG} by Freeman Dyson;
we accordingly call it the \emph{Dyson game}. 
More precisely, players in this prototype game aim to minimize a long time average (``ergodic") cost based on their distance from the origin
and on the reciprocal squared distance between one another (similar to Calogero--Moser--Sutherland models; see Remark \ref{cmsremark} below).
Essentially the same construction for logarithmic interactions holds in two (and higher) dimensions,
but there is an additional cost term incentivizing collinearity based on the reciprocal squared diameter of the circumcircle 
of the triangle formed with any two other players.
We refer to this two dimensional extension as the \emph{Coulomb game} since the Nash optimal trajectories are given
by planar Coulomb dynamics, studied recently
by Bolley-Chafa\"{i}-Fontbona \cite{bolley2018dynamics}
and Lu-Mattingly \cite{lu2019geometric}.

Merely constructing an agent--based model or a genuine player--based game yielding Coulomb interactions is not difficult, 
but it is significant for us to be able to identify how the solution depends on player information
and how the freedom to act individually can achieve ``game theoretic symmetry" 
despite the natural non-exhangeability in equilibrium. This latter feature is qualitatively consistent with the motivating example of 
the buses of Cuernavaca as well as the other observational studies above. It turns out such game theoretic symmetry \emph{fails} in the open loop model
but is present in the closed loop model of ``full information" (see the end of Section \ref{verificationtheorem}), allowing us to pursue strong ``localized" convergence of equations (see the main theorems stated in Section \ref{mainresults}).

Open loop models are often easier to analyze because opponent reactions may not be considered by players in their search for a Nash equilibrium.
The open loop model for the Dyson and Coulomb games is further simplified by a \emph{potential structure} (Lemmas \ref{potentialgame} and \ref{2dpotentialgame}), which reduces the search to a single auxiliary global problem: 
a ``central planner" can tell every player what to do and they end up not acting selfishly (though a priori they could). 
This puts us in the realm of classical statistical physics, but to continue the game theoretic interpretation,
the open loop Nash equilibrium prescribes higher repulsion to accommodate players densely packed near the origin.  
In contrast, the closed loop model realizes the option to behave selfishly through 
consideration of opponent reactions, leading to a less repulsive equilibrium that benefits players near the edge.
Thus, the closed loop equilibrium is more ``fair" in that all players incur the same cost, 
but this cost is higher than the average open loop cost; see Figure \ref{onlyfigure} for an illustration in the one dimensional case.

Further, both folklore and rigorous results of game theory suggest that the difference between these two models should disappear 
as $N \to \infty$ given mean field interactions; see Remark 2.27 and pgs.122, 212 in Carmona--Delarue \cite{bible1} for discussion of explicit solutions,
as well as the recent work of Lacker \cite{lackerinprep} for a theoretical approach. 
The intuition is that a single player cannot dramatically influence another through the empirical distribution if $N$ is large. Corollary \ref{2depsilonnash} below confirms (approximate) equilibria of the models ``converge together" for the Coulomb game; 
however, in the Dyson game, the highly singular dependence on the population allows nearby neighbors to have a large impact on a given player's cost, and
so the difference between the closed and open loop Nash equilibriums does \emph{not} disappear in the limit (cf. Figure \ref{onlyfigure}).
Consequently, in one dimension, the universal local limit of the equilibrium depends on player information, not just on the form of cost they face. 
We believe this result offers a new perspective on Problem 9 of Deift's list \cite{deift1} (though we do not construct a specific model for the parking problem); see the end of Section \ref{verificationtheorem} for discussion.

There has been a growing interest 
in explicit solutions to $N$ player games, in the mean field convergence problem, and in rank--based systems. 
As Lacker--Zariphopoulou \cite{thaleia1} recently point out, explicit solutions for $N$ player games are scarce, 
especially for the setting of full information.
A solvable prototype for this setting is the class of Linear--Quadratic (LQ) models, examples of which are reviewed in Section 2.4 of \cite{bible1}; 
see \cite{thaleia1} and references therein for some non--LQ but explicitly solvable models.
The work of Bardi \cite{bardi1} is an informative explicit case--study in the Gaussian ergodic setting,
but it is rarely remarked that the model works with ``narrow strategies" (see Fischer \cite{fischer1}).  
For the convergence problem, the early works \cite{bardi1,fischer1,lacker3} consider open loop and narrow strategies, but
much research on convergence for closed loop models with full information has been generated by the systematic approach of
Cardaliaguet--Delarue--Lasry--Lions \cite{cdll1}; see
the recent work of Lacker \cite{lackerinprep} and references therein.
Finally, there are many models in the literature that include costs depending on rank, e.g. \cite{bayrakter3, bayrakter1,lacker1, nutz2}, but 
the player states are still designed to be exchangeable.

The Dyson and Coulomb games exhibit many interesting properties that are atypical, if not new, for the literature on many player games. 
First, they furnish non--LQ but explicitly solvable $N$ player models that involve singular convolution transforms of the empirical distribution. 
Second, the Dyson game naturally features a nonlocal--to--local transition in the population argument of the $N$--Nash system as $N \to \infty$ 
(see Cardaliaguet \cite{nonlocaltolocal} for more general discussion of such transitions).
Third, the difference between the closed and open loop models of the Dyson game
does not vanish in the limit and thus the models exhibit different universal local limits.
Fourth and finally, we believe this paper is the first to work directly with the naturally ordered players in equilibrium
and to establish convergence of equations in such a strong localized sense.

We remark that consideration of the Dyson game
seems to have been anticipated somewhat in the bibliographical notes
of Chapter 23 of Villani \cite{villani2008optimal}. 
Indeed, pgs. 691-692 review Nelson's approach
to the foundations of quantum mechanics and give as an example the  Euler equation
with negative cubic pressure (cf. the mean field equation \eqref{openloopmasterequation} below). Matytsin \cite{matytsin1994large}
was the first to observe this equation arise in the qualitative description of some random matrix models, and his results were later made rigorous by Guionnet--Zeitouni \cite{guionnet2004first}, \cite{guionnet2002large} using large deviation techniques; see  Menon \cite{menon2017complex} for more discussion on this thread of literature.
These bibliographical notes of Villani conclude by broadly observing how this same class of such variational 
problems had recently arisen (unexpectedly at the time) in the initial work of Lasry-Lions \cite{mfg} on mean field games, who also in fact
reference this connection at the end of their Section 2.5.
Thus, we believe the direction we pursue here is quite natural
given the subsequent developments of many player game theory.

\subsection*{Outline}

After introducing frequently used notation in Section \ref{notation}, 
we review our main results in Section \ref{mainresults}, 
stating completely those we consider most significant. 
We then articulate the closed and open loop models for the Dyson game in Section \ref{dysongame} 
and use the solutions \eqref{ehjbsolution}, \eqref{potentialfunction} of the ergodic PDEs \eqref{dysonnash}, \eqref{opendysonnash} 
to prove Verification Theorems \ref{verificationthm}, \ref{openverificationthm} in Section \ref{verificationtheorem}.
In Section \ref{meanfieldapproximation}, we use the mean field analogs \eqref{masterfield},\eqref{meanfieldpotential} of these solutions to
\emph{guess} the limiting equations \eqref{dysonmaster},\eqref{openloopmasterequation} on Wasserstein space; 
most notably, the master equation \eqref{dysonmaster} features a local coupling.
Then, using the theory of gradient flows on Wasserstein space \cite{ambrosio1}, 
Section \ref{MFGformulation} proves a Verification Theorem \ref{meanfieldverification} for the associated mean field game formulation. 
Finally, Section \ref{justifypassage} recovers the master equation \eqref{dysonmaster} from the $N$ Nash system \eqref{dysonnash} by integrating
sequences of equations against locally optimal ensembles; Section \ref{coulombgame} pursues the same for the two dimensional Coulomb game. 

\section{Notation} \label{notation}
Fix $N \geq 2$.
We often make use of the abbreviation ``$\sum_{k:k\neq i}$" for ``$\sum_{\substack{k=1 \text{,} k \neq i}}^N$".
We write $\mc{P}(\R)$ for the set of probability measures $\mu$ on $\R$, $\mc{P}_2(\R)$ for the subset with finite second moment, 
$\int_\R |x|^2 \mu(dx)<\infty$, and $\mc{P}^p(\R)$ for the subset with densities $m(x)$ in $L^p(\R)$, $p>0$. 
We also write $\mc{P}^p_2(\R) :=  \mc{P}_2(\R)\cap\mc{P}^p(\R) $.  If $T: \R \to \R$ is Borel measurable, 
we denote the push forward of $\mu \in \mc{P}(\R)$ by $T_\# \mu := \mu \circ T^{-1}$. We always assume
$\mc{P}(\R)$ to be topologized by weak convergence of probability measures and $\mc{P}_2(\R)$
to be endowed with the $r$--Wasserstein distance with $r=2$, defined by
$$
d_r(\mu,\nu) := \min_{\gamma \in \Gamma(\mu,\nu)} \left \{  \left ( \int_{\R \times \R} |x-y|^r \ \gamma(dx,dy)  \right )^{1/r}    \right \}, \ \ r \geq 1,
$$
where $\Gamma(\mu,\nu)$ is the set of couplings $\gamma$ for $\mu,\nu \in \mc{P}_2(\R)$.
For any $\mu \in \mc{P}(\R)$
and $f \in L^1(\mu)$, we write $\mu[f] := \int_{\R} f(x) \mu(dx)$ (and the same for higher dimensions).

To emphasize the nonlocal--to--local transition when passing to the limit in equations, 
we use Greek letters ``$\mu(dx),\nu(dx),\ldots$" for probability measures
and use Latin letters ``$m(x),n(x),\ldots$" for their densities. 
We also use bold symbols ``$\x,\X$" to indicate vectors in $\R^N$ and use ``$\z,\bZ$"
for vectors in $(\R^2)^N$; whether the symbols $\boldsymbol{\phi}, \boldsymbol{\alpha}$
are vectors in $\R^N$ or $(\R^2)^N$ will be clear from context.
Accordingly, for any $\x = (x^1,\ldots,x^N) \in \R^N$, we write 
$$\mu_\x^N := \frac{1}{N} \sum_{k = 1}^N \delta_{x^k} , \ \  \mu_\x^{N,i}:= \frac{1}{N-1} \sum_{k: k \neq i} \delta_{x^k}$$ 
for the ordinary and $i$th empirical distribution of $\x$, $1 \leq i \leq N$, respectively. 
We use similar notation for $\bz = (z^1, \ldots, z^N)$, $z^i \in \R^2$, $1 \leq i \leq N$.
For norms, we write ``$|z|$'' for $z \in \RR^2$ and ``$\Vert \bx \Vert$, $\Vert \bz \Vert$'' for $\bx \in \RR^N$ and $\bz \in (\RR^2)^N$.

For partial derivatives, we often use the abbreviations such as $\partial_i := \frac{\partial}{\partial x^i}$.
A functional $\mc{U} : \mc{P}_2(\R^d) \to (-\infty, +\infty]$ is said to have
a \emph{linear functional derivative}
if there exists a function $(\mu,x) \mapsto \frac{\delta \mc{U}}{\delta \mu}(\mu)(x)$ continuous on $\mc{P}_2(\R^d) \times \R^d$
such that for all $\mu, \nu \in \mc{P}_2(\R^d)$,
$$
\mc{U}(\mu) - \mc{U}(\nu)
 = \int_0^1 \int_{\R^d} \frac{\delta \mc{U}}{\delta \mu}
 ((1-t)\nu + t \mu)(x) (\mu - \nu)(dx) dt. 
$$
Chapter 10 of Ambrosio-Gigli-Savar\'{e} \cite{ambrosio1}
puts forth a theory of subdifferential calculus for functionals
on the Wasserstein space $\mc{P}_2(\R^d)$,
and one can often interpret their intrinsic notion of minimal selection ``$\partial^o \mc{U}(\mu)(x)$'' of subdifferential as
\begin{equation} \label{wassersteingradient}
\partial_\mu \mc{U}(\mu)(x) := \nabla_x \frac{\delta \mc{U}}{\delta \mu}(\mu)(x).
\end{equation}
From this point of view, one can refer to ``$\partial_\mu$" as the \emph{Wasserstein gradient}.
The same object was independently arrived at by Lions \cite{lionslectures}
using an extrinsic approach
and thus is also referred to as the \emph{L-derivative}; see Chapter 5 of 
Carmona-Delarue \cite{bible1} or Gangbo-Tudorascu \cite{gangbo2019differentiability} for a deeper discussion and comparison.

For any $\mu \in \mc{P}(\R)$,
consider a function $h: \R \to \R$ such that, for almost every $x \in \R$, the integral $\int_\R h(x-y) \mu(dy)$ exists or its principal
value exists. Denote this quantity ``$(h * \mu)(x)$", the \emph{convolution of $h$ with $\mu$}.  
We will find it convenient to set for $z \in \R^2 \setminus \{ (0,0) \}$
\begin{equation} \label{htransforms}
	h_0(z) := \log |z|, \ \  h_1(z): = \frac{z}{|z|^2}, \ \  h_2(z):=\frac{1}{|z|^2}.
\end{equation}
We retain the same definitions when evaluated at $x \in \R \setminus \{ 0 \}$, viewed
as embedded in $\R^2 \setminus \{ (0,0) \}$.
Hence, we may write the \emph{Hilbert Transform} as
\begin{equation} \label{hilberttransform}
	H\mu(x) := \text{p.v.} \int_{\R} \frac{\mu(dy)}{x-y} = \lim_{\epsilon \downarrow 0}\int_{|x-y| > \epsilon} \frac{\mu(dy)}{x-y} = (h_1 * \mu)(x).
\end{equation}
Recall there exists $A_p>0$ such that $\Vert H\mu \Vert_p \leq A_p \Vert m \Vert_p$ 
for all $\mu \in \mc{P}^p(\R)$, $1 < p < \infty$, with density $m(x)$ (see, e.g., Theorem 1.8.8 of Blower \cite{gb1}).
Define also the transform\footnote{There is no need for a principal value integral since $z/|z|^2$ is integrable near $z=(0,0)$ in $\RR^2$.}
\begin{equation} \label{coulombtransform}
	\mathcal{H}\mu(z) :=\int_{\R^2} \frac{z-w}{|z-w|^2} \mu(dw), \ \ z \in \R^2.
\end{equation}
and write
\begin{equation}
	\label{coulombtransformproduct}
\mathcal{H}[\mu \mathcal{H}\mu](z) 
:= \int_{\R^2} \left \langle \frac{z-w}{|z-w|^2} , \mathcal{H}\mu(w)  \right \rangle \mu(dw),
\end{equation}
where $\langle \xi, \eta \rangle$ stands for ordinary dot product of $\xi, \eta \in \RR^2$. Also we let $D(\xi,\eta)$ denote the diameter of the circumcircle of the triangle
determined by $\xi,\eta,$ and $(0,0)$ in $\RR^2$. 

For any $\beta > 0$, let $\mu_\beta \in \mc{P}(\R)$ denote the Wigner semicircle law with density
\begin{equation} \label{semicircledensity}
	m_\beta(x) := \frac{1}{\pi \beta} \sqrt{2\beta - x^2} \cdot \mathbbm{1}_{[-\sqrt{2\beta} , \sqrt{2\beta}]}(x). 
\end{equation}
Throughout the paper, $(\Omega, \mc{F}, \mb{F} = (\mc{F}_t)_{t \geq 0}, \mb{P})$ will denote a complete filtered probability space
supporting an $N$-dimensional Wiener process $(\mathbf{W}_t)_{t \geq 0}$ and supporting a $2N$-dimensional Brownian motion $(\mathbf{B}_t)_{t \geq 0}$ in 
$(\R^2)^N$ whose two dimensional components we denote by $B^i_t \in \R^2$.
We often write $X \sim \mu$ to mean the random variable $X$ has distribution $\mu \in \mc{P}(\R)$. 

Define the open Weyl chamber $$\mc{W}^N := \{ \x \in \R^N : x^1 < \cdots < x^N \}$$
and write $\overline{\mc{W}^N}$ for its closure. 
Similarly, we write
$$
\mc{D}^N := \{ \bz \in (\R^2)^N : z^i \neq z^j, \text{for all } 1 \leq i < j \leq N \}. 
$$
We let $B(r) := \{ z \in \RR^2 \  : \ |z| < r \}$ denote the open ball of radius $r \geq 0$.
Finally, we will need some linear ordering $\prec$ on $\RR^2 \approx \CC$.
Given such an ordering, restrict the domain $\mc{D}^N$ by defining
$$
\mc{D}^N_{\text{ordered}} := \{ \bz \in (\RR^2)^N : z^1 \prec \cdots \prec z^N , \ z^k \neq z^\ell, 1 \leq k < \ell \leq N \}.
$$
For a concrete example of such a linear ordering $\prec$ and an application,
see Example \ref{applicationofspiral} below.

\section{Statement and review of main results} \label{mainresults}
\subsection*{The Dyson Game}
Fix $\sigma >0$ and $C_2 \in \RR$. For $1 \leq i \leq N$, let 
	\begin{equation} \label{statecost}
		F^{N,i}(\x) :=\frac{(x^i)^2}{8} + C_2 \frac{(h_2 * \mu^{N,i}_\bx)(x^i)}{N-1} = \frac{(x^i)^2}{8} + \frac{C_2}{(N-1)^2} \sum_{k:k\neq i} \frac{1}{(x^i-x^k)^2} , \ \ \x \in \mc{W}^N.
	\end{equation}
and consider the ergodic $N$-Nash system
	\begin{equation} \label{dysonnash}
		-\frac{\sigma^2}{2(N-1)} \Delta_\x v^{N,i}(\x) + \sum_{k:k \neq i} \partial_kv^{N,k}(\x)\partial_kv^{N,i}(\x) + \frac{1}{2}(\partial_iv^{N,i})^2(\x) = F^{N,i}(\x)
		- \lambda^{N,i}, \ \ \  \x \in \mc{W}^N,  \ \ 1 \leq i \leq N.
	\end{equation} 
The game theoretic counterpart of these equations is the \emph{closed loop} model of the Dyson game, detailed in Section \ref{dysongame}, where
Lemma \ref{Nsubsolution} shows that if we can write $C_2 = \beta(\frac{3}{2}\beta - 2\sigma^2)/4$ for some $\beta \in \R$, then the ergodic value pairs 
	\begin{equation} \label{ehjbsolution}
		\begin{cases}
		v^{N,i}_\beta(\x) :=
			\frac{(x^i)^2}{4} - \frac{\beta}{2} (\log | \cdot | * \mu^{N,i}_\x)(x^i), & \x \in \mc{W}^N \\
			\lambda^{N,i}_\beta := \frac{\beta}{4}+\frac{\sigma^2}{4(N-1)}  & 
		\end{cases}, \ \ 1 \leq i \leq N,
	\end{equation} 
	form a classical solution to the $N$--Nash system \eqref{dysonnash} on $\mc{W}^N$.
	Lemma \ref{openloopehjbtheorem} similarly solves the \emph{open loop} model of the Dyson game
given the relationship $C_2=\beta (\beta - 2 \sigma^2)/4$. 
Now define 
	\begin{equation} \label{roots}
		\beta_{closed}(C_2) :=\frac{2}{3} \left ( \sigma^2 + \sqrt{\sigma^4 + 6 C_2}  \right ), \ \  
		\beta_{open}(C_2) :=  \sigma^2 + \sqrt{\sigma^4 + 4 C_2}. 
	\end{equation}
Then Theorems \ref{verificationthm}, \ref{openverificationthm} verify that \emph{$\beta/\sigma^2$-Dyson Brownian motion} $(\bX^*_t)_{t \geq 0}$, with components
\begin{equation} \label{dbm}
	dX_t^{*i} = -\partial_iv^{N,i}_\beta(\X^*_t)dt + \frac{\sigma}{\sqrt{N-1}}dW^i_t =  \left [ \frac{\beta}{2}(h_1 * \mu_{\X^*_t}^{N,i})(X^{*i}_t)
	- \frac{X^{*i}_t}{2} \right ] dt + \frac{\sigma}{\sqrt{N-1}}dW^i_t,
\end{equation}
serves as a \emph{closed loop} Nash equilibrium if $\beta = \beta_{closed}(C_2)$ and 
as an \emph{open loop} Nash equilibrium if $\beta = \beta_{open}(C_2)$. These optimality concepts are reviewed in Section \ref{dysongame}. 
In particular, $\beta_{closed}(C_2) \neq \beta_{open}(C_2)$ for any $C_2 \geq -\sigma^4/6$, and they do not converge together as $N \to \infty$; see the end of Section \ref{verificationtheorem} for more discussion. 
\begin{remark} \label{cmsremark}
	Lemma \ref{Nsubsolution} showing \eqref{ehjbsolution} solves
	\eqref{dysonnash} was the real starting point of 
	this paper: it indicates that the most basic element of exact solvability
	of Calogero-Moser-Sutherland models (cf. Proposition 11.3.1 of Forrester \cite{pf1}, especially the algebraic identities \eqref{keyidentity1}, \eqref{keyidentity2} below) is compatible with Nash optimality as expressed through the $N$ Nash system \eqref{dysonnash}.	
	We note that the form ``$\beta (\beta - 2 \sigma^2)/4$" is characteristic of 
the classical Calogero--Moser--Sutherland models (up to a constant factor),
but the new relationship ``$\beta ( \frac{3}{2} \beta - 2\sigma^2)/4$'' has the interpretation of 
yielding a smaller repulsion for a given such coefficient value $C_2$, which benefits players at the ``edge'' who have more space.
Letting $\sigma=1$, we recall the significance of $\beta=2$ as corresponding to the free fermion regime (see Section 11.6 of Forrester \cite{pf1}). Our calculations suggest
$\beta=4/3$ might admit an analogous interpretation and significance, which we plan to pursue in future work.
\end{remark}

Now observe we may write the solution \eqref{ehjbsolution} to the $N$--Nash system \eqref{dysonnash}
	in the form $v^{N,i}_\beta(\x) = U_\beta(x^i,\mu_\x^{N,i})$, where 
	\begin{equation} \label{masterfield}
		U_\beta(x,\mu):= \frac{x^2}{4} - \frac{\beta}{2} (\log | \cdot | * \mu)(x), \ \ (x, \mu) \in \R \times \mc{P}^p_2(\R), \ \ 1 < p < \infty.
	\end{equation} 
	Using equation (71) of Cardaliaguet--Porretta \cite{pcap1}
	to \emph{guess} the mean field analog of the $N$--Nash system \eqref{dysonnash},
Lemma \ref{masterequation} shows that the pair $(U_\beta(x,\mu), \frac{\beta}{4})$ forms a solution to a mean field equation we refer to as the \emph{Voiculescu--Wigner master equation} on $\R \times \mc{P}^p_2(\R)$, $2\leq p < \infty$: 
		\begin{equation} \label{dysonmaster}
			\int_{\R} \partial_\mu U(x,\mu)(z) \partial_xU(z,\mu)\mu(dz)  +  \frac{1}{2} |\partial_xU(x,\mu)|^2 
			= \frac{x^2}{8} + \frac{\pi^2\beta^2}{8}m(x)^2-\lambda, 
		\end{equation}
where we recall the definition \eqref{wassersteingradient} for the Wasserstein gradient ``$\partial_\mu$".
Relying on this result,
Theorem \ref{meanfieldverification} identifies the limiting flow $\bmu^*=(\mu^*_t)_{t \geq 0}$
of the empirical measures $(\mu_{\bX^*_t}^N )_{t \geq 0}$ of \eqref{dbm} as an equilibrium of the mean field game formulation, according to Definition \ref{mfgsolutions}.

Since the results we just reviewed above are standard principles (albeit in a nonstandard and singular setting requiring somewhat special formalism and arguments), we have left their full statements to the body of the paper.
We turn now to stating completely the most significant theorems.

If the reader compares the ergodic $N$--Nash system \eqref{dysonnash} with the master equation \eqref{dysonmaster}, they might be puzzled how to go from one to the other; in particular,
	it is unclear what should account for the change in the form of the cost (not only the nonlocal--to--local transition, but also the coefficients).
	To the point, we saw the diffusion parameter $\sigma^2$ is linked to the state cost $F^{N,i}(\x)$ of \eqref{statecost} through its relationship to the singular cost coefficient, $C_2 = \beta (\frac{3}{2}\beta - 2\sigma^2)/4$,
	but only the coupling parameter $\beta$ appears (explicitly) in the master equation \eqref{dysonmaster}. The astute reader may object that \eqref{dysonmaster} was merely a guess, so it might not be the right 
	mean field analog of \eqref{dysonnash}.
	It turns out that, despite vanishing, the diffusion term \emph{does} contribute in the limit but its contribution
	cancels with the local contributions from the drift--interaction terms, leaving
	only a local contribution from the control cost that accounts for the apparent discrepancy above.

To recover the master equation \eqref{dysonmaster} from the $N$-Nash system \eqref{dysonnash} rigorously,  
	we generalize in two ways the program outlined in Remark (x) after Theorem 2.3 of Lasry--Lions \cite{mfg}: 
	first, to serve as test functions, we work with a natural class of player ensembles that are \emph{locally optimal} for the Dyson game, and second we consider \emph{localized} convergence.
	By ``locally optimal,'' we mean we can recover the master equation \eqref{dysonmaster} by integrating the $N$--Nash system \eqref{dysonnash} 
	against ensembles sharing the same (universal) local limit as the Nash--optimal ensemble;
	by ``localized,'' we mean instead of working with an exchangeable system, 
	we classify the ranked players by their mean field location.
	
		More precisely, fix $\beta>\sigma^2 >0$ and let $V(x)$ be twice continuously differentiable with $V''(x) \geq c_V$ for all $x \in \R$ and some constant $c_V>0$. 
	Let $\X \in \mc{W}^N$ be distributed according to a \emph{generalized $\beta/\sigma^2$--ensemble}: 
	\begin{equation} \label{invariantmeasuregeneral}
		\mu^N_{\beta, V}(d\x) = m^N_{\beta, V}(\x) d \x = \frac{N!}{Z_{\beta, V}^N} \cdot 
		\prod_{1 \leq k < \ell \leq N} (x^\ell - x^k)^{\beta/\sigma^2} \cdot \exp \left \{ - \frac{N-1}{\sigma^2} \sum_{i=1}^N V(x^i) \right \} \mathbbm{1}_{\mc{W}^N}(\x) d\x,
	\end{equation}
	where $Z_{\beta,V}^N<\infty$ is a normalization constant. We also write
	$\mu_\beta^N, m_\beta^N$ for $\mu_{\beta, V}^N, m_{\beta,V}^N$ when $V(x)=x^2/2$.
	Recall that the local behavior of the player ensembles $\mu^N_{\beta,V}$ is dominated by the effective repulsion parameter $\beta/\sigma^2$ and thus $\mu^N_{\beta,V}$ shares the same local limit as the Nash--optimal Gaussian ensemble
	$\mu^N_\beta$. 
	Note that $\mu^N_{\beta,V}$ is the invariant distribution of the diffusion
\begin{equation} \label{generalizeddbm}
	dX^i_t = \left [  \frac{\beta}{2} (h_1*\mu_{\X_t}^{N,i})(X^i_t)  - \frac{V'(X^i_t)}{2} \right ] dt + \frac{\sigma}{\sqrt{N-1}} dW^i_t, \ \ 1 \leq i \leq N.
\end{equation}

By Theorems 4.4.1, 4.4.3.(i), and 5.4.3 of Blower \cite{gb1}, there exists a unique measure $\mu_{\beta,V} \in \mc{P}^2(\R)$, compactly supported on a single 
interval with density $m_{\beta,V}(x) \in L^2(\R)$, 
that satisfies the $1$--Wasserstein convergence $d_1(\mu_\X^N,\mu_{\beta,V}) \overset{N \to \infty}{\to} 0$ almost surely for $\X \sim \mu_{\beta,V}^N$, and
that satisfies the \emph{Euler--Lagrange} (or \emph{Schwinger--Dyson}) equation:
	\begin{equation} \label{eulerlagrange}
		\beta H\mu_{\beta,V}(x) - V'(x) = 0, \ \ \text{for all } \ x \in \text{supp}(\mu_{\beta,V}).
	\end{equation}
As above, we write $\mu_\beta, m_\beta$ for $\mu_{\beta,V}, m_{\beta,V}$ when $V(x)=x^2/2$.
Finally, we write 
$$
\gamma^q = \gamma_{\beta,V}^q := \inf \{ x \in \R : \mu_{\beta,V}((-\infty,x]) \geq q \},  \ \ q \in [0,1].
$$

\begin{prop}\label{conjecture}
	Assume $\beta > \sigma^2 > 0$. 
	Consider a sequence $i = i(N)$ of player(s) such that $\lim_{N \to \infty} i/N = q \in [0,1]$. Then we have
	\begin{equation} \label{singularcostlimit}
		\lim_{N \to \infty} \frac{\mb{E} (h_2 * \mu^{N,i}_\X)(X^i)}{N-1}   
		= \lim_{N \to \infty} \frac{1}{(N-1)^2}  \mb{E} \sum_{k: k \neq i} \frac{1}{(X^i-X^k)^2}
	= \frac{\pi^2\beta}{3(\beta-\sigma^2)} m_{\beta,V}(\gamma^q)^2.
	\end{equation}	
\end{prop}
	\begin{rem} \label{gorinshkolnikov}
		Proposition \ref{conjecture} is an extension of the guess in Remark $3.9$ of Gorin--Shkolnikov \cite{vgmsedge} 
		to the case of a uniformly convex potential $V(x)$ and to any convergent sequence of indices, i.e., both at the edge ($q=0,1$) and in the bulk ($q \in (0,1)$).
		Indeed, our result implies their guess upon taking $V(x)=x^2/2$, $\sigma=1, q=1, \lim_{N \to \infty} X^N = \gamma^1 = \sqrt{2\beta}$. 
%	for convenience in referencing Section 3 of that paper, the relationship of our ensemble $\X$ to their notation $``\mc{X}''$ is 
%	$
%	\sqrt{\frac{N-1}{N}} \sqrt{\frac{2}{\beta}} \X \overset{d}{=} \mc{X}/N.
%	$
	Moreover, this result has other related applications; for example, 
	it immediately implies 	Lemma 3.3 of 
	Gorin--Shkolnikov \cite{vgmsedge}, an innocuous statement
	which nevertheless can take some effort to prove.
	We have not yet found the calculation of the limit \eqref{singularcostlimit} explicitly in the RMT literature,
so Proposition \ref{conjecture} illustrates how
		the $N$ Nash System \eqref{dysonnash}
		can readily lead to a basic application in random matrix theory
		that is interesting in its own right.
	But experts of Calogero-Moser-Sutherland models likely
	know how to compute the mean of the ``$1/r^2$-statistic", perhaps
	in the manner we suggest in Remark \ref{averageindex} below. Proposition \ref{conjecture} is difficult because
		one cannot exploit the algebraic identities \eqref{keyidentity1}, \eqref{keyidentity2} that occur upon averaging. 
	\end{rem}
	\begin{rem} \label{sinekernel}
		For some choices of parameters, one can compute expressions like \eqref{singularcostlimit} directly.
		To sketch this for the archetype choice $V(x)=x^2/2, \beta=2, \sigma = 1$ corresponding to the mean field scaled
		GUE ensemble \eqref{density}, 
		we can use the asymptotic formula for the sine kernel (see, e.g., Section 3.5 of Anderson-Guionnet-Zeitouni \cite{agz1}) to get 
		$$
		\frac{\mb{E}(h_2 * \mu^{N,i}_\X)(X^i)}{N-1}  	 
		\approx  m_2(\gamma^q)^2 \cdot \int_\R \frac{1}{y^2} \left ( 1 - \left ( \frac{\sin(\pi y)}{\pi y} \right )^2  \right ) dy 
		= 2\frac{\pi^2}{3} m_2(\gamma^q)^2.
		$$
		This observation emphasizes that \emph{some microscopic input is relevant to the limit \eqref{singularcostlimit}} (and thus also to our main theorems below), 
		but an interesting aspect of the proof of Proposition \ref{conjecture} is that we do not need to rely on such detailed local limit behavior
		in the bulk, $q \in (0,1)$. 
	\end{rem}

\bt \label{mainresult}
	Fix $\beta>\sigma^2>0$ and let $V: \R \to \R$ be twice continuously differentiable of at most polynomial growth and satisfying
	$V''(x) \geq c_V > 0$ for all $x \in \R$ for some constant $c_V$.
	Recall the definition \eqref{invariantmeasuregeneral} of $\mu_{\beta,V}^N$ and characterization \eqref{eulerlagrange} of $\mu_{\beta,V}$. 
	Then for any (deterministic) sequence $i=i(N)$ of indices with $\lim_{N \to \infty} i/N = q \in [0,1]$, we have the following asymptotic contributions:
	\begin{enumerate}
		\item The self--interaction term contributes
			\begin{equation} \label{controlcontributiongeneral}
				\lim_{N \to \infty} \mu^N_{\beta,V} \left [ (\partial_i v^{N,i}_\beta)^2 \right ] 
				= \frac{\pi^2 \beta^2 \sigma^2}{12(\beta-\sigma^2)} m_{\beta,V}(\gamma^q)^2 + |\partial_xU_\beta(\gamma^q,\mu_{\beta,V})|^2.
			\end{equation}

		\item The drift--interaction term contributes
			$$
			\lim_{N \to \infty} \mu^N_{\beta,V} \left [ \sum_{k:k\neq i} \partial_k v^{N,k}_\beta \partial_k v^{N,i}_\beta  \right ]
			= \frac{\pi^2 \beta^2 \sigma^2}{12(\beta-\sigma^2)} m_{\beta,V}(\gamma^q)^2 
			+ \int_\R \partial_\mu U_\beta(\gamma^q,\mu_{\beta,V})(z) \partial_xU_\beta(z,\mu_{\beta,V}) \mu_{\beta,V}(dz).
			$$
		\item The diffusion term contributes
			$$
			\lim_{N \to \infty} \mu^N_{\beta,V} \left [ - \frac{\sigma^2}{2(N-1)} \Delta_\x v^{N,i}_\beta   \right ] 
			= - \frac{\pi^2 \beta^2 \sigma^2}{6(\beta-\sigma^2)} m_{\beta,V}(\gamma^q)^2.
			$$

	\end{enumerate} 
	Thus, the local contributions cancel and so, integrated against $\mu_{\beta,V}^N$, the sequence of $i$th equations from the $N$--Nash system \eqref{dysonnash} with $C_2 = \beta ( \frac{3}{2} \beta - 2\sigma^2)/4$ in the definition
\eqref{statecost} of $F^{N,i}(\x)$ and with solution $(v^{N,k}_\beta(\x),\lambda^{N,k}_\beta)_{k=1}^N$ 
	converges to the Voiculescu-Wigner master equation \eqref{dysonmaster} with solution $(U_\beta(x,\mu), \frac{\beta}{4})$
	at $(\gamma^q, \mu_{\beta,V}) \in \R \times \mc{P}^2_2(\R)$.
\et
\subsection*{The Coulomb Game}
By a slight abuse in this section, we use the same symbols
for the analogous objects defined on $\RR^2$.
Fix constants $C_1, C_2 > 0$ and for $1 \leq i \leq N$ let 
\begin{equation}
	\label{2dstatecost}
	F^{N,i}(\bz) := \frac{|z^i|^2}{8} + 
	C_1 \int \int_{ \substack{w,u \in \RR^2 \\ w \neq u}} \frac{2}{D^2(z^i-w,z^i-u)} \mu^{N,i}_\bz(dw) \mu^{N,i}_\bz(du) + C_2 \frac{(h_2 * \mu^{N,i}_\bz)(z^i)}{N-1},
	 \ \ \ \bz \in \mc{D}^N,
\end{equation}
where we recall $D(\xi,\eta)$ is the diameter of the circumcircle of the triangle
determined by $\xi,\eta,$ and $(0,0)$ in $\RR^2$. 
Consider the ergodic $N$-Nash system on $\mc{D}^N$: 
\begin{equation} \label{2dNnash}
- \frac{\sigma^2}{2(N-1)} \Delta_\bz v^{N,i}(\bz) + \sum_{k: k \neq i} \langle \nabla_{z^k} v^{N,k}(\bz), \nabla_{z^k} v^{N,i}(\bz) \rangle 
+ \frac{1}{2} | \nabla_{z^i} v^{N,i}(\bz) |^2 = F^{N,i}(\bz) - \lambda^{N,i}, \ \ 1 \leq i \leq N,
\end{equation}
The game theoretic counterpart of these equations is the \emph{closed loop} model of the Coulomb game,  detailed in Section \ref{coulombgame}.
Lemma \ref{2dNnashprop} shows that if we can write $C_1=\beta^2/8$ and $C_2 = 3\beta^2/8$ for some  $\beta \in \RR$, then the ergodic value pairs 
\begin{equation} \label{2dNnashsoln}
\begin{cases}
	v^{N,i}_\beta(\bz) := \frac{|z^i|^2}{4} - \frac{\beta}{2} (h_0 * \mu^{N,i}_\bz)(z^i), & \bz \in \mc{D}^N \\
	\lambda^{N,i}_\beta := \frac{\beta}{4} + \frac{\sigma^2}{2(N-1)}
\end{cases}
\end{equation}
	form a classical solution to the $N$--Nash system \eqref{2dNnash} on $\mc{D}^N$. 
		Lemmas \ref{2dpotentialgame}, \ref{2dehjb} similarly solve the \emph{open loop} model of the Coulomb game
given the relationships $C_1=\beta^2/8$ and $C_2 = \beta^2/4$. 

\begin{figure}[h!] 
		\centering
		\includegraphics[width=.32 \textwidth]{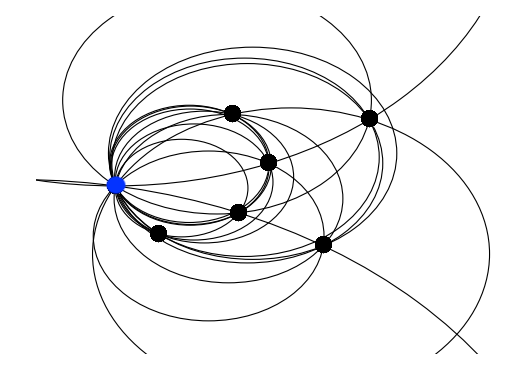}
		\includegraphics[width=.32 \textwidth]{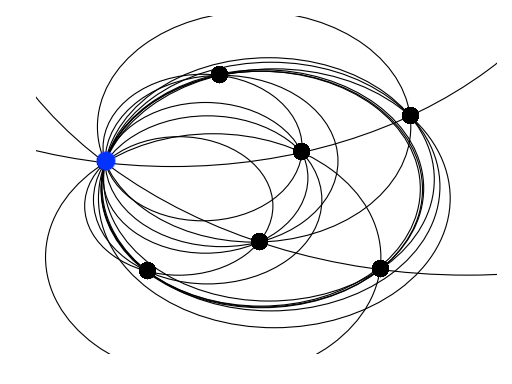}
		\includegraphics[width=.32 \textwidth]{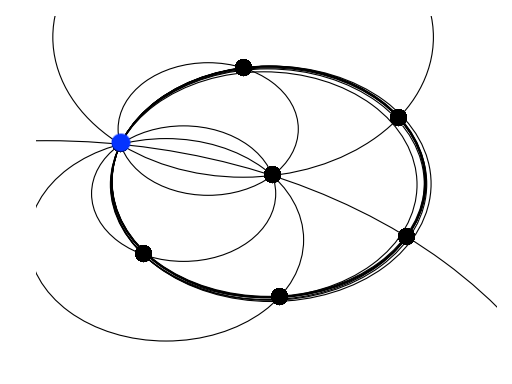}		\caption{A simulation of planar Coulomb dynamics with the circumcircles from the cost function \eqref{2dstatecost} for a given player (in blue).  The first plot shows players in an arbitrary configuration, and the next two plots show the evolution towards equilibrium. Note in particular the formation of the triangular/Abrikisov lattice; see Serfaty \cite{serfaty2015coulomb}.}
		\label{onlyfigure2}
\end{figure}
	
	Just as before,
	we can write $v^{N,i}(\bz) = U(z^i,\mu^{N,i}_\bz)$, where 
	\begin{equation} \label{2dmasterfield}
U_\beta(z, \mu) := \frac{|z|^2}{4} - \frac{\beta}{2} \int_{\R^2} \log |z- w| \mu(dw), \ \  (z,\mu) \in \RR^2 \times \calP_2(\RR^2).
\end{equation}
Lemma \ref{2dmasterequationcalc} shows that the pair $(U_\beta(z,\mu),\frac{\beta}{4})$
solves the \emph{Coulomb master equation} 
\begin{equation} \label{2dmasterequation}
\begin{aligned}
	\int_{\R^2} \langle \partial_\mu U(z, \mu)(w) , (\nabla_z U)(w, \mu) \rangle \mu(dw) & + \frac{1}{2} | (\nabla_z U)(z, \mu) |^2   \\
	& = 
	\frac{|z|^2}{8}  + \frac{\beta^2}{8} \int \int_{ \substack{w,u \in \RR^2 \\ w \neq u}}\frac{2}{D^2(z-w,z-u)} \mu(dw) \mu(du) - \lambda.
\end{aligned}
\end{equation}

Despite the limitations of the existing literature on planar Coulomb dynamics
as compared to the one dimensional case,
we are nevertheless able to obtain results concerning ``convergence
of equations'' 
as generally as for the Dyson game. Fix $\beta, \sigma>0$ and let $V : \R^2 \to \R$ be a twice continuously differentiable and $c_V$-uniformly convex function for some constant $c_V>0$. 
Let $\bZ \in \mc{D}^N_{ordered}$ be distributed according to the generalized 
$\beta/\sigma^2$-ensemble on $\mc{D}^N_{ordered}$:
\begin{equation} \label{2dorderedensemble}
\mu^N_{\beta,V}(d\bz) = m_{\beta,V}(\bz) d \bz : = \frac{N!}{Z^N_{\beta,V}} \cdot
\prod_{1 \leq k < \ell \leq N} |z^\ell - z^k|^{\beta/\sigma^2} \cdot 
\exp \left \{   - \frac{N-1}{\sigma^2} \sum_{i=1}^N V(z^i) \right \} \cdot \mathbbm{1}_{\mc{D}^N_{ordered}}(\bz) d\bz
\end{equation}
where $Z^N_{\beta,V}<\infty$ is a normalization constant. We note that at least formally
$\mu^N_{\beta,V}$ is the invariant distribution of the \emph{$\beta/\sigma^2$-planar Coulomb dynamics}
\begin{equation}
	\label{2dgeneraldynamics}
	dZ^i_t = \left [ \frac{\beta}{2} (h_1 * \mu^{N,i}_{\bZ_t})(Z^i_t) - \frac{\nabla_z V(Z^i_t)}{2} \right ] dt + \frac{\sigma}{\sqrt{N-1}} dB^i_t, \ \ \ 1 \leq i \leq N.
\end{equation}
For existence and uniqueness of \eqref{2dgeneraldynamics},
see Section 1.4.4 of Bolley-Chafa\"{i}-Fontbona \cite{bolley2018dynamics},
Theorem 2.1 of Liu-Yang \cite{liu2016propagation},
and the discussion around equations (1.5) and (1.6) of Lu-Mattingly \cite{lu2019geometric}.
%\begin{remark}
%indicates that 
%there should be a unique strong solution to generalized dynamics of the form \eqref{2dgeneraldynamics} with any initial condition in $\mc{D}^N$, but 
%the proof does not seem to have been carried out explicitly (Theorem 2.1 of Liu-Yang \cite{liu2016propagation} considers the case $V \equiv 0$, though for restricted 
%initial conditions). Very recently, Proposition 2.4 
%of Lu-Mattingly \cite{lu2019geometric} proves
%strong existence and uniqueness to the so-called \emph{inertial} version of \eqref{2dgeneraldynamics} using tailor-made Lyapunov techniques; the dynamics 
%\eqref{2dgeneraldynamics} then should appear as its overdampted limit; compare \eqref{2dgeneraldynamics} to equations (1.5) and (1.6) of \cite{lu2019geometric}. However, 
%our interest in this section is to use the static ensembles
%$\mu^N_V$ as test functions for the $N$ Nash system, so knowing strong existence and uniqueness
%to \eqref{2dgeneraldynamics} only serves to interpret $\mu^N_V$ as a possible
%invariant player ensemble for the game. 
%\end{remark}
\begin{rem}
	We do not expect the dynamics \eqref{2dgeneraldynamics} to respect  
	a given linear ordering $\prec$ of $\RR^2$, e.g., the 
	case $N=2$ rules this out for the \emph{spiral ordering} introduced in Example \ref{applicationofspiral} below.
\end{rem}

Now we know (see, e.g., Section 2.6 of Serfaty's lectures \cite{serfaty2015coulomb}
and Corollary 1.7 of Chafa\"{i}-Hardy-Ma\"{i}da \cite{chafai2018concentration}) 
that there exists a compactly supported measure $\mu_{\beta,V} \in \mathcal{P}(\R^2)$ with density $m_{\beta,V}(z) = \Delta V(z)/(\pi \beta)$ on its support
such that $d_1( \mu^{N,i}_\bZ, \mu_{\beta,V}) \overset{N \to \infty}{\to} 0$ almost surely. Moreover, $\mu_{\beta,V}$ satisfies 
the Euler-Lagrange equation
\begin{equation}
	\label{2deulerlagrange}
\beta \int_{\R^2} \frac{z-w}{|z-w|^2} m_{\beta,V}(w) dw - \nabla_z V(z) = 0, \ \ z \in \text{supp} \mu_{\beta,V}.
\end{equation}
\begin{ex} \label{ginibre}
	The prototypical case occurs when taking $V(z) = |z|^2/2$.
	The density of the equilibrium measure is then a constant
$m_{\beta,V}(z) = 2/\pi \beta$ supported on the ball $B(\sqrt{\beta/2})$. 
If further we set $\sigma=1$ and $\beta=2$, then
	 $\mu_{\beta,V}^N(\bz)$ coincides with the density of the eigenvalues of a $N \times N$ random
	 matrix belonging to the \emph{complex Ginibre ensemble}, 
	which is known to be determinantal and exactly solvable (see, e.g., Meckes-Meckes \cite{meckes2015rate} and references therein).
\end{ex}

\bt \label{big2dtheorem}
Fix $\beta, \sigma>0$.
Let $V: \RR^2 \to \RR$ be twice continuously differentiable, of at most polynomial growth, and $c_V$-uniformly convex for some $c_V>0$. 
Let $\bZ = (Z^1 \prec \cdots \prec Z^N)$ be distributed according to the ensemble $\mu^N_{\beta,V}$ of \eqref{2dorderedensemble}.
Let $i=i(N)$ be a (deterministic) sequence of indices such that $Z^i$ converges to the
macroscopic location $\gamma \in \supp \mu_{\beta,V}$ in $L^p$ for all $p \geq 1$.
Recall $D(\xi,\eta)$ is the diameter of the circumcircle of the triangle
formed by $\xi,\eta,$ and $(0,0)$ in $\RR^2$.
Then we have 
\begin{equation}
	\label{diameterlimit}
	\lim_{N \to \infty} \mathbb{E} \int \int_{ \substack{w,u \in \RR^2 \\ w \neq u}}\frac{\mu^{N,i}_\bZ(dw) \mu^{N,i}_{\bZ}(du)}{D^2(Z^i-w,Z^i-u)}  = 
\int \int_{ \substack{w,u \in \RR^2 \\ w \neq u}}\frac{\mu_{\beta,V}(dw) \mu_{\beta,V}(du)}{D^2(\gamma-w,\gamma-u)}.
\end{equation}
\begin{equation}
	\label{2dreciprocalgaplimit}
	\lim_{N \to \infty} \frac{\mb{E}(h_2 * \mu^{N,i}_\bZ)(Z^i)}{N-1} = 
\lim_{N \to \infty} \frac{1}{(N-1)^2} \mb{E} \sum_{k:k\neq i} \frac{1}{|Z^i - Z^k|^2} = 0. 
\end{equation}
Hence, integrated against $\mu^N_{\beta,V}$, 
the sequence of $i$th equations from the 
$N$-Nash system \eqref{2dNnash} with $C_1 = \beta^2/8$, any $C_2 >0$ in the definition
\eqref{2dstatecost} of $F^{N,i}(\bz)$, and
with solution $(v^{N,k}_\beta(\bz), \lambda^{N,k}_\beta)_{k=1}^N$ converges to the Coulomb master equation \eqref{2dmasterequation} with the solution $(U_\beta(z,\mu), \frac{\beta}{4})$
at the point $(\gamma, \mu_{\beta,V}) \in \RR^2 \times \mc{P}^2_2(\RR^2)$. 
\et

\begin{cor} \label{2depsilonnash} 
Let $C_1 = \beta^2/8$ for some $\beta >0$ and $C_2 \geq 3\beta^2/8$ in the definition
\eqref{2dstatecost} of $F^{N,i}(\bz)$. 
Then the planar Coulomb dynamics \eqref{2dgeneraldynamics}
with $V(z) = |z|^2/2$ form both open and closed loop $\epsilon_N$-Nash equilibria
for the Coulomb game, for some $\epsilon_N \to 0$ as $N \to \infty$.
\end{cor}
\begin{rem}
	Given the lengthscale $1/\sqrt{N}$ of the expected gap size between bulk players in two dimensions,
	one expects $\epsilon_N \lesssim 1/N$; however,
	we do not yet have a rigorous proof of this rate. 
\end{rem}

\begin{ex} \label{applicationofspiral}
Although we are not able to be as explicit as in the one dimensional case
on the specific conditions of the sequence of indices $i=i(N)$ in the general setting of 
Theorem \ref{big2dtheorem}, 
conditions can still be articulated explicitly if we let
$\prec$ denote the \emph{spiral ordering} of $\CC \approx \RR^2$, introduced by Meckes-Meckes \cite{meckes2015rate}, which is defined as follows.
%that will allow us to articulate explicitly
%the predicted location of the players $\bZ \sim \mu_V^N$ in the prototypical case
%$V(z) = |z|^2/2$, $\sigma=1$, $\beta=2$. For more general cases of $V$,
%we do not attempt to be explicit with the conditions on a sequence of players, 
%such as can for the Ginibre ensemble in equation \eqref{2dindexconditions} below.
Let $0 \in \CC$ be the smallest element and for nonzero $w,z \in \CC$, 
write $w \prec z$ if either 
\begin{enumerate}
	\item $\lfloor \sqrt{N} |w| \rfloor < \lfloor \sqrt{N} |z| \rfloor$
	\item $\lfloor \sqrt{N} |w| \rfloor = \lfloor \sqrt{N} |z| \rfloor$ and $\arg w < \arg z$
	\item $\lfloor \sqrt{N} |w| \rfloor = \lfloor \sqrt{N} |z| \rfloor$, $\arg w = \arg z$, and $|w| \geq |z|$. 
\end{enumerate}
Here, we take $\arg z \in (0, 2\pi]$, adopting the same convention as \cite{meckes2015rate}, and also $\lfloor x \rfloor$
denotes the floor function giving the largest integer less than or equal to $x \in \R$.
%For a concrete application of this ordering,
%see Example \ref{applicationofspiral} below that considers predicted locations
%for the prototypical case of the complex Ginibre ensemble.

For the archetype case of the complex Ginibre ensemble given by the choices $V(z) = |z|^2/2$, $\sigma=1$, and $\beta=2$,
Meckes-Meckes \cite{meckes2015rate} articulate 
concentration results using
predicted locations for \emph{most of} the players, i.e., 
for all but $M := N- \lfloor \sqrt{N} \rfloor^2$ players. 
Letting $1 \leq k \leq N-M = \lfloor \sqrt{N} \rfloor^2$, the predicted location for $Z^k$ is given by
\begin{equation}
	\label{predictedlocation}
	\tilde{\lambda}^k := \frac{\lceil \sqrt{k} \rceil - 1}{\sqrt{N}}
	\exp \left \{ 2 \pi \bi \cdot \frac{ k-(\lceil \sqrt{k} \rceil - 1)^2}{(2 \lceil \sqrt{k} \rceil - 1)} \right \}
\end{equation}
where $\bi := \sqrt{-1}$. In words,
the sequence $\{ \tilde{\lambda}^k \}_{k=1}^{N-M}$, which is naturally
ordered according to the spiral ordering $\prec$, starts at $0$, then runs through
$1/\sqrt{N}$ times the third roots of unity,
then $2/\sqrt{N}$ times the fifth roots of unity,
then $3/\sqrt{N}$ times the seventh roots of unity, and so on. 
Notice that this does not determine any predicted locations in the outermost annulus
$\{ z \in \RR^2 :  \sqrt{ 1- \frac{M}{N} } \leq |z| \leq 1 \}$. 

We can now recite the concentration results: if $k$ satisfies $\lceil \sqrt{k} \rceil \leq \sqrt{N} - \sqrt{\log N}$, then we have for the ordered ensemble $\bZ$ the moment estimate
\begin{equation}
	\label{ginibremomentestimate}
\mb{E}| Z^k - \tilde{\lambda}^k |^p \leq \frac{p C^p \Gamma(\frac{p+1}{2})}{N^{p/2}}
\end{equation}
for some universal constant $C>0$ and for all $p \geq 1$ (see the first estimate in the proof of Theorem 1 in Meckes-Meckes \cite{meckes2015rate}).

Hence we finally arrive at the point of this example.
Consider a sequence of indices $i = i(N)$ such that 
$1 \leq i(N) \leq \lfloor \sqrt{N} \rfloor^2$, 
$\lceil \sqrt{i(N)} \rceil \leq \sqrt{N}-\sqrt{ \log N}$, and finally 
\begin{equation}
	\label{2dindexconditions}
\lim_{N \to \infty} i(N)/N = q, \ \ \ 
\lim_{N \to \infty} \left (\lceil \sqrt{i(N)} \rceil - \sqrt{i(N)} \right ) = \theta,
\end{equation}
for some $q, \theta \in [0,1)$ (notice $1$ is omitted).
Write $\gamma^{q,\theta} := q \cdot \exp \{2 \pi \bi  \theta \}$.
Then we have that
$\lim_{N \to \infty} \tilde{\lambda}^{i(N)} = \gamma^{q,\theta}$ and 
so by the moment estimate \eqref{ginibremomentestimate} we have that
$Z^i$ converges to $\gamma^{q,\theta} \in B(1)$. 
\end{ex}

\subsection*{Higher dimensional games with logarithmic interactions}

Although logarithmic interactions are regarded as ``Coulomb" only 
in dimensions $d=1,2$ (see, e.g., Section 1.4 and Chapter 15 of Forrester \cite{pf1}), 
our main construction and calculations continue to hold in dimension $d \geq 3$.
Indeed, we could have pursued a more unified treatment of many results for general $d \geq 1$, at the cost of losing some emphasis on some key distinctions between the $d=1$ and $d \geq 2$ cases.
The corresponding results for $d \geq 3$ are most similar to the $d=2$ case:
if we let (for this short section only) $\bz=(z^1, \ldots, z^N)$ be a vector
of components $z^i \in \RR^d$, $1 \leq i \leq N$, 
we may define $F^{N,i}(\bz)$ and $v^{N,i}_\beta(\bz)$ exactly as in \eqref{2dstatecost} and \eqref{2dNnashsoln}, respectively.
Then we can calculate just as for Lemma \ref{2dNnashprop} that if we can write 
$C_1=\beta^2/8$ and $C_2=
\beta(\frac{3}{2}\beta+(d-2)\sigma^2)/4$ for some  $\beta \in \RR$, then the ergodic value pairs $(v^{N,i}_\beta(\bz), \lambda^{N,i}_\beta)_{i=1}^N$
with $\lambda^{N,i}_\beta := \beta/4+d\frac{\sigma^2}{4(N-1)}$ will solve the associated
$N$ Nash system of the form \eqref{2dNnash}.
The open loop case can similarly be solved explicitly as in Section \ref{coulombgame}
given the choices $C_1=\beta^2/8$ and { $C_2=\beta(\beta+2(d-2)\sigma^2)/4$.}
Notice for either model of player information, 
the formulas for $C_1$ are the same for any $d \geq 2$. 
Finally, one can endeavor to formulate analogs of Theorem \ref{big2dtheorem} and Corollary \ref{2depsilonnash}
for $d \geq 3$, confirming the open and closed loop models still converge together.

\section{$N$ player formulation of the Dyson Game} \label{dysongame}
\subsection*{Closed loop model}
\bd \label{admissibility}
A function $\boldsymbol{\phi} : \mc{W}^N \to \R^N$ is \emph{admissible} if for every $\x_0 \in \mc{W}^N$, 
there exists a unique strong solution $(\X_t )_{t \geq 0} = (\X^{\boldsymbol{\phi}}_t)_{t \geq 0}$ to the stochastic differential equation
\begin{equation} \label{Ndynamics}
			d\X_t = \boldsymbol{\phi}(\X_t) dt + \frac{\sigma}{\sqrt{N-1}} d\mathbf{W}_t, \ \ \X_0 = \x_0, \ \
\sigma \geq 0 
		\end{equation} 
		that remains in $\overline{\mc{W}^N}$ and satisfies the integrability condition
		\begin{equation} \label{admissibleintegrability}
			\mb{E} \int_0^T \left [ \Vert \boldsymbol{\phi}(\X_t)\Vert^2
			+ \Vert \X_t \Vert^2 + \sum_{1 \leq k < \ell \leq N} \frac{1}{(X^\ell_t - X^k_t)^2} 
			\right ] dt < \infty, \ \text{for all} \ \  T > 0.
		\end{equation} 
		We denote the class of such feedbacks by $\mc{A}^{(N)}$.
\ed
The class $\mc{A}^{(N)}$ is fairly rich; 
indeed, Theorem 2.2 of C\'{e}pa--L\'{e}pingle \cite{ecdl1} offers
general solvability of \eqref{Ndynamics} under the state constraint, while the stronger condition \eqref{admissibleintegrability} will need to be checked (see the proof
of Theorem \ref{verificationthm}). Also, the constraint ``$\X_t \in \overline{\mc{W}^N}$ for all $t > 0$" is 
consistent with the framework suggested by the early work \cite{jlpl1} of Lasry--Lions for state--constrained problems and is necessary if $C_2 \leq 0$,
but this condition will be \emph{forced} by the form of singular cost if $C_2>0$.
% (see Remark \ref{optimalparameterremark}).

The \emph{closed loop model for the $N$ player Dyson game} can be formulated as follows.
Fix any feedback profile $\boldsymbol{\phi} \in \mc{A}^{(N)}$ and $C_2 \in \R$.
Interpreting the components $(X^i_t)_{t \geq 0}$ of \eqref{Ndynamics} as \emph{players}, we accordingly define for every $1 \leq i \leq N$
the $i$th player's admissible class $\mc{A}^i(\boldsymbol{\phi}^{-i})$ to be the collection of $\psi:\R^N \to \R$ 
such that $(\psi,\boldsymbol{\phi}^{-i}):= (\phi^1, \ldots, \phi^{i-1},\psi,\phi^{i+1}, \ldots, \phi^N) \in \mc{A}^{(N)}$. 
Recall the definition \eqref{statecost} of the state--cost $F^{N,i}(\bx)$.
Then the search for Nash equilibria in the closed loop model requires each player $i$, $1\leq i \leq N$, to minimize the ergodic cost
$$
J^{N,i}(\psi |\x_0, \boldsymbol{\phi}^{-i}) := \limsup_{T \to \infty} \frac{1}{T} \mb{E} \int_0^T
\left [ \frac{1}{2} \psi(\X_t)^2+ F^{N,i}(\X_t) \right] dt
$$ 
over deviations $\psi \in \mc{A}^i(\boldsymbol{\phi}^{-i})$, 
subject to $(\X_t)_{t \geq 0} = (\X^{(\psi, \boldsymbol{\phi}^{-i})}_t)_{t \geq 0}$ satisfying \eqref{Ndynamics} with $\X_0 = \x_0 \in \mc{W}^N$.  
\bd \label{Nnashequilibrium}
A feedback profile $\boldsymbol{\phi}^* = (\phi^{*1}, \ldots, \phi^{*N}) \in \mc{A}^{(N)}$
is a \emph{closed loop Nash equilibrium over classes $\mc{A}^i \subset \mc{A}^i(\boldsymbol{\phi}^{*-i})$, $1 \leq i \leq N$}, 
if for all $1 \leq i \leq N$, $\phi^{*i} \in \mc{A}^i$ and for every $\x_0 \in \mc{W}^N$, we have
$$\inf_{\psi \in \mc{A}^i} J^{N,i}(\psi | \x_0, \boldsymbol{\phi}^{*-i})=J^{N,i}(\phi^{*i} | \x_0 , \boldsymbol{\phi}^{*-i}).$$
\ed
 
	We now solve explicitly the $N$--Nash system \eqref{dysonnash} for the closed loop Dyson game, 
	the system of $N$ ergodic Hamilton--Jacobi--Bellman (HJB) equations associated to the search for a Nash equilibrium (compare with equation (1) of \cite{cdll1} and see Section 2.5.3 of \cite{bible1} for some intuition).
	\begin{lem}
		 \label{Nsubsolution}
	Assume $\sigma \geq 0$ and that the coefficient $C_2$ of \eqref{statecost} satisfies $C_2 \geq -\sigma^4/6$,
	so we may write $C_2 = \beta (\frac{3}{2}\beta - 2\sigma^2)/4$ for some $\beta \in \R$. 
	Then the pairs $(v_\beta^{N,i}(\x), \lambda^{N,i}_\beta)_{i=1}^N$ defined in
	\eqref{ehjbsolution} 
	form a classical solution to the $N$--Nash system \eqref{dysonnash} on $\mc{W}^N$.
\end{lem}
	\bpf
The proof follows by direct calculation. First we collect some facts:
\begin{equation} \label{partials}
\partial_i v^{N,i}_\beta(\x) = \frac{x^i}{2} - \frac{\beta}{2} (h_1 * \mu^{N,i}_\x)(x^i) ,
\ \ \partial^2_iv^{N,i}_\beta(\x) = \frac{1}{2} + \frac{\beta}{2} (h_2 * \mu^{N,i}_\x)(x^i) 
\end{equation} 
and similarly if $k \neq i$ 
$$
\partial_k v^{N,i}_\beta(\x) = \frac{\beta}{2(N-1)} \frac{1}{x^i-x^k}, \ \ \partial^2_kv^{N,i}_\beta(\x) = \frac{\beta}{2(N-1)}\frac{1}{(x^i-x^k)^2}. 
$$
Then we can compute
$$
-\frac{\sigma^2}{2(N-1)} \Delta_\x v^{N,i}_\beta(\x) = -\frac{\sigma^2}{2(N-1)} \sum_{k=1}^N \partial_k^2 v^{N,i}_\beta(\x)
=  -\frac{\sigma^2}{4(N-1)} -\frac{\sigma^2\beta}{2} \frac{(h_2 * \mu^{N,i}_\x)(x^i)}{N-1}  
$$
and
\begin{equation} \label{controlcalc}
\begin{aligned}
	& \frac{1}{2}(\partial_iv^{N,i}_\beta)^2(\x) = \frac{1}{2} \left ( \frac{x^i}{2} - \frac{\beta}{2}(h_1 * \mu^{N,i}_\x)(x^i)\right)^2 \\
	& = \frac{(x^i)^2}{8} + \frac{\beta^2}{8} \frac{(h_2 * \mu^{N,i}_\x)(x^i)}{N-1}
	 - \frac{\beta}{4} x^i \cdot (h_1*\mu^{N,i}_\x)(x^i) 
	+  \frac{\beta^2}{8(N-1)^2} \sum_{k:k\neq i} \sum_{\ell: \ell \neq i,k} \frac{1}{x^i-x^k} \frac{1}{x^i-x^\ell}.
\end{aligned}
\end{equation}
Similarly, we have
\begin{equation} \label{driftinteraction}
\ba
& \sum_{k: k\neq i} \partial_k v^{N,k}_\beta(\x)\partial_k v^{N,i}_\beta(\x)  =
\frac{\beta}{2(N-1)} \sum_{k: k \neq i} \left (  \frac{x^k}{2}-\frac{\beta}{2}(h_1 * \mu^{N,k}_\x)(x^k) \right) \frac{1}{x^i - x^k} \\
& = \frac{\beta^2}{4} \frac{(h_2 * \mu^{N,i}_\x)(x^i)}{N-1} 
+  \frac{\beta}{4(N-1)} \sum_{k: k \neq i} \frac{x^k}{x^i-x^k} 
-  \frac{\beta^2}{4(N-1)^2} \sum_{k: k\neq i} \sum_{\ell: \ell \neq k,i} \frac{1}{x^k-x^\ell}\frac{1}{x^i-x^k}.
\ea
\end{equation}
The two final terms of \eqref{controlcalc}, \eqref{driftinteraction} cancel by the algebra
\begin{equation} \label{basicalgebra}
	2\sum_{k: k\neq i} \sum_{\ell: \ell \neq k,i} \frac{1}{x^k-x^\ell}\frac{1}{x^i-x^k} = 
	\sum_{k: k\neq i} \sum_{\ell: \ell \neq k,i} \frac{1}{x^k-x^\ell} \left [ \frac{1}{x^i-x^k} - \frac{1}{x^i-x^\ell} \right ]
	=\sum_{k:k\neq i} \sum_{\ell: \ell \neq i,k} \frac{1}{x^i-x^k} \frac{1}{x^i-x^\ell},
\end{equation}
while the two second--to--last terms of \eqref{controlcalc}, \eqref{driftinteraction} combine to yield the constant $-\frac{\beta}{4}$. 
Putting everything together completes the proof.
\epf

\subsection*{Open loop model} \label{openloop}
To formulate the \emph{open loop model for the $N$ player Dyson game}, we proceed as above.
\bd \label{openadmissibility}
A profile $(\boldsymbol{\alpha}_t)_{t \geq 0} =((\alpha^1_t, \ldots, \alpha^N_t))_{t \geq 0}$ of $\R$--valued processes is \emph{admissible} if
it is $\mb{F}$--progressively measurable and for every $\x_0 \in \mc{W}^N$, 
the process $(\X_t)_{t \geq 0} = (\X^{\boldsymbol{\alpha}}_t)_{t \geq 0}$ defined by 
		\begin{equation} \label{openNdynamics}
			d\X_t = \boldsymbol{\alpha}_t dt + \frac{\sigma}{\sqrt{N-1}} d\mathbf{W}_t, \ \ \X_0 = \x_0, \ \ \sigma \geq 0 
		\end{equation} 
		remains in $\overline{\mc{W}^N}$ and satisfies the integrability condition
		\begin{equation} \label{openadmissibleintegrability}
			\mb{E} \int_0^T \left [ \Vert \boldsymbol{\alpha}_t \Vert^2
				+ \Vert \X_t \Vert^2 + \sum_{1 \leq k < \ell \leq N} \frac{1}{(X^\ell_t - X^k_t)^2} 
			\right ]
			dt < \infty, \ \text{for all} \ \  T > 0.
		\end{equation} 
We denote the class of such admissible strategies by $\mb{A}^{(N)}$.  
\ed
Fix a strategy profile $(\boldsymbol{\alpha}_t)_{t \geq 0} \in \mb{A}^{(N)}$ and $C_2 \in \R$. Define for every $1 \leq i \leq N$
the $i$th player's admissible class $\mb{A}^i(\boldsymbol{\alpha}^{-i})$ to be the collection of $\R$--valued processes $(\eta_t)_{t \geq 0}$ such that
$$((\eta_t,\boldsymbol{\alpha}^{-i}_t))_{t \geq 0} := ((\alpha^1_t, \ldots, \alpha^{i-1}_t,\eta_t,\alpha^{i+1}_t, \ldots, \alpha^N_t) )_{t \geq 0} \in \mb{A}^{(N)},$$  
henceforth abbreviated ``$(\eta,\boldsymbol{\alpha}^{-i})$". Then the search for Nash equilibria in the open loop model requires each player $i$, $1\leq i \leq N$,
to minimize the ergodic cost (recall the definition $F^{N,i}$ of the state cost \eqref{statecost}) 
\begin{equation} \label{openloopcostfunctional}
	J^{N,i}(\eta \  |\x_0,\boldsymbol{\alpha}^{-i}) := \limsup_{T \to \infty} \frac{1}{T} \mb{E} \int_0^T
\left [ \frac{1}{2} \eta_t^2+ F^{N,i}(\X_t) \right] dt.
\end{equation}
over deviations $\eta \in \mb{A}^i(\boldsymbol{\alpha}^{-i})$, subject to $(\X_t)_{t \geq 0}  = (\X^{(\eta,\boldsymbol{\alpha}^{-i})}_t)_{t \geq 0}$
satisfying \eqref{openNdynamics} with $\X_0 = \x_0 \in \mc{W}^N$. 
Note by a slight abuse,
we maintain the same notation despite now working with control processes instead of feedbacks.
We omit an explicit definition of open loop Nash equilibrium since it is already indicated by Definition \ref{Nnashequilibrium}.

\bl \label{potentialgame}
Fix $C \in \RR$ and define a global cost function by 
\begin{equation} \label{openstatecost}
	F^N(\x) := \frac{\Vert \x \Vert^2}{8} + C  \sum_{i=1}^N \frac{(h_2 * \mu_\x^{N,i})(x^i)}{N-1}, \ \ \x \in \mc{W}^N, 
\end{equation} 
and corresponding global cost functional
\begin{equation} \label{opencostfunctional}
	J^N(\boldsymbol{\alpha}|\x_0) := \limsup_{T \to \infty} \frac{1}{T} \mb{E} \int_0^T \left [ \frac{1}{2} \Vert \boldsymbol{\alpha}_t \Vert^2 + F^{N}(\X_t)\right] dt
\end{equation}
over $(\boldsymbol{\alpha}_t)_{t\geq 0} \in \mb{A}^{(N)}$,
subject to $(\X_t)_{t \geq 0}  = (\X_t^{\boldsymbol{\alpha}} )_{t \geq 0}$ satisfying \eqref{openNdynamics} with $\X_0=\x_0 \in \mc{W}^N$.
Suppose that in definitions \eqref{statecost}, \eqref{openstatecost} of $F^{N,i}(\bz)$, $F^N(\bz)$
we take $C=C_2/2$.
Then the open loop model for the Dyson game is a \emph{potential game} in the following sense:
For any profile $(\boldsymbol{\alpha}_t)_{t \geq 0} \in \mb{A}^{(N)}$ such that the limit in \eqref{opencostfunctional} exists, and 
for any deviation $(\eta_t)_{t \geq 0} \in \mb{A}^i(\boldsymbol{\alpha}^{-i})$, $1 \leq i \leq N$, such that 
the limit in \eqref{openloopcostfunctional} exists, we have
\begin{equation} \label{potentialgamecondition}
	J^N((\eta,\boldsymbol{\alpha}^{-i}) | \x_0 ) - J^N(\boldsymbol{\alpha}| \x_0 ) = 
	J^{N,i}(\eta \  |\x_0,\boldsymbol{\alpha}^{-i}) - J^{N,i}( \alpha^i \  |\x_0,\boldsymbol{\alpha}^{-i}).
\end{equation}
\el 
\bpf  
First we calculate
$$
	\sum_{\ell =1}^N \sum_{k:k \neq \ell} \frac{1}{(x^\ell - x^k)^2} 
	= 2 \sum_{k:k \neq i} \frac{1}{(x^i-x^k)^2} + \sum_{\ell : \ell \neq i } \sum_{k:k \neq \ell,i} \frac{1}{(x^\ell - x^k)^2}.
	$$
Then under the assumptions of the statement,
a straightforward check confirms that condition \eqref{potentialgamecondition} holds, which exactly meets Definition 2.23 in \cite{bible1} for potential game.
\epf

As indicated by the characterizing condition \eqref{potentialgamecondition}, the potential game structure of
Lemma \ref{potentialgame} allows us to reduce the search for an open loop Nash equilibrium 
to a single auxiliary global problem. Given this setting of classical optimal control, the minimizer of \eqref{opencostfunctional} can be achieved
by strategies in closed loop feedback form and any such candidate is characterized by a solution to the ergodic HJB equation
\begin{equation} \label{opendysonnash}
	-\frac{\sigma^2}{2(N-1)} \cdot \frac{1}{N} \Delta_\x W(\x) + \frac{1}{2N} \Vert \nabla_\x W(\x) \Vert^2  
	= \frac{1}{N} F^{N}(\x) - \lambda, \ \ \x \in \mc{W}^N, 
\end{equation}
(the factor of ``$\frac{1}{N}$" will give the correct scale for the comparison at the end of Section \ref{verificationtheorem} and anticipates taking limits).
\begin{rem}
	If $C_2>0$, then the existence and uniqueness for bounded below, 
	unbounded above solutions to \eqref{opendysonnash} follows from
	Theorem 3.1 of Barles--Meireles \cite{gbjm1}.
\end{rem}
Compare the next statement with Lemma \ref{Nsubsolution}.
\begin{lem} \label{openloopehjbtheorem}
Assume $\sigma^2 \geq 0$ and that the coefficient $C_2$ satisfies $ C_2 \geq -\sigma^4/4$, so we may write 
$C_2 = \beta (\beta - 2 \sigma^2)/4$ for some $\beta \in \R$. 
Then the ergodic value pair 
	\begin{equation} \label{potentialfunction}
	\begin{cases}
		W_\beta(\x): = \frac{\Vert \bx \Vert^2}{4}  - \frac{\beta}{2(N-1)} \sum_{1 \leq k < \ell \leq N} \log (x^\ell-x^k),
		& \x \in \mc{W}^N \\
		\lambda_\beta^N := \frac{\beta}{8} + \frac{\sigma^2}{4(N-1)} & \text{\emph{otherwise}}
	\end{cases}
	\end{equation}
forms a classical solution to the HJB equation \eqref{opendysonnash} on $\mc{W}^N$.
\end{lem}
\bpf
The proof is exactly given by Propositions 11.3.1 of Forrester \cite{pf1}, which we repeat for the convenience of the reader. 
Recall the algebraic identities (see pg.252 of \cite{agz1}) 
	\begin{equation} \label{keyidentity1}
		\sum_{i=1}^N x^i\cdot (h_1 * \mu^{N,i}_\x)(x^i) = \frac{N}{2}. 
	\end{equation}
	and
	\begin{equation} \label{keyidentity2}
		 \sum_{i=1}^N \frac{(h_2 * \mu_\x^{N,i})(x^i)}{N-1} = \sum_{i=1}^N [(h_1*\mu^{N,i}_\x)(x^i)]^2. 
	\end{equation} 
Since $\partial_i W_\beta(\x) = \partial_i v^{N,i}_\beta(\x)$, we have by \eqref{partials}, \eqref{keyidentity1}, and \eqref{keyidentity2}  
$$
- \frac{\sigma^2}{2N(N-1)} \Delta_\x W_\beta(\x)  = - \frac{\sigma^2}{4(N-1)} 
- \frac{\beta \sigma^2}{4} \frac{1}{N} \sum_{i=1}^N \frac{(h_2 * \mu_\x^{N,i})(x^i)}{N-1}
$$
and
$$
\frac{1}{2N} \Vert \nabla_x W_\beta(\x)  \Vert^2 = \frac{\beta^2}{8} \cdot \frac{1}{N} \sum_{i=1}^N \frac{(h_2 * \mu_\x^{N,i})(x^i)}{N-1}
+ \frac{1}{N} \sum_{i=1}^N \frac{(x^i)^2}{8} - \frac{\beta}{8}. 
$$
\epf
%\begin{remark} \label{rootremark}
%	The same discussion of Remark \ref{optimalparameterremark} applies here. 
%	But note in the range $0 < \beta < 2 \sigma^2$ when $C<0$, the non--uniqueness for bounded below, 
%	unbounded above solutions to the ergodic HJB \eqref{opendysonnash} does not contradict
%	Theorem 3.1 of Barles--Meireles \cite{gbjm1} because in this range the ``righthandside" $F^N(\x)$ of the ergodic HJB \eqref{opendysonnash} 
%	is no longer bounded below; however, $F^N(\x)$ trivially becomes bounded below
%	in the critical case $C=0$, i.e., $\beta=0, 2 \sigma^2$! This apparent contradiction of non--uniqueness with Theorem 3.1 of \cite{gbjm1} evaporates when we recall 
%	the state constraint and integrability condition \eqref{openadmissibleintegrability} imposed in our notion of admissibility Definition \ref{openadmissibility}.
%\end{remark}

\section{Verification theorems} \label{verificationtheorem}
\subsection*{Closed loop model}
Assume that the coefficient $C_2$ from \eqref{statecost} 
can be written $C_2 = \beta (\frac{3}{2}\beta - 2\sigma^2)/4 > - \sigma^4/8$ for some $\beta > \sigma^2 > 0$.
Then, by Lemma \ref{Nsubsolution}, the set of solution pairs \eqref{ehjbsolution} to the $N$--Nash system \eqref{dysonnash} 
furnishes a feedback profile $\boldsymbol{\phi}_\beta^*(\x) : = (-\partial_iv^{N,i}_\beta(\x))_{i=1}^N$ with trajectories $(\X^*_t )_{t \geq 0} $
given by \emph{$\beta/\sigma^2$-Dyson Brownian motion} \eqref{dbm}.
Since $\beta > \sigma^2$, $(\X^*_t)_{t \geq 0}$ remains in the interior $\mc{W}^N$ for all $t>0$, even if $\X_0^* = \x_0 \in \overline{\mc{W}^N}$ (see 
\cite{ecdl2,lrzs1}), but we still need to check that $\boldsymbol{\phi}_\beta^*$ satisfies the integrability condition \eqref{admissibleintegrability}.
%\begin{remark}
%	The appearance of the ratio $\beta/\sigma^2<\infty$ (the ``inverse temperature")
%	in the name for the dynamics \eqref{dbm} is justified by the form of invariant distribution \eqref{invariantmeasure} below.
%	Our notion of admissibility, Definition \ref{admissibility}, prevents the case of infinite temperature $\beta/\sigma^2 = 0$,
%	or ``complete independence", from occurring for the Dyson game, even in a limiting sense, 
%	because the critical case ``$\beta = \sigma^2$" can only be interpreted through the limit $\beta \downarrow \sigma^2$, as we will see.	
%	However, essentially all the results of this paper continue to hold verbatim in the deterministic case $\sigma^2=0$, 
%	when there is no temperature $\beta/\sigma^2=\infty$ and thus the equilibrium state is ``frozen"; see Remark \eqref{frozen}.	
%\end{remark}

For $(\x,\alpha^i) \in \mc{W}^N \times \R$, write the cost as
	\begin{equation} \label{fullcost}
		f_\beta^{N,i}(\x,\alpha^i) := 
	 \frac{(\alpha^i)^2}{2} + \frac{(x^i)^2}{8} + \frac{\beta}{4} \left (\frac{3}{2}\beta - 2\sigma^2 \right ) \cdot \frac{(h_2 * \mu^{N,i}_\x)(x^i)}{N-1}
	\end{equation}
	Note the (control) Hamiltonian of player $i$, $1 \leq i \leq N$, is given for
	$(\x,\boldsymbol{y}^i,\boldsymbol{\alpha}) \in \mc{W}^N \times \R^N  \times \R^N$ by
	\begin{equation} \label{hamiltonian}
	H_\beta^i(\x,\boldsymbol{y}^i,\boldsymbol{\alpha}) := \sum_{k=1}^N y^{ik} \alpha^k + f_\beta^{N,i}(\x,\alpha^i).
	\end{equation}
	The following \emph{Verification Theorem} is proven in detail because, somewhat surprisingly,
	its content supports our proof of Proposition \ref{conjecture} in Section \ref{justifypassage}.
	\bt \label{verificationthm}
	Fix $\beta > \sigma^2 > 0$ in \eqref{fullcost} and recall from \eqref{ehjbsolution} 
	the solution pairs $(v^{N,i}_\beta(\x), \lambda^{N,i}_\beta)$, $1 \leq i \leq N$, to the $N$--Nash system \eqref{dysonnash}. 
	Fix an interior initial condition $\x_0 \in \mc{W}^N$.
	Then $\boldsymbol{\phi}^*_\beta(\x) := (-\partial_i v^{N,i}_\beta(\x))_{i=1}^N$
	is a closed loop Nash equilibrium over the classes $\mc{A}^i \subset \mc{A}^{i}(\boldsymbol{\phi}_\beta^{*-i})$ of deviations
	$\psi^i(\x)$, $1 \leq i \leq N$, such that $(\X_t )_{t \geq 0}  = (\X^{(\psi^i,\boldsymbol{\phi}_\beta^{*-i})}_t )_{t \geq 0} $ 
	of \eqref{Ndynamics} satisfies the stability conditions 
	\begin{equation} \label{controlconditions}
		\ba
		\limsup_{T \to \infty} \frac{1}{T} \mb{E} \int_0^T f_\beta^{N,i}(\X_t, \psi^i(\X_t)) dt < \infty, \ \ \  
		& \limsup_{T \to \infty} \frac{1}{T} \mb{E}[v^{N,i}_\beta(\X_T) - v^{N,i}_\beta(\x_0)] = 0.
		\ea
	\end{equation}
	Further, the cost to each player $i$ under the equilibrium dynamics $(\X^*_t)_{t \geq 0} = (\X^{\boldsymbol{\phi}_\beta^*}_t)_{t \geq 0}$ satisfies
	$$
	\inf_{\psi \in \mc{A}^i} J^{N,i}(\psi | \x_0, \boldsymbol{\phi}^{*-i}_\beta) =  J^{N,i}(\phi^{*i}_\beta | \x_0, \boldsymbol{\phi}^{*-i}_\beta) 
	= \lambda^{N,i}_\beta =  \frac{\sigma^2}{4(N-1)}+ \frac{\beta}{4}.
	$$ 
	\et
	\bpf
	We check the candidate Nash equilibrium $\boldsymbol{\phi}^*_\beta$ satisfies the integrability condition \eqref{admissibleintegrability}
	and that the corresponding dynamics $(\X^*_t)_{t \geq 0}$ satisfy the stability conditions in \eqref{controlconditions}. First, we review some facts.
	Observe $W_\beta(\x)$ of \eqref{potentialfunction} is uniformly convex: for any vector $v \in \R^N$ and $\x \in \mc{W}^N$
	$$
	v^\dagger \nabla \nabla W_\beta(\x)v = \frac{1}{2}\sum_{k=1}^N v_i^2 +   
	\frac{\beta}{2(N-1)} \sum_{1 \leq k < \ell \leq N} \frac{(v_\ell - v_k)^2}{(x_\ell - x_k)^2} 
	\geq \frac{1}{2}\sum_{k=1}^N v_i^2.	
	$$
	Notice we may write the global dynamics of the system \eqref{dbm} as the gradient flow
	$$
	d\X_t^* = -\nabla_\bx W_\beta(\X_t^*)dt + \frac{\sigma}{\sqrt{N-1}} d\mathbf{W}_t, \ \ \X^*_0 = \x_0,
	$$
	with (unique) globally invariant log--concave probability measure given by the $\beta/\sigma^2$--ensemble:
		\begin{equation} \label{invariantmeasure}
		\ba
		\mu_\beta^N(d\x) = m^N_\beta(\x) d\x  & := \frac{N!}{Z_{\beta}^N} \cdot  \exp \left \{ -\frac{2(N-1)}{\sigma^2}W_\beta(\x) \right \}
		\mathbbm{1}_{\mc{W}^N}(\x) d\x,  \\
		\ea
	\end{equation}
	where $Z_{\beta}^N < \infty$ is the normalization constant (compare these expressions with (13), (14), and (16) of Dyson \cite{theOG}). 
	Write $P_tf(\x):=\mb{E}_\x f(\X^*_t)$ for the semigroup and the generator as
	$$ 
	\mc{L} : = \frac{\sigma^2}{2(N-1)} \Delta_\x  - \nabla_\x W_\beta(\x) \cdot \nabla_\x.
	$$
	Recall that invariance means $\mu^N_\beta(\mc{L}f) = 0$ for suitable $f$. More precisely,
	we say $f$ is in the domain of $\mc{L}$ to mean there exists a function $g$ such that for every
	$\x \in \mc{W}^N$, $\int_0^T |g(\X^*_t)| dt <\infty$ and
	the process $M_T:= f(\X_T^*) - f(\x) - \int_0^T g(\X^*_t) dt$ is an $\mb{F}$--martingale; one then writes $\mc{L}f = g$. 
	The dynamics are reversible with respect to $\mu^N_\beta$ and for any $f \in L^1(\mu^N_\beta)$, we have (by ``asymptotic
	flatness,'' following from the monotonicity of the drift; see Section 7.3 of Arapostathis-Borkar-Ghosh \cite{arapostathis2012ergodic})
	\begin{equation} \label{ergodicity}
		\frac{1}{T} \int_0^T P_tf(\x_0) dt \overset{T \to \infty}{\rightarrow} \mu^N_\beta[f]= \int f(\x) \mu^N_\beta(d\x).
	\end{equation}
	
	Now we check the integrability condition \eqref{admissibleintegrability} using ideas from the proof of Lemma 4.3.3 of Anderson--Guionnet--Zeitouni \cite{agz1}.
	Since $x^2 - 2\beta \log(1+|x|)$ is uniformly bounded below and $\log |x-y| \leq \log (1+|x|) + \log (1+|y|)$ for $x,y \in \R$, we can estimate
	for any $1 \leq k \neq \ell \leq N$
	\begin{equation} \label{lyapunovinequality}
	\ba
		W_\beta(\x) + \frac{\beta}{4(N-1)} \log |x^k - x^\ell| \geq \frac{1}{4} \sum_{i=1}^N \left [ (x^i)^2 - 2\beta \log(1+|x^i|)  \right ] \geq - M',
	\ea
	\end{equation}
	for some constant $M' = M'(N) \in \R$ independent of $k,\ell$; indeed, we also have $W_\beta(\x) \geq -M'$.
	Defining $T_M := \inf\{ t \geq 0 \ : \ W_\beta(\X_t^*) \geq M  \}$, the estimate \eqref{lyapunovinequality} implies that on the event $T_M>T$,
	each gap can be controlled: $|X^k_t-X^\ell_t|\geq \exp\left [-\frac{4(N-1)}{\beta}(M+M') \right ]$ for all $t \leq T$ and $k \neq \ell$.
	Hence, we can use Ito's formula along with Lemma \ref{openloopehjbtheorem} to compute up until the time $T_M$ (recall the definition \eqref{openstatecost} of $F^N(\bx)$)
	\begin{equation} \label{lyapunovsde}
	\begin{aligned}
	 & dW_\beta(\X_t^*) =  \left ( - F^N(\bX^*_t) + N \lambda^N_\beta
	  -  \frac{1}{2} \Vert \nabla_\bx W_\beta(\bX^*_t) \Vert^2 \right )  dt + \frac{\sigma}{\sqrt{N-1}}\nabla_\bx W_\beta(\X^*_t) \cdot d\mathbf{W}_t.   \\
	& = \left ( - \frac{ \Vert \bX^{*}_t \Vert^2}{4} + \frac{N}{4} \left (\beta + \frac{\sigma^2}{N-1} \right ) - 
		\frac{\beta}{4} (\beta- \sigma^2)\sum_{i=1}^N \frac{(h_2 * \mu_{\X^*_t}^{N,i})(X^{*i}_t)}{N-1}  \right ) dt
		+ \frac{\sigma}{\sqrt{N-1}}\nabla_\bx W_\beta(\X^*_t) \cdot d\mathbf{W}_t.
		\end{aligned}
	\end{equation}
	where the local martingale in \eqref{lyapunovsde} stopped at time $T_M$ is a true martingale.
	Putting everything together and recalling $\beta > \sigma^2$, we have
	\begin{equation} \label{finiteh2}
	\mb{E} \sum_{i=1}^N \int_0^{T \wedge T_M} (h_2 * \mu^{N,i}_{\X_t^*})(X_t^{*i}) dt 
	\leq \frac{4(N-1)}{\beta (\beta-\sigma^2)} \cdot \left [\frac{N}{4} \left (\beta + \frac{\sigma^2}{N-1} \right ) T + W_\beta(\bx_0) + M' \right ] < \infty.
	\end{equation}
	The proof of Lemma 4.3.3 of \cite{agz1} shows $\lim_{M\to \infty} T_{M} = +\infty$ almost surely, 
	so an application of Fatou's lemma implies the 
	expectation of the time integral of the reciprocal gaps squared is finite for every $T>0$. 
	In addition, since $\mu^N_\beta$ is the global invariant measure, our work also implies 
	that if $\X^* \sim \mu^N_\beta$, we have $\mb{E} \sum_{i=1}^N (h_2 * \mu_{\X^*}^{N,i})(X^{*i}) < \infty$.
	(Notice this argument fails for $\beta = \sigma^2$ and 
	in fact these expectations can be shown to diverge by comparison with a $2$--dimensional Bessel process.)
 
	Hence, it is now easy to see that the conditions of \eqref{admissibleintegrability}, \eqref{controlconditions} hold for the 
	candidate Nash equilibrium $\boldsymbol{\phi}^*_\beta$ when $\beta>\sigma^2>0$. For example, we now know that $v^{N,i}_\beta$ is in
	the domain of $\mc{L}$, so by the fundamental theorem of calculus, the fact that $\mc{L}v^{N,i}_\beta \in L^1(\mu^N_\beta)$,
	ergodicity \eqref{ergodicity}, and then invariance of $\mu^N_\beta$, 
	the second condition of \eqref{controlconditions} follows from
	\begin{equation} \label{ergodiclimit}
	\lim_{T \to \infty} \frac{1}{T} \mb{E}[v^{N,i}_\beta(\X^*_T) - v^{N,i}_\beta(\x_0)]
	= \lim_{T \to \infty} \frac{1}{T} \int_0^T P_t (\mc{L}v^{N,i}_\beta)(\x_0) dt 
	= \mu^N_\beta[\mc{L}v^{N,i}_\beta] = 0. 
	\end{equation} 

	Now fix $1 \leq i \leq N$ and $\psi \in \mc{A}^i$. Let $(\X_t )_{t \geq 0}  = (\X^{(\psi^i,\boldsymbol{\phi}_\beta^{*-i})}_t )_{t \geq 0} $ 
	satisfy \eqref{Ndynamics}. Using the fact that $(v^{N,i}_\beta(\x),\lambda_\beta^{N,i})$, $1 \leq i \leq N$, 
	solves the $N$--Nash system \eqref{dysonnash} and that Ito's formula holds
	for all time as long as no collisions occur,  we can compute (as in Proposition 2.11 of Carmona--Delarue \cite{bible1})
	\begin{equation} \label{Nverificationcalc}
		\ba
		& dv^{N,i}_\beta(\X_t)+[f_\beta^{N,i}(\X_t,\psi(\X_t)) - \lambda^{N,i}_\beta ]dt 
		=  \frac{\sigma}{\sqrt{N-1}} \sum_{k=1}^N \partial_kv^{N,i}_\beta(\X_t) dW^k_t\\
		& +[H_\beta^i(\X_t,\nabla_\x v^{N,i}_\beta(\X_t), (\psi(\X_t),\boldsymbol{\phi}^*_\beta(\X_t)^{-i})) 
		- H_\beta^i(\X_t,\nabla_\x v^{N,i}_\beta(\X_t), \boldsymbol{\phi}^*_\beta(\X_t)) ]dt 
		\ea
	\end{equation}
	First, the local martingale on the right of \eqref{Nverificationcalc} is a true martingale by the integrability condition \eqref{admissibleintegrability}.
	Second, the Hamiltonian \eqref{hamiltonian} satisfies a strict \emph{Isaacs' condition} (see Definition 2.9 of \cite{bible1}),
	so the difference of Hamiltonian values is nonnegative.  
	Third, we can use the two conditions of \eqref{controlconditions} 
	to take expectations and limits to conclude
	$$
	J^{N,i}(\psi | \x_0, \boldsymbol{\phi}^{*-i}) = \limsup_{T \to \infty} \frac{1}{T} \mb{E}
	\int_0^T f_\beta^{N,i}(\X_t, \psi(\X_t)) dt 
	 \geq  \lambda^{N,i}_\beta.
	$$
	with equality if and only if $\psi \equiv  \phi^{*i}_\beta = - \partial_i v^{N,i}_\beta $.  
	This confirms that the ergodic constant coincides with the associated minimal cost and
	concludes the proof.
	\epf
\subsection*{Open loop model}
To express the cost of the global functional $J^N$ from \eqref{opencostfunctional}, we write (compare with \eqref{fullcost} above)
	\begin{equation} \label{opencostfunction}
	f_{\beta}^{N}(\x,\boldsymbol{\alpha}) := \frac{1}{2} \Vert \boldsymbol{\alpha} \Vert^2 + \frac{\Vert \x \Vert^2}{8} 
		+ \frac{\beta}{8} (\beta-2\sigma^2) \sum_{i=1}^N \frac{(h_2 * \mu_\x^{N,i})(x^i)}{N-1}. 
	\end{equation}

\bt \label{openverificationthm}
	Fix $\beta > \sigma^2 > 0$ in \eqref{opencostfunction} and recall from \eqref{potentialfunction} the solution pair $(W_\beta(\x), \lambda^{N}_\beta)$ to the ergodic HJB equation \eqref{opendysonnash}.	
	Fix an interior initial condition $\x_0 \in \mc{W}^N$. Let $(\X^*_t)_{t \geq 0}$ be given by $\beta/\sigma^2$--Dyson Brownian motion \eqref{dbm}. 
	Then the profile $(\boldsymbol{\alpha}^*_t)_{t \geq 0} := (- \nabla_\x W_\beta(\X_t^*))_{t \geq 0}$
	is a global minimizer over the class $\mb{A} \subset \mb{A}^{(N)}$ of strategies $(\boldsymbol{\alpha}_t)_{t \geq 0}$ such that 
	$(\X_t)_{t \geq 0}  = (\X^{\boldsymbol{\alpha}}_t )_{t \geq 0} $ of \eqref{openNdynamics} satisfies the stability conditions
	\begin{equation} \label{opencontrolconditions}
		\ba
		\limsup_{T \to \infty} \frac{1}{T} \mb{E} \int_0^T f_\beta^{N}(\X_t, \boldsymbol{\alpha}_t) dt < \infty, \ \ \  
		& \limsup_{T \to \infty} \frac{1}{T} \mb{E}[W_\beta(\X_T) - W_\beta(\x_0)] = 0.
		\ea
	\end{equation}
	The global cost under the equilibrium dynamics $(\X^*_t)_{t \geq 0} =(\X^{\boldsymbol{\alpha}^*}_t)_{t \geq 0}$ then satisfies
	$$
	\inf_{\boldsymbol{\alpha} \in \mb{A}} J^{N}(\boldsymbol{\alpha} | \x_0) 
	=  J^{N}(\boldsymbol{\alpha}^* | \x_0) = N \cdot \lambda^N_\beta =  N \cdot \left ( \frac{\sigma^2}{4(N-1)} + \frac{\beta}{8} \right ). 
	$$
	Further, for the Dyson game, the strategy $(\boldsymbol{\alpha}^*_t)_{t \geq 0} \in \mb{A}^{(N)}$ is also an open loop Nash equilibrium over classes 
	$\mb{A}^i \subset \mb{A}^i(\boldsymbol{\alpha}^{*-i})$, $1 \leq i \leq N$, of deviations $(\eta_t)_{t \geq 0}$ 
	such that the limits in \eqref{opencontrolconditions} exist, i.e., 
	$$
	\inf_{ \eta \in \mb{A}^i} J^{N,i}(\eta | \x_0,\boldsymbol{\alpha}^{*-i} ) =  J^{N,i}(\alpha^{*i} | \x_0, \boldsymbol{\alpha}^{*-i}), \ \ \ 1 \leq i \leq N.
	$$ 
	Finally, when $\beta = \sigma^2$ in the cost \eqref{opencostfunction} of $J^N$,
	the global minimization problem admits a minimizing sequence with value $N \cdot \lambda^N_{\sigma^2}$.
	\et 
	\bpf
	The proof follows essentially verbatim that of the closed loop Verification Theorem \eqref{verificationthm}, 
	except for the last two statements. The first one follows because we have identified in Proposition \ref{potentialgame} that the potential game structure 
	holds over stable deviations where the limit of the time average cost \eqref{openloopcostfunctional} exists. 
	To sketch the last claim, for $\beta>\sigma^2$, $1 \leq i \leq N$, and with $\X \sim \mu_\beta^N$ of \eqref{invariantmeasure}, we have
	\begin{equation} \label{infimizingcalc}
	\ba
	\mb{E}[f_{\beta}^{N}(\X,-\nabla_\x W_\beta(\X) )] - \mb{E}[f_{\sigma^2}^{N}(\X,-\nabla_\x W_\beta(\X))] 
		& = \left [  \frac{\beta}{8}(\beta - 2 \sigma^2 ) - \left ( - \frac{\sigma^4}{8} \right )    \right ] 
		\sum_{i=1}^N \frac{\mb{E} (h_2 * \mu^{N,i}_\X)(X^i)}{(N-1)} \\
		& = \left [ \frac{\beta}{4} - \frac{(\beta+\sigma^2)}{8}  \right ] (\beta-\sigma^2) \sum_{i=1}^N  \frac{\mb{E} (h_2 * \mu^{N,i}_\X)(X^i)}{(N-1)} \\
		& \overset{\beta \downarrow \sigma^2}{\to} 0,
	\ea
	\end{equation}
	where we rely on \eqref{finiteh2} to know that the factor right of the square brackets is bounded as $\beta \downarrow \sigma^2$.  
	\epf
		\begin{rem} \label{averageindex}
%	The above proof shows the content of the Verification Theorem \ref{verificationthm} is 
%		the right structure to combine with the integration by parts trick \eqref{ibp1},\eqref{ibp2}.
%		To see this, one can symmetrize a ranked particle system to ensure exchangeability and tightness
%		while preserving the empirical distribution; see Sections I.3.(e) and II.5
%		of Sznitman \cite{ass1}. 
%		Correspondingly, 
It is now straightforward to conclude 
		an integrated version of the limit \eqref{singularcostlimit}.
		Indeed, either make use of Theorem \ref{openverificationthm} directly or
		average the optimal cost calculation of Theorem \ref{verificationthm} over all indices combined with the algebraic identities \eqref{keyidentity1}, \eqref{keyidentity2} to compute
		\begin{equation} \label{averageindexcalc}
			\lim_{N \to \infty} \frac{1}{N} \cdot \sum_{i=1}^N \frac{ \mb{E}(h_2 * \mu_\X^{N,i})(X^i)}{N-1} = \frac{1}{2(\beta-\sigma^2)} 
			= \frac{\pi^2 \beta}{3(\beta - \sigma^2)} \mb{E}[m_\beta(\bar{X})^2],
		\end{equation}
		where $\X \sim \mu_\beta^N$ and $\bar{X} \sim \mu_\beta$.
		Thus, Proposition \ref{conjecture} is difficult because
		one cannot exploit the algebraic identities \eqref{keyidentity1}, \eqref{keyidentity2} that occur upon averaging. 
	\end{rem}

%\begin{remark} \label{optimalparameterremark3}
%	A similar discussion as in Remark \ref{optimalparameterremark2} applies for Theorem \ref{openverificationthm}, but note 
%	that we may take any $\beta \geq \sigma^2$ in the auxiliary global problem; indeed, compared to \eqref{optimalparameter}, 
%	the analogous open loop optimal parameter that permits verification is now $\beta^* :=\max \{ \beta, 2\sigma^2 - \beta \}$ for \emph{any} $\beta \in \R$.
%\end{remark}
\subsection*{Comparison of closed and open loop models} \label{deift9} 
	Keeping $\sigma>0$ fixed, we know by Theorems \ref{verificationthm}, \ref{openverificationthm} that the Nash--optimal repulsion
	parameters are given by the larger roots in the variable ``$\beta$" of the quadratic relationships 
	$C_2 = \frac{\beta}{2} \left ( \frac{3}{4} \beta - \sigma^2  \right )$ from Lemma \ref{Nsubsolution} for $C_2 >-\sigma^4/8$
	and $C_2 = \frac{\beta}{4} (\beta - 2 \sigma^2)$ from Lemma \ref{openloopehjbtheorem} for $C_2 \geq -\sigma^4/4$; these roots are respectively given by $\beta_{closed}(C_2), \beta_{open}(C_2)$ of \eqref{roots}, which are graphed in top left plot of Figure \ref{onlyfigure}.
	For the open loop model, we can approximate individual player costs using the mean field equation \eqref{openloopmasterequation} below
	(see top right plot of Figure \ref{onlyfigure}).		
	We can also approximate the average cost under the open loop equilibrium, 
	which we denote by $\bar{\lambda}^N_{open}(C_2)$; 
	namely, averaging \eqref{openloopcostfunctional} over $1 \leq i \leq N$ and using the limit calculation \eqref{averageindexcalc} of Remark \ref{averageindex},
	we have for $N$ large 
	$$
	\bar{\lambda}^N_{open}(C_2) := \lambda^N_{\beta_{open}(C_2)} +\frac{C_2}{2} \cdot \frac{1}{N} \sum_{i=1}^N \frac{\mathbb{E}(h_2*\mu_{\X}^{N,i})(X^i)}{N-1} 
	\approx \frac{\beta_{open}(C_2)}{8} +\frac{C_2}{4(\beta_{open}(C_2)-\sigma^2)},
	$$
	where $\X \sim \mu_{\beta_{open}(C_2)}^N$ (recall definition \eqref{invariantmeasure} of $\mu_\beta^N$).
	\begin{figure}[h!] 
		\centering
		\includegraphics[width=.49 \textwidth]{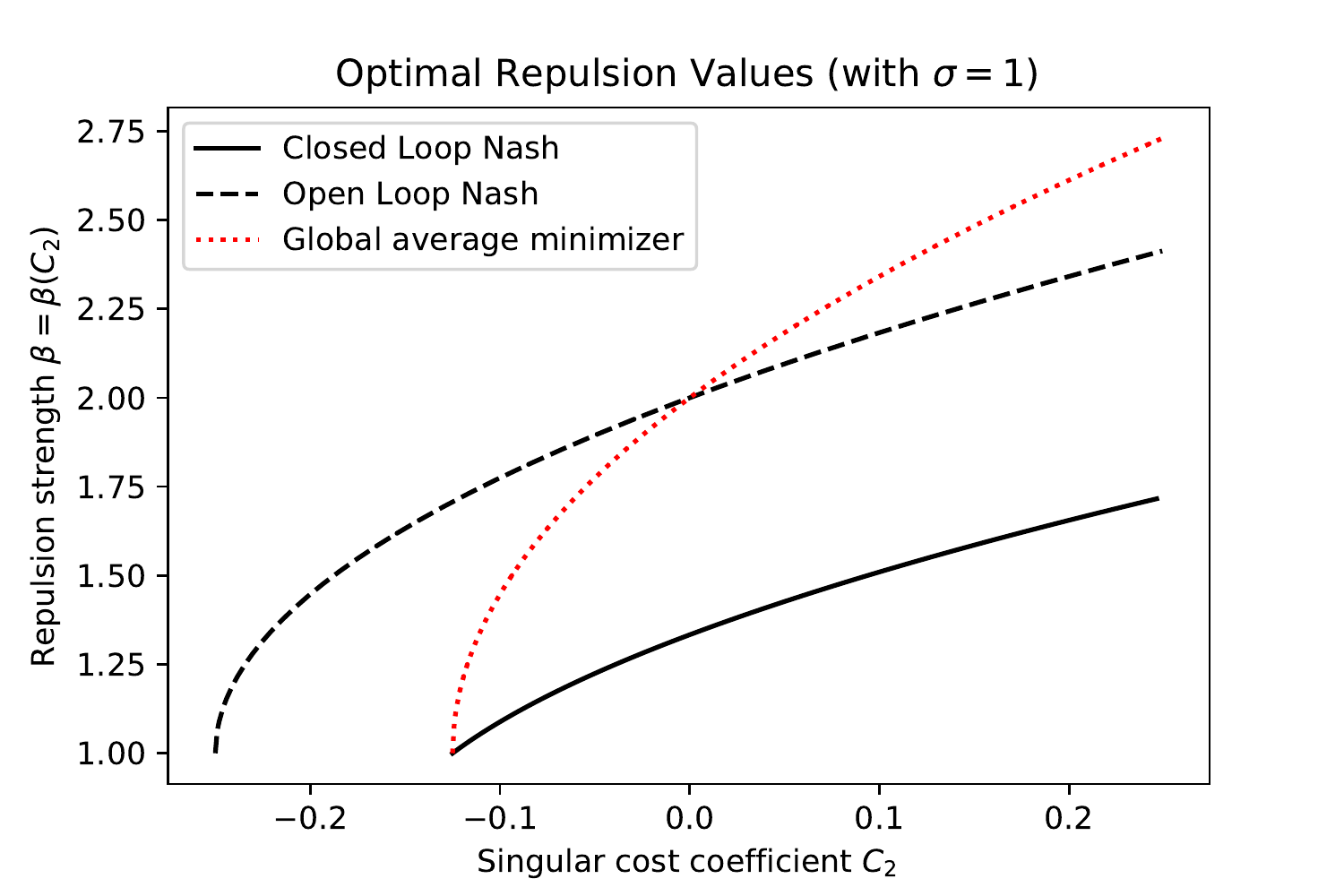}
		\includegraphics[width=.49 \textwidth]{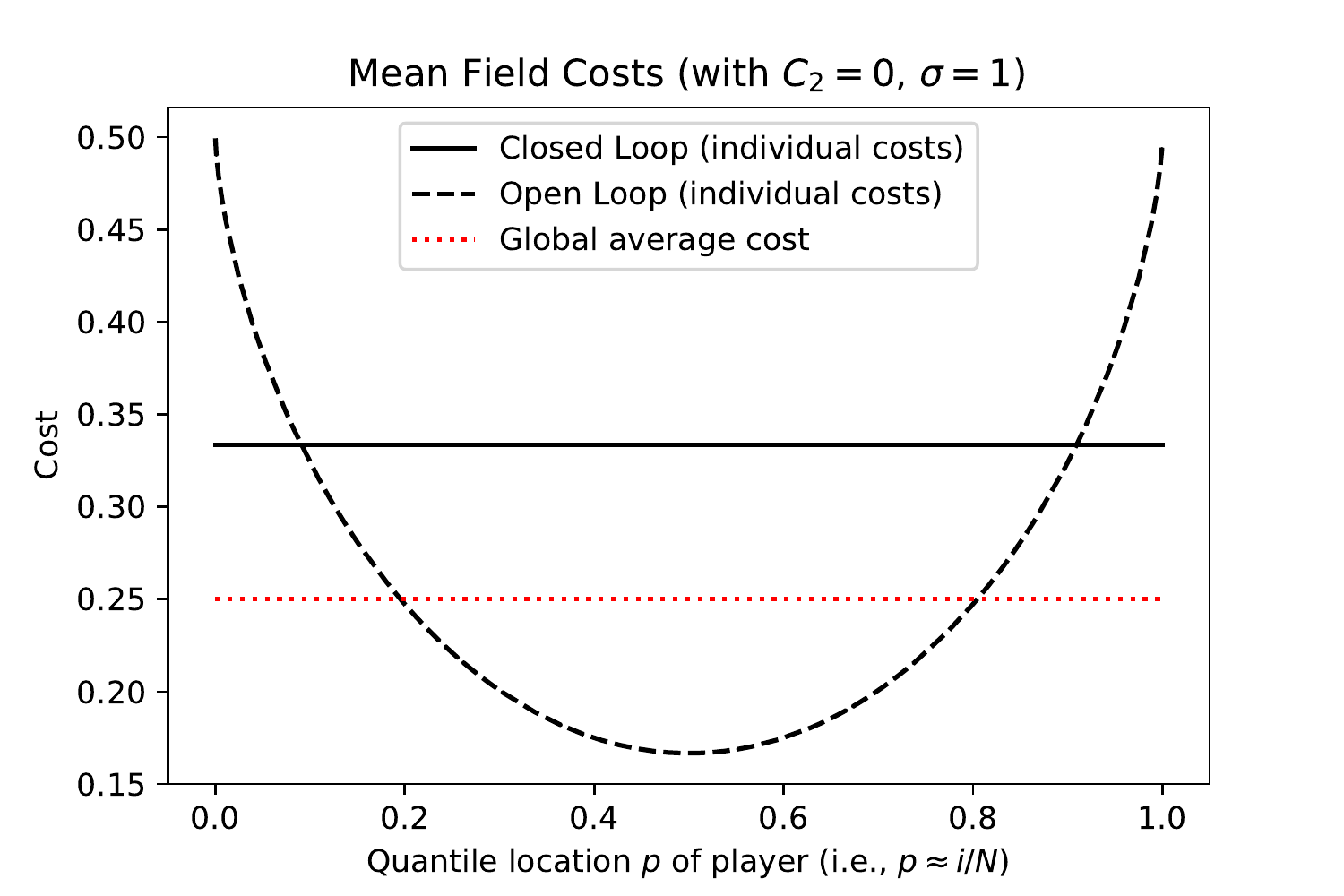}
		\includegraphics[width=.49 \textwidth]{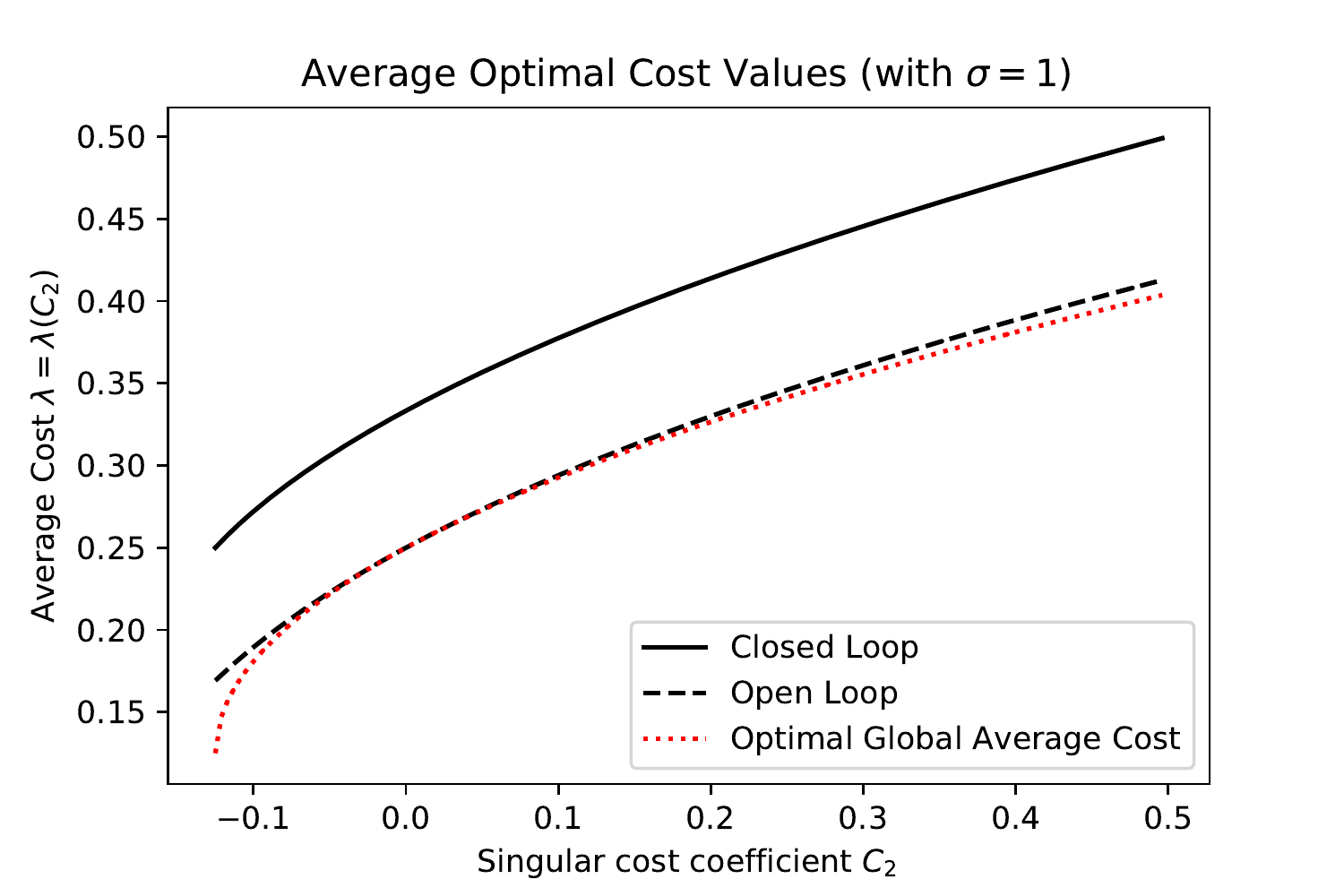}
		\includegraphics[width=.49 \textwidth]{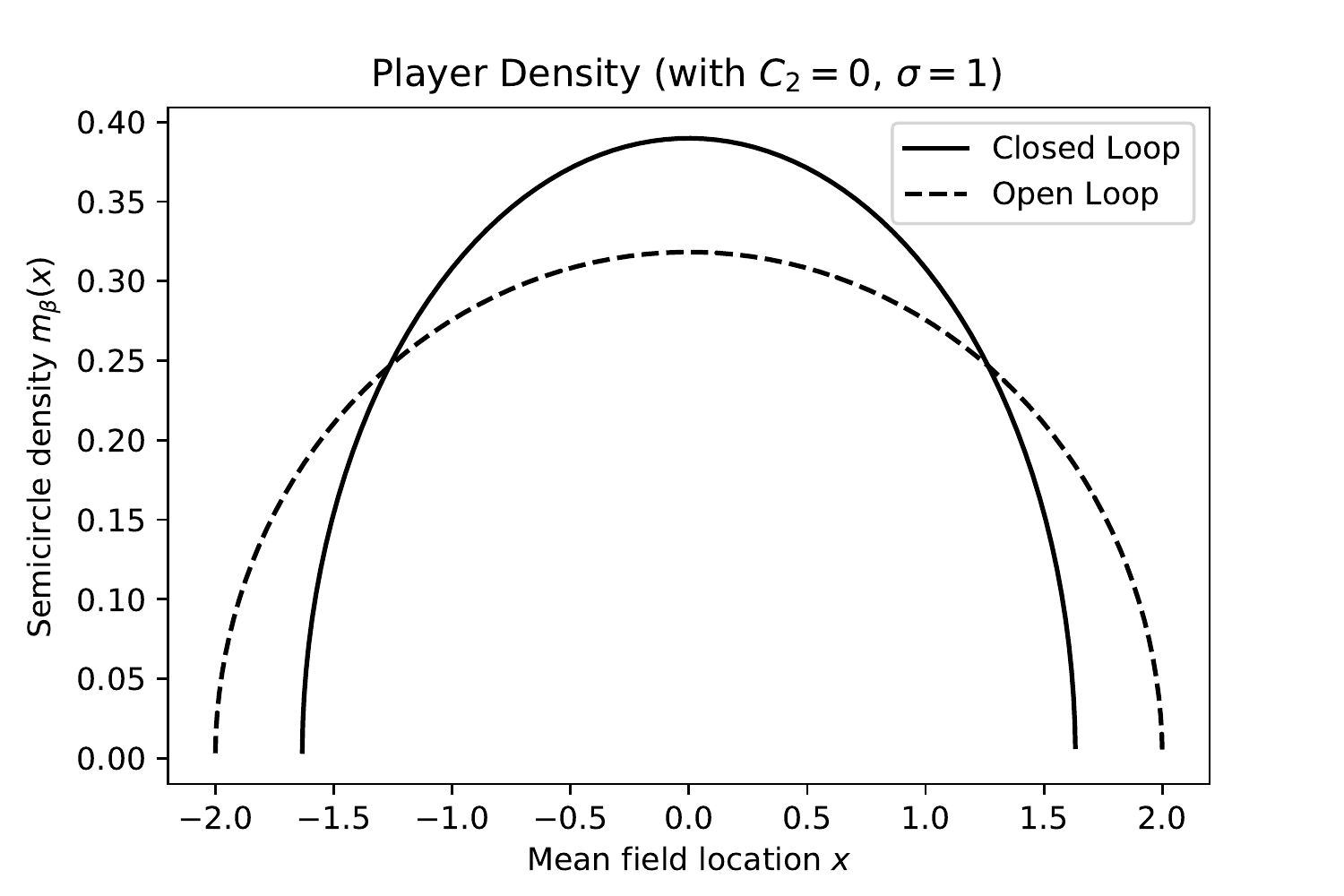}
		\caption{The top left plot compares the Nash--optimal $\beta_{closed}(C_2), \beta_{open}(C_2)$ and global average minimizing $\beta_{open}(2C_2)$ parameters. 
		The bottom left compares the average cost $\bar{\lambda}_{open}^N(C_2)$ under the open loop Nash equilibrium 
		with the optimal global average cost $\lambda^N_{\beta_{open}(2C_2)}$, which only coincide at $C_2=0$; in particular, except at this coincidence, 
		the Nash equilibriums are never global average minimizers.
		Fixing $C_2=0$, the top right compares the (approximate) optimal costs by location for individual players of each model. 
		The bottom right compares the equilibrium player densities.}
		\label{onlyfigure}
	\end{figure}
	
	Notice the open loop strategy prescribes a higher repulsion to accommodate the players densely packed around the origin, i.e.,
	$\beta_{open}(C_2)>\beta_{closed}(C_2)$, and in doing so achieves a lower average cost, 
	i.e., $\bar{\lambda}^N_{open}(C_2)< \frac{1}{N} \sum_{i=1}^N \lambda_{\beta_{closed}(C_2)}^{N,i} \approx \frac{\beta_{closed}(C_2)}{4}$.
	However, the closed loop Nash equilibrium is fair to all (no matter their rank) while the open loop Nash equilibrium
	has greater cost for players away from the origin.  
	Lastly, recall $\lambda^N_{\beta_{open}(C_2)}$ is an auxiliary cost; indeed,
	$\beta_{open}(2C_2)$ minimizes the actual global average cost with minimum $\lambda^N_{\beta_{open}(2C_2)}$, for given $C_2 > -\sigma^4/8$
	(see red curves in Figure \ref{onlyfigure}). 

	For either model of player information, we may write the equilibrium profile in feedback form using the function 
		$$
		\boldsymbol{\phi}_\beta^*(\x) =  - \nabla_\x W_\beta(\x) = (-\partial_iv^{N,i}_\beta(\x))_{i=1}^N 
		= \left ( \frac{\beta}{2}(h_1*\mu_\x^{N,i})(x^i) - \frac{x^i}{2}  \right )_{i=1}^N.
		$$	
		Our discussion above implies a surprising consequence:
		the strategies $\boldsymbol{\phi}_{\beta_{closed}(C_2)}^*(\x)$, $\ \boldsymbol{\phi}_{\beta_{open}(C_2)}^*(\x)$ do not converge together as $N \to \infty$ (unless $C_2=0$ \emph{and}
		we remove the state--constraint and reciprocal gap integrability conditions from our Definition \ref{admissibility} of admissibility). Indeed,
	the closed and open loop Nash equilibriums always yield different mean field behavior.  
	This discrepancy appears in the radius of the limiting semicircle law $\mu_\beta$ (see the bottom right plot of Figure \ref{onlyfigure}),
	but the repulsion parameter $\beta/\sigma^2$ also distinguishes the local limit behavior. 
	We thus believe this result offers a new perspective on Problem 9 of Deift's list \cite{deift1}, 
	which asks to construct a model to explain \v{S}eba's findings \cite{goegue}
	that gaps between parked cars can exhibit GUE statistics ($\beta=2\sigma^2$) on a two--way street but GOE statistics ($\beta=\sigma^2$) on a one--way street.
	Although we do not construct a specific model for this observational study, 
	we have shown rigorously that the local limit can vary depending only on the chosen model of player information 
	and not just on the optimization problem players face. For example, if $C_2=0$, $\sigma=1$, 
	then $\beta_{closed}(0) =  4/3$ and $\beta_{open}(0) = 2$ 
	(see the right two plots in Figure \ref{onlyfigure}; note the latter corresponds to GUE \eqref{density} under mean field scaling).

	\section{Mean field equations} \label{meanfieldapproximation}
	Recall the definition \eqref{masterfield} of $U_\beta(x,\mu)$,
	which is readily observed to be the mean field analog of the solution $v^{N,i}_\beta(\x)$ of \eqref{ehjbsolution}.
	Contrastingly, the state cost $F^{N,i}(\x)$ of \eqref{statecost} cannot naively be written as a function with a probability measure argument
	because the singular cost term ``$(h_2 * \mu)(x)$" is not well defined for $x \in \text{supp}(\mu)$. 
	Instead, we will find that this ill--behaved transform should be replaced by a local term proportional to the squared density ``$m(x)^2$".

To serve as the mean field analog of the ergodic HJB equation \eqref{opendysonnash} 
	for the global minimization problem auxiliary to the open loop model,
	we introduce the following ergodic Hamilton--Jacobi equation on the Wasserstein space $\mc{P}^3_2(\R)$: 
	\begin{equation} \label{openloopmasterequation}
	\ba
	\frac{1}{2} \int_\R |\partial_\mu \mc{U}(\mu)(x)|^2 \mu(dx) 
		= \int_\R \frac{x^2}{8} \mu(dx) + \frac{\pi^2\beta^2}{24}  \int_\R m(x)^3 dx   - \lambda
	\ea
\end{equation}
 (such an ergodic Hamilton-Jacobi equation appears in (1.4) of Gangbo--\'{S}wi\c{e}ch \cite{gangbo1}).  
 Given the form of the function $W_\beta(\x)$ of \eqref{potentialfunction}, we are motivated to consider the functional
\begin{equation} \label{meanfieldpotential}
	\mc{U}_\beta(\mu):= \frac{1}{4} \left ( \int_\R x^2 \mu(dx) - \beta \int_\R \int_\R \log |x-y| \mu(dy) \mu(dx)  \right ), \ \ 
	\mu \in \mc{P}^p_2(\R), \ \ 1 < p < \infty.
\end{equation}

	\bl \label{lemmaformalcalcs}
	The \emph{formal} calculations 
	\begin{equation} \label{formalcalcs}
	\partial_\mu \mc{U}_\beta(\mu)(x) = \partial_xU_\beta(x,\mu) = \frac{x}{2} - \frac{\beta}{2}H\mu(x) , 
		\ \ \ \partial_\mu U_\beta(x,\mu)(z) =  \frac{\beta}{2} \frac{1}{x-z}, 
	\end{equation} 
	obtained from the computational expression \eqref{wassersteingradient} can be understood rigorously as follows: 

	For every $\mu \in \mc{P}_2^3(\R)$, 
	the \emph{subdifferential of minimum norm} is the unique element $\partial_\mu \mc{U}_\beta(\mu)(x) \in L^2(\mu)$ satisfying, 
	for every $\phi \in C^{1}_0(\R)$, 
	\begin{equation} \label{weakcarrillo}
	\int_\R \partial_\mu \mc{U}_\beta(\mu)(z) \phi(z) \mu(dz) =
	\int_\R \frac{z}{2} \phi(z) \mu(dz) - \frac{\beta}{4} \int_\R \int_\R \frac{\phi(z) - \phi(y)}{z-y} \mu(dz) \mu(dy).
	\end{equation}

	For every $\mu \in \mc{P}_2^p(\R)$, $2\leq p<\infty$, with density $m(x)$ and for every $\phi(x) \in L^2(\mu)$, we define  
	\begin{equation} \label{wassersteinaction}
	\int_\R \partial_\mu U_\beta(x,\mu)(z) \phi(z) \mu(dz) 
		:= \lim_{\delta \downarrow 0} \frac{U_\beta(x,(T_\delta)_\# \mu) - U_\beta(x,\mu)}{\delta}  = \frac{\beta}{2}H[\phi \cdot m](x),
	\end{equation}
	where $T_\delta(z):=z+\delta\phi(z)$.
	\el
	\bpf
	Either Lemma 3.7 of Carrillo--Ferreira--Precioso \cite{carrillo1} or Lemma 5.3 of Berman--\"{O}nnheim \cite{berman2016propagation} 
	establishes the existence of a minimal selection $\partial_\mu \mc{U}_\beta(\mu)(x) \in L^2(\mu)$ for every $\mu \in \mc{P}_2^3(\R)$,
	and the weak expression \eqref{weakcarrillo} appears as (3.12) of the first source or as (5.3) of the second.
	Turning to $U_\beta(x,\mu)$, the expression on the righthand-side of
	 \eqref{wassersteinaction} makes sense since $\phi(x)m(x) \in L^{\frac{2p}{p+1}}(\R)$:
	\begin{equation} \label{holder}
	\int_\R (\phi(x)^2m(x))^{\frac{p}{p+1}} \cdot m(x)^{\frac{p}{p+1}} dx 
	\leq \left ( \int_\R \phi(x)^2 m(x) dx \right )^{p/(p+1)} \left ( \int_\R m(x)^p dx \right )^{1/(p+1)} <\infty,
	\end{equation} 
	by H\"{o}lder's inequality. Then we can compute the derivative in a transport direction directly: 
	\begin{equation} \label{metricslopecalc}
		\ba
\lim_{\delta \downarrow 0} \frac{U_\beta(x,(T_\delta)_\# \mu) - U_\beta(x,\mu)}{\delta} 
		& =   \lim_{\delta \downarrow 0} \frac{-\frac{\beta}{2} \int_{\R} \big[\log(|z+\delta\phi(z)-x|-\log(|z-x|) \big] \mu(dz)}{\delta}
		=\frac{\beta}{2} H[\phi \cdot m](x), 
		\ea
	\end{equation}
	where the last equality follows by the monotone convergence theorem (see also the proof of Lemma 2.45 of Deift--Kriecherbauer--McLaughlin \cite{deiftpotential}).
	\epf
\begin{lem}
\label{masterequation} 
	Fix $\beta \in \R$. For every $\mu \in \mc{P}_{2}^p(\R)$, $2 \leq p < \infty$, the pair $(U_\beta(x,\mu), \frac{\beta}{4})$ 
	forms a solution of the Voiculescu--Wigner master equation \eqref{dysonmaster}. 
	For every $\mu \in \mc{P}_{2}^3(\R)$,
	the pair $(\mc{U}_\beta(\mu),\frac{\beta}{8})$ forms a solution to the ergodic Hamilton--Jacobi equation \eqref{openloopmasterequation}.
\end{lem}
%	\begin{remark}
%	As Cardaliaguet points out on pg. 3 of \cite{nonlocaltolocal}, the local dependence on the righthandside of the master equation \eqref{dysonmaster}
%	makes it difficult to define a general notion of solution. For our purposes, it is enough that we can make sense of each term and establish equality. 
%	\end{remark}
	\bpf
	For the first statement,
	note by item (2) of Theorem 2.2 in Carton--Lebrun \cite{cl1} that
	for any $m \in L^p(\R)$ with $2 \leq p < \infty$, we have the Hilbert transform \emph{product rule}
	\begin{equation} \label{productformula}
		\pi^2 m^2(x) = (Hm)^2(x) - 2H[mHm](x),  
	\end{equation}
	for almost every $x \in \R$. Moreover, since $\mu \in \mc{P}^p_2(\R)$, we have $x m(x) \in L^{\frac{2p}{p+1}}(\R)$ by \eqref{holder}.
	Since $\frac{2p}{p+1} \geq \frac{4}{3}$, the two expressions of \eqref{formalcalcs} from Lemma \ref{lemmaformalcalcs} can now be combined rigorously to compute
	\begin{equation} \label{threetwo}
	\int \partial_\mu U_\beta(x,\mu)(z)\partial_xU_\beta(z,\mu)\mu(dz) =  \frac{\beta}{4}H[zm(z)](x) - \frac{\beta^2}{4} H[mHm](x),
	\end{equation}
	Next, we have 
	$$
	 \frac{1}{2} \left ( \frac{x}{2} - \frac{\beta}{2}H\mu(x) \right )^2 = \frac{x^2}{8} + \frac{\beta^2}{8}(H\mu)^2(x) -  \frac{\beta}{4} x H\mu(x),
	$$ 
	and these last two equations lead us to compute	for almost every $x \in \R$
	$$
	\frac{\beta}{4} \left ( H[zm(z)](x) - x Hm(x) \right) = - \frac{\beta}{4}. 
	$$
 	The statement then follows by combining terms and using the product rule \eqref{productformula}.

For the second claim, Lemma 6.5.4 of Blower \cite{gb1} or Lemma 3.3 of Voiculescu \cite{voic1} establishes the identity 
\begin{equation} \label{bloweridentitiy}
\int_\R (H\mu)^2(x) \mu(dx) = \frac{\pi^2}{3} \int_\R m^3(x) dx
\end{equation}
for $\mu \in \mc{P}^3(\R)$. Further, for $\mu \in \mc{P}^3_2(\R)$ we can integrate the identity
$$
x (H\mu)(x) = 1 - \text{p.v.} \int_\R \frac{y}{y-x} \mu(dy) 
$$
against $\mu$ to compute $\int_\R x (H\mu)(x) \mu(dx) = \frac{1}{2}$. The proof then follows by direct computation.
\epf
	\begin{rem} \label{voicconsistency}
		The notion of \emph{free information}, defined on $L^3(\R)$ by $\int m^3(x) dx$,
		was introduced by Voiculescu in \cite{voic1} and the identity \eqref{bloweridentitiy} indicates the role of ``$H\mu(x)$" as \emph{score function}. 
		The Hamilton--Jacobi equation \eqref{openloopmasterequation} essentially appears as the heuristic limit (1.4.4) of Biane--Speicher \cite{pbrs1}; indeed,
		when $\lambda=\frac{\beta}{8}$, the righthandside of \eqref{openloopmasterequation}
		can be written as the \emph{relative free Fisher information}, defined on $\mc{P}_2^3(\R)$ by (cf. Section 6.1 of \cite{pbrs1})
		\begin{equation} \label{relativeinfo}	
		\frac{1}{2} \int \left ( \frac{\beta}{2}H\mu(x) - \frac{x}{2} \right )^2 \mu(dx).
		\end{equation}
		Further, taking the Fr\'{e}chet derivative of \eqref{relativeinfo} formally
		yields the righthandside of \eqref{dysonmaster} with $\lambda = \frac{\beta}{4}$.
	Therefore, finding a measure $\bar{\mu}$ such that $U_\beta(x,\bar{\mu})$ is constant corresponds to 
		a first order condition on the relative free Fisher information \eqref{relativeinfo},
		which is consistent with Voiculescu's information--minimizing characterization of Wigner's semicircle law, Proposition 5.2 of \cite{voic1}.
		These considerations motivate why we refer to \eqref{dysonmaster} as the Voiculescu--Wigner master equation 
		(although Lions \cite{lionslectures} is responsible for the lefthandside).  
The idea that RMT-type statistics occur when some metric of information is minimized appears repeatedly in the literature;
see Chapter 3.6 of Mehta \cite{mehta} for the origins of this idea for a fixed ``symmetry class," but the attempt to identify
		a critical point of the repulsion parameter $\beta/\sigma^2$ using \emph{mutual} information seems to be more recent, e.g., \cite{voic6,pw1}.
	\end{rem}

\section{Mean field game formulation of the Dyson Game} \label{MFGformulation}

	Fix a curve $\bmu=(\mu_t)_{t \geq 0} \in C([0,\infty),\mc{P}_2(\R))$ of probability measures with densities $(m_t(x))_{t \geq 0}$. 
	Since we lose the diffusion in the limit, we adopt a \emph{weak formulation}
	of the associated mean field game, where one controls the law of the state.
\begin{defn}\label{meanfieldadmissibility}
	A feedback functional $\phi : \R \times \mc{P}(\R) \to \R$ is \emph{admissible for $\bmu=(\mu_t)_{t \geq 0}$}
	if there exists a unique flow $\bnu^\phi = (\nu_t)_{t \geq 0} \in C([0,\infty),\mc{P}_2(\R))$ with densities $(n_t(x))_{t \geq 0}$
	solving the transport equation (in the distributional sense)
	\begin{equation}
		\label{meanfielddynamics}
		\partial_t \nu_t(dx) + \nabla_x \circ ( \phi(x, \mu_t) \nu_t(dx) )=0, \ \ \nu_0 = \mu_0,
	\end{equation}
%	for any $\mc{F}_0$--measurable initial condition $\xi$ with law in $\mc{P}^3_2(\R)$,
%	there is a unique solution $(X_t)_{t \geq 0} = (X^\phi_t)_{t \geq 0}$ to the $\R$--valued dynamics
%	\begin{equation} \label{meanfielddynamics}
%		dX_t = \phi(X_t, \mu_t) dt,  \ \ X_0 = \xi,
%	\end{equation} 
%	such that $(X_t)_{t \geq 0}$ 
and satisfying the integrability condition
	\begin{equation} \label{meanfieldintegrability}
	  \int_0^T \int_\RR [\phi(x,\mu_t)^2 + m_t^2(x)] n_t(x)dx dt  < \infty, \ \ \text{for all} \ \  T>0.
	\end{equation} 
	We denote the class of such feedbacks by $\mc{A}(\bmu)$. 
\end{defn}
	For any initial law $\nu \in \mc{P}_2(\RR)$, consider the optimization problem of minimizing the ergodic cost
	\begin{equation} \label{meanfieldcostfunctional}
	J(\phi|\nu, \bmu ) := 
	\limsup_{T\to \infty} \frac{1}{T}  \int_0^T \int_\RR \left [\frac{1}{2} \phi(x,\mu_t)^2 + \frac{x^2}{8} + \frac{\pi^2 \beta^2}{8} \cdot m_t(x)^2 \right] n_t(x) dx dt
	\end{equation}
	over $\phi \in \mc{A}(\bmu)$, subject to controlled flows $\bnu^\phi = (\nu_t)_{t \geq 0} \in C([0,\infty),\mc{P}_2(\R))$ satisfying \eqref{meanfielddynamics}.
\begin{defn} \label{mfgsolutions}
	A pair $(\phi,\bmu)$
	with $\bmu = (\mu_t)_{t \geq 0} \in C([0,\infty),\mc{P}_2(\R))$
	and $\phi(x,\mu) \in \mc{A}(\bmu)$
	 is an \emph{ergodic strong mean field equilibrium (MFE) over the class $\mc{A} \subset \mc{A}(\bmu)$}
	if the flow $\bnu^\phi = (\nu_t)_{t \geq 0}$ of \eqref{meanfielddynamics}
	satisfies the fixed point condition $\bnu^\phi = \bmu$ and the optimality condition 
	\begin{equation} \label{mfgoptimality}
		\inf_{\psi \in \mc{A}} J(\psi | \mu_0 , \bmu ) = J(\phi | \mu_0, \bmu ).  
	\end{equation}
\end{defn}
Recall $\beta/\sigma^2$--Dyson Brownian motion $(\X^*_t)_{t \geq 0}$ of \eqref{dbm}. 
Theorem 1 of Rogers--Shi \cite{lrzs1} (see also Section 4.3.2 of \cite{agz1})
provides natural conditions under which, for any $T>0$, $(\mu^N_{\X^*_t})_{0 \leq t \leq T}$ converges a.s. 
on $C([0,T]; \mc{P}(\R))$ as $N \to \infty$ to the unique solution $(\mu^*_t)_{0 \leq t \leq T}$ of the McKean--Vlasov equation 
\begin{equation} \label{mkv1}
	\int_\R f(x) \mu_T(dx) = \int_\R f(x) \mu_0(dx) + \frac{1}{2} \int_0^T \left (
	\frac{\beta}{2} \int_\R \int_\R \frac{f'(x)-f'(y)}{x-y} \mu_t(dx) \mu_t(dy) - \int_\R x f'(x) \mu_t(dx) \right )dt
\end{equation}
for any twice continuously differentiable test function $f$ with $f,xf'(x),f''(x)$ bounded.
The flow $\bmu^*=(\mu^*_t)_{t \geq 0}$ is thus the natural candidate for a
strong MFE in the sense of Definition \ref{mfgsolutions}.  
Before stating the next theorem that verifies this candidate, we define for $(x,\mu,y,\alpha) \in \R \times \mc{P}^3_2(\R) \times \R \times \R$ the running cost
	\begin{equation} \label{meanfieldcost}
		f(x,\mu,\alpha) := \frac{\alpha^2}{2} + \frac{x^2}{8} + \frac{\pi^2\beta^2}{8} m(x)^2 
	\end{equation}
	and the (control) Hamiltonian 
	\begin{equation} \label{meanfieldhamiltonian} 
		H(x,\mu,y,\alpha) : = y \cdot \alpha + f(x,\mu,\alpha).
	\end{equation}
	Recall the definitions \eqref{masterfield}, \eqref{meanfieldpotential} of $U_\beta(x,\mu)$, $\mc{U}_\beta(\mu)$, respectively.

\begin{thm} \label{meanfieldverification}
	Fix $\beta>0$ and an initial condition $\mu^*_0 \in \mc{P}_2^3(\R)$. Assume $\bmu^* = (\mu_t^*)_{t \geq 0}$ satisfies \eqref{mkv1}. Consider the feedback 
\begin{equation} \label{mainfeedback}
	\phi^*_\beta(x,\mu):= - \partial_\mu \mc{U}_\beta(\mu)(x) = -\partial_xU_\beta(x,\mu) = \frac{\beta}{2} H\mu(x) - \frac{x}{2}.
\end{equation}
	Then $\bmu^* = (\mu^*_t)_{t \geq 0} \in C([0,\infty),\mc{P}^3_2(\R))$
	and the pair $(\phi^*_\beta, \bmu^* )$ 
	forms an ergodic strong MFE over the class $\mc{A} \subset \mc{A}(\bmu^*)$ of deviations $\psi(x,\mu)$ such that 
	$\bnu^\psi = (\nu_t)_{t \geq 0}$ satisfies \eqref{meanfielddynamics} with $\nu_0 = \mu_0^*$ and the stability conditions
	\begin{equation} \label{meanfieldcontrolconditions}
\begin{aligned}
		\limsup_{T \to \infty} & \frac{1}{T} \int_0^T \int_\RR f(x, \mu^*_t, \psi(x, \mu^*_t)) \nu_t(dx) dt < \infty, \\ 
		 \limsup_{T \to \infty} & \frac{1}{T} \left [\int_\RR U_\beta(x, \mu^*_T) \nu_T(dx) - \int_\RR U_\beta(x,\mu^*_0)\nu_0(dx) \right ] = 0.
\end{aligned}
	\end{equation}
	Further, the cost achieved by the optimal feedback $\phi^*_\beta(x,\mu)$ satisfies
	$$
	\inf_{ \psi \in \mc{A}} J(\psi |\mu_0^*,\bmu^*) = J(\phi^*_\beta|\mu_0^*,\bmu^*) = \lambda_\beta=\frac{\beta}{4}.
	$$
\end{thm}
\begin{proof}
	We first note that the curve $\bmu = (\mu_t^*)_{t\geq 0}$ is in $C([0,\infty),\mc{P}^3_2(\R))$ by known regularity results: first, 
	$\mu_t^*$ remains in $\mc{P}_2(\R)$ for all $t >0$ being a gradient flow of the ``free energy'' functional $\mc{U}_\beta(\mu)$
	by Theorem 3.2.(1) of Carrillo--Ferreira--Precioso \cite{carrillo1}
	\footnote{As pointed out in Remark 5.9 of Berman--\"{O}nnheim \cite{berman2016propagation}, there is apparently an error here 
	regarding the domain of $\mc{U}_\beta$, but the positive initial density in $L^3(\R)$
	ensures sufficient regularity of the flow $(\mu_t^*)_{t \geq 0}$ for our purposes.
	}; 
	second, $\mu_t^*$ remains in $\mc{P}^3(\R)$ for all $t >0$
	by Corollary 5.3 of Biane--Speicher \cite{pbrs1} or Remark 7.7 of Biler--Karch--Monneau \cite{regularityofflow}
	after rescaling variables as in Section 2.1 of \cite{carrillo1} (see also Proposition 4.7 of Voiculescu \cite{voic1}). Now 
	Theorem 3.8 of \cite{carrillo1} implies the densities $(m^*_t(x))_{t \geq 0}$ of $(\mu_t^*)_{t\geq 0}$ satisfy the following
	nonlinear transport equation (in the distributional sense; cf. (8.1.3) of \cite{ambrosio1}):
\begin{equation} \label{mkv2}
	\partial_t m_t^*(x) 
	  + \nabla_x  \circ \left ( m_t^*(x) \phi^*_\beta(x,\mu^*_t) \right ) = 0.
\end{equation}
With the regularity of $\bmu^*$ in hand,
the integrability condition \eqref{meanfieldintegrability} from our notion of admissibility follows
from H\"{o}lder's inequality and from the \emph{free energy identity} (see Theorem 11.2.1 of \cite{ambrosio1}, Proposition 6.1 of \cite{pbrs1}, or Theorem 3.2.(4) of \cite{carrillo1}):
	$$
	\int_0^T \int_\R  \phi_\beta^*(x,\mu^*_t)^2 \mu_t^*(dx) dt = \mc{U}_\beta(\mu^*_0) - \mc{U}_\beta(\mu^*_T)  < \infty. 
	$$
Hence, we have completely checked the Definition \ref{meanfieldadmissibility} of admissibility for the feedback \eqref{mainfeedback}, i.e., $\phi_\beta^*(x,\mu) \in \mc{A}(\bmu^*)$.

	Now it suffices to check the optimality condition \eqref{mfgoptimality} 
	and the stability conditions \eqref{meanfieldcontrolconditions} $\bmu = (\mu_t^*)_{t\geq 0}$. 
	Toward this end, 
	consider an admissible and stable $\psi(x,\mu) \in \mc{A}$, and let $\bnu^\psi = (\nu_t)_{t \geq 0}$ solve \eqref{meanfielddynamics} with $\nu_0 = \mu^*_0$. 
	Then by Lemma \ref{lemmaformalcalcs} and the chain rule of Lemma \ref{chainrule} below, we have (compare with \eqref{Nverificationcalc}) 
	\begin{equation} \label{mfgverificationcalc}
\begin{aligned}
		&\int_\RR U_\beta(x, \mu^*_T) \nu_T(dx) - \int_\RR U_\beta(x,\mu^*_0)\nu_0(dx)   \\ 
		& =  \int_0^T \int_\RR \left ( \partial_xU_\beta(x,\mu_t^*) \cdot \psi(x,\mu_t^*)  
		+ \int_\R \partial_\mu U_\beta(x,\mu_t^*)(z) \phi^*_\beta(z,\mu^*_t) m_t^*(z) dz \right ) \nu_t(dx) dt    \\
		& = \int_0^T \int_\R  \left [\lambda_\beta - f(x,\mu^*_t,\psi(x,\mu^*_t)) \right ] \nu_t(dx) dt  \\
		& + \int_0^T \int_\R \left [ H(x,\mu^*_t,\partial_x U_\beta(x,\mu^*_t),\psi(x,\mu^*_t))
		- H(x,\mu^*_t,\partial_xU_\beta(x,\mu^*_t),-\partial_x U_\beta(x,\mu^*_t))  \right ] \nu_t(dx) dt,
\end{aligned}
	\end{equation}
	where the second equality uses that $U_\beta(x,\mu)$ solves the master equation \eqref{dysonmaster} for $\mu \in \mc{P}^3_2(\R)$ by Lemma \ref{masterequation}. 
	Since the difference of the Hamiltonian values in \eqref{mfgverificationcalc} is nonnegative, we have for $\psi \in \mc{A}$
	$$
	J(\psi | \mu_0^* , \bmu^*) = \limsup_{T \to \infty} \frac{1}{T} \int_0^T \int_\R f(x,\mu^*_t,\psi(x,\mu^*_t)) \nu_t(dx) dt \geq \lambda_\beta = \frac{\beta}{4}.
	$$
	To establish equality here under the feedback $\phi_\beta^*(x,\mu)$ of \eqref{mainfeedback}, 
	note for this choice the difference of the Hamiltonians in \eqref{mfgverificationcalc} vanishes, so we just need to 
	establish the second condition in \eqref{meanfieldcontrolconditions}.  
	But from Theorem 3.2.(3) of \cite{carrillo1}, we have the second moment convergence of $\mu_T^*$ to the semicircle
	law $\mu_\beta$ of \eqref{semicircledensity}
	(recall it is the minimum of $\mc{U}_\beta$ by Theorem 3 of Ben Arous--Guionnet \cite{arous1}) as well as the relative entropy type estimate
	$$
	0 \leq \mc{U}_\beta(\mu_T^*) - \mc{U}_\beta(\mu_\beta) \leq e^{-c(T-s)} [\mc{U}_\beta(\mu^*_s) - \mc{U}_\beta(\mu_\beta)]
	$$
	for every $0\leq s \leq T$ and some constant $c>0$. Combining these two facts, 
	we readily have 
	$$
	\frac{1}{T} \int_\RR U_\beta(x,\mu^*_T) \mu_T^*(dx) = \frac{1}{T} \left [   2 \cdot \mc{U}_\beta(\mu_T^*) - \frac{1}{4} \int_\R x^2 \mu_T^*(dx)  \right ] 
	\overset{T \to \infty}{\to} 0,
	$$
	as required. This completes the proof after checking the application of the chain rule for \eqref{mfgverificationcalc}.
\end{proof}
	Since by Lemma \ref{lemmaformalcalcs} 
	we interpret ``$\partial_\mu U_\beta(x,\mu)(z)$'' through the derivative \eqref{wassersteinaction} rather than as a subdifferential of minimum norm 
	(the latter characterization does not seem straightforward to realize),
	we prove the following chain rule to complement Lemma \ref{lemmaformalcalcs} and justify \eqref{mfgverificationcalc} above.
\begin{lem} \label{chainrule}
	Assume the same setting as for Theorem \ref{meanfieldverification}. Let $\psi(x,\mu) \in \mc{A}(\mu^*_\cdot)$
	so that the laws $\bnu^\psi = (\nu_t)_{t \geq 0} \in C([0,\infty),\mc{P}^3_2(\R))$ 
	have densities $n_t(x)$ satisfying (in the distributional sense) the transport equation \eqref{meanfielddynamics}.
	Then we have for any $T \geq 0$ the following chain rule:
	$$
\begin{aligned}
	\int_{\R}U_\beta(x,\mu_T^*)& \nu_T(dx) - \int_{\R}U_\beta(x,\mu_0^*)\nu_0(dx)  = \\
	& \int_0^T \int_{\R} \left ( \partial_x U_\beta(x,\mu^*_t)\psi(x,\mu^*_t)
	+ \int_{\R}   \partial_\mu U_\beta(x,\mu^*_t)(z)\phi_\beta^*(z,\mu^*_t)\mu_t(dz) \right )\nu_t(dx)dt.
\end{aligned}
	$$
\end{lem}
\begin{proof}
	For any $\nu \in \mc{P}^p_2(\R)$, $1 < p <\infty$, with density $n(x)$, we have by Minkowski's inequality
	\begin{equation} \label{minkowski}
\begin{aligned}
	\left | \int_{\R} \log(|z-x|)\nu(dx) \right |\leq & \int_{|z-x|\leq 1} |\log(|z-x|)|\nu(dx)+\int_{|z-x|> 1}| \log(|z-x|)|\nu(dx)\\
	\leq& 2\Big(\int_0^1 |\log(x)|^\frac{p}{p-1}dx\Big)^\frac{p-1}{p}\|n\|_{L^p(\R)} + \int_{\R}|z-x|\nu(dx)\\
	\leq& C_p \left(\|n\|_{L^p(\R)} + z^2+\int_{\R}x^2\nu(dx)\right).
\end{aligned}
\end{equation}
for some constant $C_p>0$. 
Hence, the mapping $\mu \mapsto \int_\R U_\beta(x,\mu) \nu(dx)$ is continuous on $\mc{P}_2(\R)$ with respect to the $2$--Wasserstein distance $d_2$.
Now fix $t \geq 0$ and let $p=3$. Define $T_\delta(x) := x+\delta \psi(x,\mu^*_t)$, $\delta>0$. 
We know by Proposition 8.4.6 of Ambrosio-Gigli-Savar\'{e} \cite{ambrosio1} that $d_2(\mu_{t+\delta}^*,(T_\delta)_{\#}\mu^*_t)= o(\delta)$.
Hence, by Lemma \ref{lemmaformalcalcs} and the continuity of $\mu\mapsto \int_{\R}U_\beta(x,\mu)\nu(dx)$, we have for almost every $t \geq 0$
$$
\begin{aligned}
	\int_{\R}U_\beta (x,\mu_{t+\delta}^*)\nu(dx)-\int_{\R}U_\beta (x,\mu_t^*)\nu(dx) & =  
	\int_\R \left [ U_\beta(x,(T_\delta)_\# \mu_t^*) - U_\beta(x,\mu_t^*)  \right ] \nu(dx)  \\
	& + \int_\R U_\beta(x,\mu_{t+\delta}^*) \nu(dx)  - \int_\R U_\beta(x,(T_\delta)_\# \mu_t^*) \nu(dx) \\
	& = \delta\int_{\R} \int_\R \partial_\mu U_\beta(x,\mu_t^*)(z)\phi_\beta^*(z,\mu_t^*) \mu^*_t(dz) \nu(dx)+o(\delta).
\end{aligned}
$$
By \eqref{holder} and \eqref{threetwo}, the integrand $x \mapsto \int_\R \partial_\mu U_\beta(x,\mu_t^*)(z)\partial_xU_\beta(z,\mu_t^*) \mu^*_t(dz)$ is in $L^{3/2}(\R)$, 
and so its integral with respect to $\nu \in \mc{P}^3_2(\R)$ is finite by H\"{o}lder's inequality.
The statement then follows from the product rule and the distributional equation for $(\nu_t)_{t \geq 0}$ (along with a density argument relying on the integrability condition \eqref{meanfieldintegrability} to allow for test functions in $L^{3/2}(\RR)$; cf. Remark 8.1.1 of \cite{ambrosio1}).
\end{proof}

%	\begin{remark} \label{quantileconvergence}
%	To recover the ranked player perspective in the mean field setting,
%	we can use equations \eqref{mkv1}, \eqref{mkv2} satisfied by $(\mu_t^*)_{t\geq 0}$
%		to \emph{formally} compute the corresponding flow of distribution functions $(F_t^*(x))_{t \geq 0}$ as
%	$$
%	\partial_t F_t^*(x) = m_t^*(x)\cdot \left ( \frac{x}{2} - \frac{\beta}{2}Hm^*_t(x)   \right).
%	$$
%	From here, we may compute the dynamics for the quantiles $\gamma_t^q:= \inf \{ x \in \R : F_t^*(x) \geq q  \}$, $q \in [0,1]$:
%	\begin{equation} \label{quantiletrajectories}
%	d \gamma_t^q = \frac{-\partial_t F_t^*(\gamma^q_t)}{m^*_t(\gamma_t^p)} dt = 
%	\left ( \frac{\beta}{2}Hm^*_t(\gamma_t^q)  - \frac{\gamma_t^q}{2} \right ) dt = - \partial_x U_\beta(\gamma_t^q,\mu^*_t) dt.
%	\end{equation} 
%	To make these manipulations rigorous, one can probably use the same line of argument as for Proposition 4 of Crisan--Kurtz--Lee \cite{kurtz1}.
%		Note \eqref{quantiletrajectories} is exactly the expected limiting dynamics of a sequence $X^{*i}_t$ of player(s) $i=i(N)$
%		with $\lim_{N \to \infty} i/N = q \in [0,1]$, which motivates the discussion on convergence of optimal trajectories in Remark \ref{extensions}.  
%	\end{remark}

	\section{Recovering the master equation from the $N$-Nash system} \label{justifypassage}

\begin{proof}[Proof of Proposition \ref{conjecture}]
	First, as $N \to \infty$ we have $\mb{E}(X^i - \gamma^q)^2 \to 0$. To see this, we know by Corollary 6.3.5 of Blower \cite{gb1} that $\mu_{\beta,V}^N$ 
	satisfies a logarithmic Sobolev inequality with constant $(N-1)/\sigma^2$ and thus by Proposition 6.7.3 of the same reference satisfies a Poincar\'{e}
	inequality with the same constant, implying $\text{Var}(X^i) \to 0$ as $N \to \infty$.
	The fact that $\mb{E} X^i \to \gamma^q$ as $N \to \infty$ 
	follows from the results of Section 2.6 of \cite{agz1}, which yields the desired $L^2$ convergence. 

	Next, by Section 2.6.2 of Anderson-Guionnet-Zeitouni \cite{agz1},
we know that $\max_{1 \leq i \leq N} |X^i| = \max \{ -X^1, X^N \}$
is exponentially tight at speed $N$. This allows one to restrict
some limit results involving the empirical measure
to compact sets; more specifically, for any $p\geq 1$, we have the $p$--Wasserstein convergence $d_p(\mu_\X^N,\mu_{\beta,V}) \overset{N \to \infty}{\to} 0$ almost surely
(see the discussion after Theorem 1.12
of Chafa\"{i}-Hardy-Ma\"{i}da \cite{chafai2018concentration}).
By the Euler-Lagrange identity \eqref{eulerlagrange}, we have that
$$
\begin{aligned}
 V'(\gamma)^2 & = \beta^2 (H m_{\beta,V})^2(\gamma), \ \ \ 
\int_\R \frac{V'(\gamma) - V'(z)}{\gamma - z} m_{\beta,V}(z)dz   = \beta \left [ (H m_{\beta,V})^2(\gamma) - H[m_{\beta,V} H m_{\beta,V}](\gamma)   \right ].
\end{aligned}
$$
For fixed $N$, we can compute
	$$
	\sum_{k: k \neq i } \frac{1}{x^i - x^k} \cdot \prod_{1 \leq k < \ell \leq N} (x^{\ell} - x^{k})^{\beta/\sigma^2} = 
	\frac{\sigma^2}{\beta} \partial_{i} \left [ \prod_{1 \leq k < \ell \leq N} (x^{\ell} - x^{k})^{\beta/\sigma^2} \right ] 
	$$
	and
	$$
	\left [  \left ( \sum_{k: k \neq i } \frac{1}{x^i - x^k} \right )^2 
	- \frac{\sigma^2}{\beta} \sum_{k: k \neq i } \frac{1}{(x^i - x^k)^2}   \right ] 
	\cdot \prod_{1 \leq k < \ell \leq N} (x^{\ell} - x^{k})^{\beta/\sigma^2} = 
	\left ( \frac{\sigma^2}{\beta} \right )^2 \partial_{i}^2 \left [ \prod_{1 \leq k < \ell \leq N} (x^{\ell} - x^{k})^{\beta/\sigma^2} \right ]. 
	$$ 
	Hence, we have by integration by parts (see the proof of Lemma 4.3.17 of \cite{agz1} and Section 3 of \cite{vgmsedge})
	\begin{equation} \label{ibp1}
		\mb{E} \left [ V'(X^i) \cdot (h_1 * \mu_\X^{N,i})(X^i) \right ] = \frac{\mb{E}V'(X^i)^2}{\beta} - \frac{\sigma^2 \mb{E}V''(X^i)}{\beta(N-1)}  
		 \overset{N \to \infty}{\to} \frac{V'(\gamma^q)^2}{\beta}
		 = \beta (H m_{\beta,V})^2(\gamma^q),
	\end{equation}
	and similarly
	\begin{equation} \label{ibp2}
		\mb{E} \left [ [(h_1 * \mu^{N,i}_\X)(X^i)]^2 - \frac{\sigma^2}{\beta} \frac{(h_2 * \mu^{N,i}_\X)(X^i)}{N-1} \right ] 
		\overset{N \to \infty}{\to}  \frac{V'(\gamma^q)^2}{\beta^2} = (H m_{\beta,V})^2(\gamma^q),
	\end{equation}
	where the limits hold by the $p$-Wasserstein convergence 
	and the polynomial growth of $V$.
	
	Now let $\mc{L}_{\beta,V}$ be the generator of \eqref{generalizeddbm} and
	take $v^{N,i}_{\beta, V}(\x):= \frac{V(x^i)}{2} - \frac{\beta}{2} (h_0 * \mu^{N,i}_\x)(x^i)$, $\x \in \mc{W}^N$, in the invariance identity $\mu_{\beta,V}^N[\mc{L}_{\beta,V}( v^{N,i}_{\beta, V})] = 0$
	(this further application of integration by parts along with Lemma \ref{Nsubsolution} is behind the optimal cost calculation of
	Theorem \ref{verificationthm} for the Gaussian case; see \eqref{ergodiclimit}, \eqref{Nverificationcalc}.
	Indeed, to see that $v^{N,i}_{\beta, V}$ is in the domain of the generator $\mathcal{L}_{\beta,V}$, note that the dynamics \eqref{generalizeddbm} are nonexplosive by Corollary 6.9 of Graczyk--Ma{\l}ecki \cite{pgjm1},
and a calculation similar to \eqref{lyapunovsde} and \eqref{finiteh2} confirms admissibility of \eqref{generalizeddbm} for the Dyson game, and thus the required integrability, when $\beta>\sigma^2$.)
	Calculating as in the proof of Lemma \ref{Nsubsolution} and letting $N \to \infty$ in this invariance identity, we have 
\begin{align}
	\frac{\beta^2}{4}  [ (H m_{\beta,V})^2(\gamma^q) & - H[m_{\beta,V} H m_{\beta,V}](\gamma^q)   ]  = \frac{\beta}{4} \int_\R \frac{V'(\gamma^q) - V'(z)}{\gamma^q - z} m_{\beta,V}(z)dz \nonumber \\
	&
	=\lim_{N \to \infty} \frac{\beta^2}{8} \mb{E} \left [ [(h_1 * \mu^{N,i}_\X)(X^i)]^2 - \frac{\sigma^2}{\beta} \frac{(h_2 * \mu^{N,i}_\X)(X^i)}{N-1}  \right] \nonumber
	\\
	&+ \lim_{N \to \infty} \mb{E} \left [ \frac{V'(X^i)^2}{4} - \frac{\beta}{4} V'(X^i) \cdot (h_1*\mu^{N,i}_\X)(X^i) 
	+ \frac{3\beta}{8} (\beta-\sigma^2)\frac{(h_2 * \mu^{N,i}_\X)(X^i)}{N-1} \right ] \nonumber \\
	& = \frac{\beta^2}{8} \cdot \frac{V'(\gamma^q)^2}{\beta^2} + \frac{V'(\gamma^q)^2}{4} - \frac{\beta}{4} \cdot \frac{V'(\gamma^q)^2}{\beta}  
	+ \frac{3\beta}{8} (\beta-\sigma^2)  \lim_{N \to \infty}  \frac{\mb{E}(h_2 * \mu^{N,i}_\X)(X^i)}{N-1} \nonumber \\
		& = \frac{ \beta^2}{8}(H m_{\beta,V})^2(\gamma^q) + \frac{3\beta}{8} (\beta-\sigma^2)  \lim_{N \to \infty}  \frac{\mb{E}(h_2 * \mu^{N,i}_\X)(X^i)}{N-1}. \nonumber
\end{align}
	Subtracting the term ``$\frac{ \beta^2}{8}(H m_{\beta,V})^2(\gamma^q)$'' from 
	both sides of this last expression, we arrive at
	\begin{equation} \label{generallimit}
		\lim_{N \to \infty} \frac{\mb{E} (h_2*\mu^{N,i}_\X)(X^i)}{N-1} = \frac{8}{3\beta(\beta-\sigma^2)} \cdot 
		 	\frac{\beta^2}{8}  \left [ (H m_{\beta,V})^2(\gamma^q) - 2H[m_{\beta,V} H m_{\beta,V}](\gamma^q)    \right ] = \frac{\pi^2\beta}{3(\beta-\sigma^2)} m_{\beta,V}(\gamma^q)^2,
	\end{equation}
	where the last equality follows by an application the Hilbert transform product rule \eqref{productformula} 
	since $m_{\beta,V}(x) \in L^2(\R)$ (item (2) of Theorem 2.2 in Carton--Lebrun \cite{cl1}).
This completes the proof.
\end{proof}
	Observe the control term has vanishing expectation in optimal equilibrium, i.e., $\mu^N_\beta \left [ \partial_i v^{N,i}_\beta \right ]=0$, 
	and by Euler--Lagrange \eqref{eulerlagrange}, the analogous mean field identity holds exactly for the semicircle law $\mu_\beta$, i.e.,
	$\partial_x U_\beta(x,\mu_\beta) = 0$ for $x \in \text{supp}(\mu_\beta)$.  
	Nevertheless, the next theorem shows that as $N \to \infty$ the control cost term ``$\mu_\beta^N \left [\frac{1}{2}(\partial_i v^{N,i}_\beta)^2 \right]$'' 
	still contributes a local term in the bulk, 
	i.e., at a location $\gamma^q$ with $q \in (0,1)$. We will now see that the calculation of this contribution
	will establish Theorem \ref{mainresult}, one of the main results of this paper.
%	the \emph{localized convergence of equations},
%	from the $N$--Nash system \eqref{dysonnash} to the Voiculescu--Wigner master equation \eqref{dysonmaster}, against \emph{locally optimal} player ensembles.
\begin{proof}[Proof of Theorem \ref{mainresult}]
First note by the Euler--Lagrange equation \eqref{eulerlagrange}, we have
$$
\begin{aligned}
|\partial_x U_\beta(x,\mu_{\beta,V})|^2 & = \left (  \frac{\beta}{2} H\mu_{\beta,V}(x) - \frac{x}{2}  \right )^2 =
\left (\frac{V'(x)}{2}  - \frac{x}{2}  \right )^2 = \frac{V'(x)^2}{4}  + \frac{x^2}{4} - \frac{x V'(x)}{2}, \ \ x \in \text{supp} (\mu_{\beta,V}), 
\end{aligned}
$$
Let $\X \sim \mu^N_{\beta,V}$. Then we can compute
$$
\begin{aligned}
		 \mu_{\beta,V}^N \left [ (\partial_i v^{N,i}_\beta)^2 \right ] 
		&= \frac{\beta^2}{4} \mb{E} \left [  [(h_1 * \mu_\X^{N,i})(X^i)]^2
		- \frac{\sigma^2}{\beta} \frac{(h_2* \mu_\X^{N,i})(X^i)}{N-1}  \right ] \\
		& + \mb{E} \left [ \frac{(X^i)^2}{4} - \frac{\beta}{2} X^i \cdot (h_1 * \mu_\X^{N,i})(X^i) 
		+ \frac{\beta \sigma^2}{4} \frac{(h_2* \mu_\X^{N,i})(X^i)}{N-1}  \right ]
		\\
		& \overset{N \to \infty}{\to} \frac{\beta^2}{4}\cdot \frac{V'(\gamma^q)^2}{\beta^2} 
		+ \frac{(\gamma^q)^2}{4}  - \frac{\beta}{2}\cdot \frac{\gamma^q V'(\gamma^q)}{\beta} 
		+ \frac{\pi^2 \beta^2 \sigma^2}{12(\beta - \sigma^2)} m_\beta(\gamma^q)^2 \\
		& = |\partial_xU_\beta(\gamma^q,\mu_{\beta,V})|^2 + \frac{\pi^2 \beta^2 \sigma^2}{12(\beta - \sigma^2)} m_\beta(\gamma^q)^2, 
\end{aligned}
		$$
where we have used \eqref{singularcostlimit}, \eqref{ibp2} and the analog of the calculation for \eqref{ibp1}.
The limiting expression for the diffusion term is immediate by \eqref{singularcostlimit}. Finally, using
these two limit calculations and the $N$--Nash system \eqref{dysonnash} itself, we have 
\pagebreak
\begin{align}
 \mu_{\beta,V}^N &  \left [ \sum_{k:k\neq i} \partial_k v^{N,k}_\beta \partial_k v^{N,i}_\beta \right ] \nonumber \\ 
 &=  \mu_{\beta,V}^N \left [\frac{\sigma^2}{2(N-1)} \Delta_\x v^{N,i}_\beta  -  \frac{1}{2} (\partial_i v^{N,i}_\beta)^2 
+  \frac{(x^i)^2}{8} + \frac{\beta}{2} \left ( \frac{3}{4} \beta - \sigma^2 \right ) \frac{(h_2 *\mu^{N,i}_\X)(x^i)}{N-1}   \right ] - \lambda_\beta^{N,i} \nonumber \\
& \overset{N \to \infty}{\to} \frac{\pi^2 \beta^2 \sigma^2}{8(\beta - \sigma^2)} m_{\beta,V}(\gamma^q)^2 
-  \frac{1}{2} |\partial_xU_\beta(\gamma^q,\mu_{\beta,V})|^2 
+  \frac{(\gamma^q)^2}{8} + \frac{\beta}{2} \left ( \frac{3}{4} \beta - \sigma^2 \right ) 
\cdot \frac{\pi^2 \beta}{3(\beta-\sigma^2)} m_{\beta,V}(\gamma^q)^2 - \frac{\beta}{4} \nonumber \\
& =\frac{\pi^2 \beta^2 \sigma^2}{12(\beta - \sigma^2)} m_{\beta,V}(\gamma^q)^2 
+ \left (\frac{(\gamma^q)^2}{8} + \frac{\pi^2 \beta^2}{8} m_{\beta,V}(\gamma^q)^2 - \frac{1}{2} |\partial_xU_\beta(\gamma^q,\mu_{\beta,V})|^2 - \frac{\beta}{4}  \right ). \nonumber
\end{align}
Since compactness of support implies $\mu_{\beta,V}$ has finite second moment,
we can use Lemma \ref{masterequation}, which confirms the pair $(U_\beta(x,\mu), \frac{\beta}{4})$ 
satisfies the master equation \eqref{dysonmaster} on $\R \times \mc{P}^2_2(\R)$. Hence, the terms in the parentheses 
become the desired expression, completing the proof.
\end{proof}
%	\begin{remark} \label{criticalcasefailure}
%		Besides rigorously connecting the $N$--Nash system \eqref{dysonnash} to the Voiculescu--Wigner master equation \eqref{dysonmaster}, 
%		Theorem \ref{mainresult} is nontrivial because it establishes convergence at the location $\gamma^q$. This 
%	fact confirms to some degree the intuition that players in the Dyson game really only care about 
%	the population at their mean field ``location'' as $N \to \infty$.
%	\end{remark}

		We view the convergence in \eqref{controlcontributiongeneral} 
		as a localized version of the heuristic limit (1.4.4) of Biane--Speicher \cite{pbrs1} 
		for relative free Fisher information \eqref{relativeinfo} (cf. Remark \ref{voicconsistency}). 
		Their limit (1.4.4) under the (Nash--optimal) Gaussian ensemble $\mu_\beta^N$ of \eqref{invariantmeasure} 
		is actually not difficult to compute directly using the argument of Remark \ref{averageindex},
	but we provide a more general computation as the first item of the following corollary of our work, which 
	  confirms the analogous convergence of equations for the auxiliary global problem associated to the open loop model.

\begin{cor} \label{bianespeichercomputation} 
	Assume $\beta > \sigma^2>0$ and recall the definitions \eqref{invariantmeasuregeneral} of $\mu_{\beta,V}^N$ 
	and \eqref{eulerlagrange} of $\mu_{\beta,V}$. Let $\bar{X} \sim \mu_{\beta,V}$. Then we have the following asymptotic contributions:
\begin{enumerate} 
	\item The drift term contributes
		$$
\begin{aligned}
		\lim_{N \to \infty} \mu_{\beta,V}^N \left [  \frac{1}{N}\Vert \nabla_\bx W_\beta \Vert^2  \right ] &
		  =\frac{\pi^2 \beta^2 \sigma^2}{12(\beta-\sigma^2)} \mb{E}m_{\beta,V}(\bar{X})^2 + \frac{1}{2} \int_\R | \partial_\mu \mc{U}_\beta(\mu_{\beta,V})(x)|^2\mu_{\beta,V}(dx)
\end{aligned}
		$$	
	
	\item The diffusion term contributes
		$$
		\lim_{N \to \infty} \mu_{\beta,V}^N \left [ - \frac{\sigma^2}{2(N-1)} \cdot \frac{1}{N} \Delta_\x W_\beta \right ]
		=    - \frac{\pi^2 \beta^2 \sigma^2}{12(\beta-\sigma^2)} \mb{E}  m_{\beta,V}(\bar{X})^2.
		$$
\end{enumerate}
	Consequently, the ergodic HJB equation \eqref{opendysonnash} with $(W_\beta(\x),\lambda_\beta^N)$
	associated to the open loop game converges against $(\mu_{\beta,V}^N)_{N \geq 2}$ to 
	the Hamilton--Jacobi equation \eqref{openloopmasterequation} with $(\mc{U}_\beta(\mu),\frac{\beta}{8})$ at $\mu_{\beta,V} \in \mc{P}^3_2(\R)$. 
\end{cor}

\section{The Coulomb Game} \label{coulombgame}

%To complete our case-study, 
%we now endeavor to generalize the construction of the Dyson game 
%to yield planar Coulomb dynamics in equilibrium. 
%It turns out such behavior no longer 
%coincides with dynamics of random eigenvalues in the complex plane even
%for special choices of parameters 
%(i.e., the Dyson eigenvalue dynamics and Coulomb gas dynamics only happen to coincide for dimension $d=1$), 
%so we believe they are more aptly called \emph{Coulomb Games}
%and will even consider the Dyson game to be its one dimensional version. 
	In this section, we do not pursue the details of verification theorems for the implicit $N$ player or mean field game formulations of the Coulomb game. 
	Such results were already exemplified in detail for the one dimensional case
	and it is clear how such statements would generalize.
Our first lemma explains how the middle term of a player's cost \eqref{2dstatecost}, which incentivizes collinearity, arises.
Recall the definitions \eqref{coulombtransform}, \eqref{coulombtransformproduct} of $\mathcal{H}\mu(z)$
and $\mathcal{H}[\mu\mathcal{H}\mu](z)$, respectively.
\begin{lem}
	\label{circumcircle}
For any $z,w,u \in \RR^2$, we have the identity
\begin{equation}
	\label{circumcircleidentity}
	\begin{aligned}
	\left \langle \frac{z-w}{|z-w|^2} , \frac{z-u}{|z-u|^2} \right \rangle -  
	\left \langle \frac{z-w}{|z-w|^2} , \frac{w-u}{|w-u|^2} \right \rangle - \left \langle \frac{z-u}{|z-u|^2} , \frac{u-w}{|u-w|^2} \right \rangle =   \frac{2}{D^2(z-w,z-u)}.
	\end{aligned}
\end{equation}
In particular, for any $\mu \in \mc{P}^p_2(\R^2)$, $p >2$, 
$$
	| \mathcal{H}\mu(z) |^2 - 2 \mathcal{H}[\mu\mathcal{H}\mu](z)
= \int \int_{ \substack{w,u \in \RR^2 \\ w \neq u}}\frac{2\mu(dw) \mu(du)}{D^2(z-w,z-u)}.
$$
%while for $\mu = \mu_\bz^{N,i}$, $\bz \in \mc{D}^N$, one must include diagonal terms:
%$$
%| \mathcal{H}\mu^{N,i}_\bz(z^i) |^2 - 2 \mathcal{H}[\mu^{N,i}_\bz\mathcal{H}\mu^{N,i}_\bz](z^i)
%= \int \int_{ \substack{w,u \in \RR^2 \\ w \neq u}}\frac{2 \mu^{N,i}_\bz(dw) \mu^{N,i}_\bz(du)}{D^2(z^i-w,z^i-u)} + 
%\frac{1}{(N-1)^2} \sum_{k:k\neq i} \frac{1}{|z^i-z^k|^2}.
%$$
\end{lem}
\begin{proof}
The left hand side of \eqref{circumcircleidentity} can be written
	$$
	\frac{1}{|z-w|^2 |z-u|^2|w-u|^2} \cdot \left [ |w-u|^2 \langle z-w, z-u \rangle 
	 - |z-u|^2 \langle z-w, w-u \rangle - |z-w|^2 \langle z-u, u-w \rangle  \right ].
	$$
	Expanding the factor $|w-u|^2 =|(w-z)+(z-u)|^2$ in the first term inside the
	 square brackets gives
	$$
	\frac{2\left [ |z-w|^2|z-u|^2 - \langle z-w , z-u \rangle^2 \right ]}{|z-w|^2 |z-u|^2|w-u|^2}
	 = \frac{2 \sin^2 \theta_{z-w,z-u}}{|w-u|^2} = \frac{2}{D^2(z-w,z-u)},
	$$
	where $\theta_{z-w,z-u}$ is the angle between the vectors $z-w$ and $z-u$. For the last claim,
	note that for $\mu\in \mathcal{P}_2^p(\R^2)$, $p>2$, we have that $\mathcal{H}\mu$ is uniformly bounded:
\begin{align*}
	|\mathcal{H}\mu(z)|\leq \left(\int_{\R^2: |w|<1}|w|^{\frac{-p}{p-1}}dw\right)^{\frac{p-1}{p}}  \|m\|_{L^p(\R^2)}+\Big|\int_{\R^2: |z|\geq 1}\mu(dy)\Big|.
\end{align*}
	This completes the proof.
\end{proof}

\subsection*{Closed loop model}

\begin{defn} \label{2dclosedadmissibility}
A feedback profile $\boldsymbol{\phi} : (\R^2)^N \to (\R^2)^N$ is \emph{admissible} if for every $\bz_0 \in \mc{D}^N$, 
there exists a unique strong solution $(\bZ_t)_{t \geq 0} = (\bZ^{\boldsymbol{\phi}}_t)_{t \geq 0}$ to the stochastic differential equation 
		\begin{equation} \label{2dclosedNdynamics}
			d\bZ_t = \boldsymbol{\phi}(\bZ_t) dt + \frac{\sigma}{\sqrt{N-1}} d\mathbf{B}_t, \ \ \bZ_0 = \bz_0, \ \ \sigma \geq 0 
		\end{equation} 
		satisfying the integrability condition
		\begin{equation} \label{closedadmissibleintegrability}
			\mb{E} \int_0^T \left [ \Vert \boldsymbol{\alpha}_t \Vert^2
				+ \Vert \bZ_t \Vert^2 + \sum_{1 \leq k < \ell \leq N} \frac{1}{|Z^\ell_t - Z^k_t|^2} 
			\right ]
			dt < \infty, \ \text{for all} \ \  T > 0.
		\end{equation} 
We denote the class of such admissible strategies by $\mc{A}^{(N)}$.  
\end{defn}

\begin{defn}
	The $N$ player dynamic game associated with
	the player state costs \eqref{2dstatecost} in the ergodic regime
	as in Section \ref{dysongame} is called
	the \emph{Coulomb game}.
If in addition we impose the notion of admissibility Definition \ref{2dclosedadmissibility}, we call it the \emph{closed loop (or Markovian) model for the Coulomb game}.
\end{defn}

\subsection*{Open loop model}

\begin{defn} \label{2dopenadmissibility}
A profile $(\boldsymbol{\alpha}_t)_{t \geq 0} =((\alpha^1_t, \ldots, \alpha^N_t))_{t \geq 0}$ of $\R^2$--valued processes is \emph{admissible} if
it is $\mb{F}$--progressively measurable and for every $\bz_0 \in \mc{D}^N$, 
the process $(\bZ_t)_{t \geq 0} = (\bZ^{\boldsymbol{\alpha}}_t)_{t \geq 0}$ defined by 
		\begin{equation} \label{2dopenNdynamics}
			d\bZ_t = \boldsymbol{\alpha}_t dt + \frac{\sigma}{\sqrt{N-1}} d\mathbf{B}_t, \ \ \bZ_0 = \bz_0, \ \ \sigma \geq 0 
		\end{equation} 
		satisfies the integrability condition
		\begin{equation} \label{openadmissibleintegrability}
			\mb{E} \int_0^T \left [ \Vert \boldsymbol{\alpha}_t \Vert^2
				+ \Vert \bZ_t \Vert^2 + \sum_{1 \leq k < \ell \leq N} \frac{1}{|Z^\ell_t - Z^k_t|^2} 
			\right ]
			dt < \infty, \ \text{for all} \ \  T > 0.
		\end{equation} 
We denote the class of such admissible strategies by $\mb{A}^{(N)}$.  
\end{defn}
\begin{defn}
	The $N$ player game dynamic game associated with
	the player state costs \eqref{2dstatecost} in the ergodic regime
	as in Section \ref{dysongame} and with the notion of admissibility Definition \ref{2dopenadmissibility} is called
	the \emph{open loop model for the Coulomb game}.
\end{defn}
\subsection*{Potential structure}
By a slight abuse, we continue to write the player cost
as $J^{N,i}(\psi |\z_0, \boldsymbol{\phi}^{-i})$ for the closed loop model or 
as $J^{N,i}(\boldsymbol{\alpha} |\z_0, \boldsymbol{\alpha}^{-i})$ for the open loop model, 
where we use the state cost $F^{N,i}(\bz)$ of \eqref{2dstatecost}.

\begin{lem} \label{2dpotentialgame}
Fix $C>0$ and define the global state cost
\begin{equation}
	\label{2dglobalcost}
	F^N(\bz) := \frac{\Vert \bz \Vert^2}{8} + C  \sum_{i=1}^N \left | (h_1 * \mu^{N,i}_\bz)(z^i) \right |^2 = \frac{\Vert \bz \Vert^2}{8} + \frac{C}{(N-1)^2}  \sum_{i=1}^N \left | \sum_{k : k\neq i} \frac{z^i - z^k}{|z^i-z^k|^2} \right |^2, \ \ \bz \in \mc{D}^N.
\end{equation}
Consider the optimization problem of minimizing
\begin{equation}
	\label{2dopencostfunctional}
J^N(\boldsymbol{\alpha} | \bz_0) :=  \limsup_{T \to \infty} \frac{1}{T} \mb{E} \int_0^T  \left [  \frac{1}{2} \Vert \boldsymbol{\alpha}_t \Vert^2 + F^N(\bZ_t) \right ] dt, 
\end{equation}
over $\boldsymbol{\alpha} \in \mb{A}^N$, subject to $(\bZ_t)_{t \geq 0} = (\bZ^{\boldsymbol{\alpha}}_t)_{t \geq 0}$, $\bZ_0 = \bz_0$, of \eqref{2dopenNdynamics}. 
Suppose that in definitions \eqref{2dstatecost}, \eqref{2dglobalcost} of $F^{N,i}(\bz)$, $F^N(\bz)$
we take $C=C_1 = C_2/2$.
Then the open loop model for the Coulomb game is a \emph{potential game} in the following sense:
For any profile $(\boldsymbol{\alpha}_t)_{t \geq 0} \in \mb{A}^{(N)}$ such that the limit in \eqref{2dopencostfunctional} exists, and 
for any deviation $(\eta_t)_{t \geq 0} \in \mb{A}^i(\boldsymbol{\alpha}^{-i})$, $1 \leq i \leq N$, such that 
the limit in \eqref{2dopencostfunctional} exists, we have
\begin{equation} \label{2dpotentialgamecondition}
	J^N((\eta,\boldsymbol{\alpha}^{-i}) | \bz_0 ) - J^N(\boldsymbol{\alpha}| \bz_0 ) = 
	J^{N,i}(\eta \  |\bz_0,\boldsymbol{\alpha}^{-i}) - J^{N,i}( \alpha^i \  |\bz_0,\boldsymbol{\alpha}^{-i}).
\end{equation}
\end{lem}
\begin{proof}
The proof is the same as for Lemma \ref{potentialgame} except 
we must also make use of Lemma \ref{circumcircle}. We thus rewrite the interaction term of the global state cost as
$$
\begin{aligned}
 \sum_{i=1}^N\sum_{j:j\not=i}\sum_{k:k\not=i} \left\langle\frac{z^i - z^j}{|z^i-z^j|^2},\frac{z^i - z^k}{|z^i-z^k|^2}\right\rangle	  =&\ \sum_{i=1}^N\sum_{j:j\not=i}\sum_{k:k\not=i,j} \left\langle\frac{z^i - z^j}{|z^i-z^j|^2},\frac{z^i - z^k}{|z^i-z^k|^2}\right\rangle+\sum_{i=1}^N\sum_{j:j\not=i}\frac{1}{|z^i-z^j|^2}\\
	=&\ \frac{1}{3}\sum_{i=1}^N\sum_{j:j\not=i}\sum_{k:k\not=i,j}\frac{2}{D^2(z^i-z^j,z^i-z^k)}+\sum_{i=1}^N\sum_{j:j\not=i}\frac{1}{|z^i-z^j|^2},
\end{aligned}
$$
where for the last equality we have used Lemma 9.1 along with cyclically permuting the $i, j, k$ indices.  We then see that the factor of $\frac{1}{3}$ corrects the triple counting of the reciprocal squared diameters of $F^{N,i}(\bz)$, so that condition \eqref{2dpotentialgamecondition} holds provided
$C=C_1 = C_2/2$. As before in Section \ref{dysongame}, this exactly meets Definition 2.23 in \cite{bible1} for potential game.
\end{proof}
\begin{rem}
The proof shows that instead of \eqref{2dglobalcost}, we might as well 
be working with 
\begin{equation} \label{2dglobalcostbetter}
	F^N(\bz) := \frac{\Vert \bz \Vert^2}{8} + 
	\frac{C_1}{3}\sum_{i=1}^N \int \int_{ \substack{w,u \in \RR^2 \\ w \neq u}} \frac{2}{D^2(z^i-w,z^i-u)} \mu^{N,i}_\bz(dw) \mu^{N,i}_\bz(du)+ \frac{C_2}{2}\sum_{i=1}^N \frac{(h_2 * \mu^{N,i}_\bz)(z^i)}{N-1}.
\end{equation}
This more general form is in fact required to solve the HJB \eqref{globalHJB} in general dimension $d \geq 1$.
\end{rem}
As in Section \ref{dysongame}, a minimizer of \eqref{2dopencostfunctional} can be achieved
by a strategy in closed loop feedback form based on a solution to the ergodic HJB equation
\begin{equation} \label{globalHJB}
- \frac{\sigma^2}{2(N-1)} \cdot \frac{1}{N} \Delta_\bz W(\bz) + \frac{1}{2N} \Vert \nabla_\bz W(\bz)\Vert^2 =
\frac{1}{N} F^N(\bz) - \lambda, \ \ \bz \in \mc{D}^N.
\end{equation}

\subsection*{Solving the $N$ Nash system and ergodic HJB equation}

Compare the next result with Lemma \ref{Nsubsolution} for the one dimensional case.
\begin{lem}
	\label{2dNnashprop}
Let $\beta \in \RR$, $C_1=\beta^2/8$, and $C_2 = 3\beta^2/8$. 
Then the ergodic value pairs $(v^{N,i}_\beta(\bz), \lambda^{N,i}_\beta)_{i=1}^N$ of \eqref{2dNnashsoln}
form a classical solution to the $N$--Nash system \eqref{2dNnash} on $\mc{D}^N$.
\end{lem}
\begin{proof}
Just as in the one dimensional case, we can compute directly
$$
\nabla_{z^i} v^{N,i}(\bz)
= \frac{z^i}{2} - \frac{\beta}{2} (h_1 * \mu^{N,i}_\bz)(z^i), \ \ \ 
\nabla_{z^k} v^{N,i}(\bz) = \frac{\beta}{2(N-1)} \frac{z^i-z^k}{|z^i-z^k|^2}, \ \ k \neq i.
$$
However, quite differently from the one dimensional case, 
the function $h_0(z) = \log | z |$ satisfies $\Delta_z h_0(z) = 2 \pi \delta(z)$.
Hence, the Laplacian has the rather easy form
$$
- \frac{\sigma^2}{2(N-1)} \Delta_\bz v^{N,i}(\bz) = - \frac{\sigma^2}{2(N-1)} ,
$$
as long as $z^i \neq z^j$ for all $i \neq j$, i.e., $\bz \in \mc{D}^N$.
Write $\tilde{z}^{ik}:= \frac{z^i - z^k}{|z^i - z^k|^2}$. Then we have
$$
\frac{1}{2} | \nabla_{z^i} v^{N,i}(\bz) |^2 = \frac{|z^i|^2}{8} 
+ \frac{\beta^2}{8} \frac{(h_2 * \mu^{N,i}_\bz)(z^i)}{N-1}
+ \frac{\beta^2}{8(N-1)^2} \cdot \sum_{k \neq i} \sum_{\ell \neq i,k} \langle \tilde{z}^{ik}, \tilde{z}^{i\ell} \rangle 
- \frac{\beta}{4(N-1)} \sum_{k:k \neq i} \langle z^i, \tilde{z}^{ik} \rangle,
$$
and the interaction term becomes
$$
\sum_{k: k \neq i} \langle \nabla_{z^k} v^{N,k}(\bz) , \nabla_{z^k} v^{N,i} (\bz) \rangle 
= \frac{\beta^2}{4} \frac{(h_2 * \mu^{N,i}_\bz)(z^i)}{N-1} 
 - \frac{\beta^2}{4(N-1)^2} \cdot \sum_{k \neq i} \sum_{\ell \neq i,k} \langle \tilde{z}^{ik}, \tilde{z}^{k\ell} \rangle + \frac{\beta}{4(N-1)} \sum_{k:k \neq i} \langle z^k, \tilde{z}^{ik} \rangle.
$$
The final terms of the previous two equations come together to give the constant $-\beta/4$.
Combining all remaining terms along with Lemma \ref{circumcircle} 
then completes the proof.
\end{proof}

Compare the next result with Lemma 
\ref{openloopehjbtheorem} for the one dimensional
case.
\begin{lem} \label{2dehjb}
Fix $\beta \in \RR$ and let $C=C_1 = \beta^2/8$ and $C_2 = \beta^2/4$ in the global cost \eqref{2dglobalcost} or \eqref{2dglobalcostbetter}. Then the pair
$$
\begin{cases}
	W_\beta(\bz) := \frac{\Vert \bz \Vert^2}{4} - \frac{\beta}{2(N-1)} \sum_{1 \leq k < \ell \leq N} \log |z^\ell- z^k|, & \bz \in \calD^N \\
	\lambda^N_\beta :=\frac{\beta}{8} + \frac{\sigma^2}{2(N-1)}
\end{cases}
$$
forms a classical solution to the ergodic HJB \eqref{globalHJB} on $\bz \in \calD^N$.
\end{lem}
\begin{proof}
Similarly as above, we can compute directly
$$
\nabla_{z^i} W_\beta(\bz) 
= \frac{z^i}{2} - \frac{\beta}{2(N-1)} \sum_{k:k \neq i} \frac{z^i-z^k}{|z^i-z^k|^2},
 \ \ \ \Delta_\bz W_\beta(\bz) = N,
$$
as well as 
$$
\frac{1}{2} \Vert \nabla_\bz W_\beta(\bz)\Vert^2 = \sum_{i=1}^N \frac{|z^i|^2}{8} + \frac{\beta^2}{8} \sum_{i=1}^N \left | \frac{1}{N-1} \sum_{k:k\neq i} \frac{z^i - z^k}{|z^i-z^k|^2}  \right |^2
 - \frac{\beta}{8} \cdot N.
$$
Putting everything together completes the proof.
\end{proof}
\begin{rem}
Considering Lemmas \ref{2dpotentialgame},
\ref{2dNnashprop}, and \ref{2dehjb},
we notice the choice of $C_1$ is the same for either model, but the choice $C_2 = \beta^2/4$ in the open loop case
disagrees with the larger choice $C_2 = 3\beta^2/8$ in the closed loop case. 
In particular, unlike in one dimension, \emph{the closed and open loop models are not simultaneously explicitly solvable in higher dimensions.}
But we saw that
players in the closed loop game on the line will use a lower repulsion in equilibrium, 
and there is a similar interpretation in the plane:
players of the closed loop game will adopt the repulsion $\beta$
despite facing the higher singular cost coefficient $C_2 = 3\beta^2/8$.

%A key difference from the one dimensional case, though, is that the influence of $C_2$ is expected to disappear in the limit: the typical distance between (mean field scaled) particles with Coulomb interactions in the plane will now be of order $1/\sqrt{N}$, much larger than the $1/N$ length-scale witnessed in one dimension.
%Hence, given a common value $C_1 = \beta^2/8>0$,
%planar Coulomb dynamics with coupling parameter $\beta>0$ will be an $\epsilon_N$-Nash equilibrium regardless 
%of what $C_2$ is. 
\end{rem}

\subsection*{Mean field equations}
In contrast to the one dimensional case, 
the two dimensional state costs $F^N(\bz)$ and $F^{N,i}(\bz)$
can safely be replaced with their naive mean field analogs upon dropping
the reciprocal squared gaps cost term.

\begin{lem} \label{2dmasterequationcalc}
Fix $\beta>0$. Define
\begin{equation} \label{2dmeanfieldpotential}
\mc{U}_\beta(\mu) := \int_{\R^2} \frac{|z|^2}{4} \mu(dz) - \frac{\beta}{4} \int_{\R^2} \int_{\R^2} \log |z - w| \mu(dz) \mu(dw), \ \ 
\mu \in \calP_2(\RR^2).
\end{equation}
Then for $\mu \in \mathcal{P}_2^p(\R^2)$, $p>2$, the pair $(\mc{U}_\beta(\mu), \frac{\beta}{8})$ satisfies the ergodic Hamilton--Jacobi equation 
\begin{equation} \label{2dopenloopmasterequation}
\begin{aligned}
	\frac{1}{2} \int_\R |\partial_\mu \mc{U}(\mu)(z)|^2 \mu(dz) 
		& = \int_{\R^2} \left ( \frac{|z|^2}{8} +  \frac{\beta^2}{8}  \left | \int_{\R^2} \frac{z-w}{|z-w|^2} \mu(dw) \right |^2  \right ) \mu(dz) - \lambda \\
		& = \int_{\R^2} \frac{|z|^2}{8} \mu(dz) +  \frac{\beta^2}{24} \int_{\R^2}  \left ( \int \int_{ \substack{w,u \in \RR^2 \\ w \neq u}}\frac{2}{D^2(z-w,z-u)} \mu(dw) \mu(du) \right ) \mu(dz) - \lambda,
\end{aligned}
\end{equation}
where we recall $D(\xi,\eta)$ is the diameter of the circumcircle of the triangle determined 
by $\xi,\eta,$ and $(0,0)$ in $\RR^2$. 
Similarly, recall the definition \eqref{2dmasterfield}
of $U_\beta(z,\mu)$.
Then for $\mu \in \mathcal{P}_2^p(\R^2)$, $p>2$, the pair $(U_\beta(z,\mu), \frac{\beta}{4})$ forms a solution to the Coulomb master equation \eqref{2dmasterequation}.
\end{lem}
\begin{proof}
The first equality in \eqref{2dopenloopmasterequation} is straightforward;
%  as
% $$
% 	\partial_\mu \mathcal{U}(\mu)(z)=\frac{z}{2}-\frac{\beta}{2}\mathcal{H}\mu(z),
% $$
% and
% \begin{align*}
% \int_{\R^2}\int_{\R^2}\frac{z\cdot (z-w)}{|z-w|^2}\mu(dw)\mu(dz)=&\ \int_{\R^2}\int_{\R^2}\Big[1+\frac{w\cdot(z-w)}{|z-w|^2}\Big]\mu(dw)\mu(dz),
% \end{align*}
% which yields
% \begin{align*}
% \int_{\R^2}\int_{\R^2}\frac{z\cdot (z-w)}{|z-w|^2}\mu(dw)\mu(dz)=\frac{1}{2}.
% \end{align*}
for the second equality, we have
$$
\begin{aligned}
\int_{\RR^2} |\mc{H}\mu(z)|^2 \mu(dz) & = 
	\int_{\R^2}\int_{\R^2} \int_{\R^2} \left \langle \frac{z-w}{|z-w|^2}, \frac{z-u}{|z-u|^2} \right \rangle \mu(dw) \mu(du) \mu(dz) \\
	& = \frac{1}{3} \int_{\R^2}\int_{\R^2} \int_{\R^2}  \left [ \left \langle \frac{z-w}{|z-w|^2}, \frac{z-u}{|z-u|^2} \right \rangle  - 2 \left \langle \frac{z-w}{|z-w|^2}, \frac{w-u}{|w-u|^2} \right \rangle \right ] \mu(dw) \mu(du) \mu(dz) \\
		& = \frac{1}{3} \int_{\R^2} \left [ | \mc{H}\mu(z)|^2 - 2 \mc{H}[\mu \mc{H}\mu](z)  \right ]\mu(dz). \\
\end{aligned}
$$
Hence, we arrive at
$$
\frac{1}{2} \int_\R |\partial_\mu \mc{U}(\mu)(z)|^2 \mu(dz)  = 
\int_{\R^2} \frac{|z|^2}{8} \mu(dz) +  \frac{\beta^2}{24} \int_{\R^2}  \left [ | \mc{H}\mu(z)|^2 - 2 \mc{H}[\mu \mc{H}\mu](z)  \right ] \mu(dz) - \lambda.
$$
But by Lemma \ref{circumcircle} we know that 
$$
\begin{aligned}
	| \mathcal{H}\mu(z) |^2 - 2 \mathcal{H}[\mu\mathcal{H}\mu](z)
= \int \int_{ \substack{w,u \in \RR^2 \\ w \neq u}}\frac{2}{D^2(z-w,z-u)} \mu(dw) \mu(du),
\end{aligned}
$$
which completes the proof of \eqref{2dopenloopmasterequation}. The proof of 
\eqref{2dmasterequation} is similar.
\end{proof}

\subsection*{Recovering the master equation from the $N$ Nash system} \label{2dconvergenceofeqns}

\begin{proof}[Proof of Theorem \ref{big2dtheorem}]
As in the one dimensional case,
we know that $\max_{1 \leq i \leq N} |Z^i|$
is exponentially tight at speed $N$ by Theorem 1.12
of Chafa\"{i}-Hardy-Ma\"{i}da \cite{chafai2018concentration}. 
This again implies the $p$--Wasserstein convergence $d_p(\mu_\bZ^N,\mu_{\beta,V}) \overset{N \to \infty}{\to} 0$ almost surely.
Note that the Euler-Lagrange identity \eqref{2deulerlagrange} and $\Delta_z \log |z| = 2 \pi \delta_0(z)$ together imply that the density $m_{\beta,V}(z)$ 
is proportional to $\Delta_z V(z)$ on the support of $\mu_{\beta,V}$; 
hence, $m_{\beta,V}(z)$ is bounded with compact support. 
Again by the Euler-Lagrange identity \eqref{2deulerlagrange}, 
we have that
$$
\begin{aligned}
|\nabla V(\gamma)|^2 & = \beta^2 |\mc{H}m_{\beta,V}(\gamma)|^2 \\
\int_{\R^2} \left \langle \nabla V(\gamma) - \nabla V(w) ,  \frac{\gamma - w}{|\gamma - w|^2} \right \rangle m_{\beta,V}(w) dw  & = \beta \left [ |\mc{H}m_{\beta,V}(\gamma)|^2 - \mc{H}[m_{\beta,V} \mc{H}m_{\beta,V}](\gamma)   \right ].
\end{aligned}
$$
Then we calculate
$$
\sum_{k:k \neq i} \frac{z^i - z^k}{|z^i - z^k|^2} \prod_{1 \leq k < \ell \leq N} |z^\ell- z^k|^{\beta/\sigma^2}
 = \frac{\sigma^2}{\beta} \nabla_{z^i} \left [ \prod_{1 \leq k < \ell \leq N} |z^\ell- z^k|^{\beta/\sigma^2} \right ]
$$
$$
 \left ( \frac{\sigma^2}{\beta}  \right )^2 \Delta_{z^i} \left [ \prod_{1 \leq k < \ell \leq N} |z^\ell- z^k|^{\beta/\sigma^2} \right ]
 =
\left | \sum_{k:k \neq i} \frac{z^i - z^k}{|z^i - z^k|^2} \right |^2 \prod_{1 \leq k < \ell \leq N} |z^\ell- z^k|^{\beta/\sigma^2}.
$$
%and finally ({\color{red} we probably dont need these calculations)
%\[
%\begin{aligned}
%\left ( \frac{\sigma^2}{\beta}  \right )^2 \sum_{j : j \neq i} \nabla_{z^j} \circ \nabla_{z^i} & \left [ \prod_{1 \leq k < \ell \leq N} |z^\ell- z^k|^{\beta/\sigma^2} \right ]
%= \\
%&
%\left \langle \sum_{k:k \neq i} \frac{z^i - z^k}{|z^i - z^k|^2}, \sum_{j : j \neq i} \sum_{\ell:\ell \neq j} \frac{z^j - z^\ell}{|z^j - z^\ell|^2} \right \rangle \prod_{1 \leq k < \ell \leq N} |z^\ell- z^k|^{\beta/\sigma^2} \\
%& - 
%\left \langle \sum_{k:k \neq i} \frac{z^i - z^k}{|z^i - z^k|^2}, \sum_{\ell:\ell \neq i} \frac{z^i - z^\ell}{|z^i - z^\ell|^2} \right \rangle  \prod_{1 \leq k < \ell \leq N} |z^\ell- z^k|^{\beta/\sigma^2},
%\end{aligned}
%\]
%where the last equality uses the fact that
%\[
% 0 = \sum_{j =1}^N \sum_{\ell: \ell \neq j}  \frac{z^j - z^\ell}{|z^j - z^\ell|^2}  = \sum_{j : j \neq i} \sum_{\ell:\ell \neq j} \frac{z^j - z^\ell}{|z^j - z^\ell|^2}
% + \sum_{\ell:\ell \neq i} \frac{z^i - z^\ell}{|z^i - z^\ell|^2}. 
%\]
%}
Using these, we can compute the limits
\begin{equation} \label{2dcomputethelimits}	
\begin{aligned}
\mathbb{E}  \langle \nabla V(Z^i) , (h_1 * \mu^{N,i}_\bZ )(Z^i) \rangle
 & = \frac{\mb{E} \left | \nabla V(Z^i) \right |^2}{\beta}  - \frac{1}{N-1} \frac{\sigma^2}{\beta}
 \mb{E} \Delta_{z^i} V(Z^i) \to  \frac{\left | \nabla V(\gamma) \right |^2}{\beta}
  = \beta |\mc{H}m_{\beta,V}(\gamma)|^2 \\
\lim_{N \to \infty} \mathbb{E} \left | (h_1 * \mu^{N,i}_\bZ)(Z^i) \right |^2 & = \lim_{N \to \infty} \frac{\mathbb{E} | \nabla V(Z^i) |^2}{\beta^2}
= \frac{ | \nabla V(\gamma) |^2}{\beta^2} = |\mc{H}m_{\beta,V}(\gamma)|^2,
\end{aligned}
\end{equation}
where the convergences again are assured by polynomial growth of $V$.
Now, we note that the collection of functions
$$
v^{N,i}_{\beta,V}(\bz) := \frac{V(z^i)}{2} - \frac{\beta}{2} (h_0 * \mu^{N,i}_\bz)(z^i), \ \ 1 \leq i \leq N,
$$
satisfies the system of equations
\begin{equation}
\label{2dgeneralNnash}
\begin{aligned}
& - \frac{\sigma^2}{2(N-1)} \Delta_\bz v^{N,i}+ \sum_{k: k \neq i} \langle \nabla_{z^k} v^{N,k} , \nabla_{z^k} v^{N,i} \rangle 
+ \frac{1}{2} | \nabla_{z^i} v^{N,i}(\bz) |^2  \\
& = \frac{| \nabla_{z^i} V(z^i) |^2}{8} + \frac{\beta^2}{4} \int \int_{\substack{w,u \in \RR^2 \\ w \neq u}} \frac{\mu^{N,i}_\bz(dw) \mu^{N,i}_\bz(du)}{D^2(z^i-w,z^i-u)} 
  + \frac{3\beta^2}{8} \frac{(h_2 * \mu^{N,i}_\bz)(z^i)}{N-1} \\
  & - \frac{\beta}{4} \int_{\R^2} \left \langle \nabla V(z^i) - \nabla V(w) ,  \frac{z^i - w}{|z^i - w|^2}  \right \rangle \mu^{N,i}_\bz(dw)  - \frac{\sigma^2}{4(N-1)} \Delta_{z^i} V(z^i). 
\end{aligned}
\end{equation}
If we further write
$$
\mathcal{L}_{\beta,V} f := \frac{\sigma^2}{2(N-1)} \Delta_\bz f - \sum_{k=1}^N \langle \nabla_{z^k} v^{N,k}_{\beta,V} , \nabla_{z^k} f \rangle,
$$
then just as in the one dimensional case we have the invariance identity $\mu_{\beta,V}^N \left [ \mathcal{L}_{\beta,V} (v^{N,i}_{\beta,V} )  \right ] = 0$.
Now, letting $N \to \infty$ in this invariance identity and using the computations \eqref{2dcomputethelimits}, \eqref{2dgeneralNnash}, we have
%$$ old calcs
%\begin{aligned}
%& \frac{\beta}{4} \int_{\R^2} \left \langle \nabla V(\gamma) - \nabla V(w) ,  \frac{\gamma - w}{|\gamma - w|^2} \right \rangle m_V(w) dw = -  \lim_{N \to \infty} \frac{\beta}{4} \mathbb{E}  \langle \nabla V(Z^i) , (h_1 * \mu^{N,i}_\ZZ )(Z^i) \rangle   \\
% & +\lim_{N \to \infty} \left [ \frac{ \mathbb{E} | \nabla_{z^i} V(Z^i) |^2}{4} 
% + \frac{\beta^2}{4} \mathbb{E}  \left [ \left | (h_1 * \mu^{N,i}_\ZZ)(Z^i)  \right |^2 - 
%\frac{1}{(N-1)^2}  \sum_{k:k \neq i} \sum_{\ell : \ell \neq i, k} \left \langle h_1(Z^i - Z^k) ,  h_1(Z^k - Z^\ell) \right \rangle  \right ]  \right ] \\
%& = \frac{|\nabla V(\gamma)|^2}{4} + \frac{\beta^2}{4} \cdot \frac{|\nabla V(\gamma)|^2}{\beta^2}  - \frac{\beta}{4} \cdot   \frac{|\nabla V(\gamma)|^2}{\beta}  - \lim_{N \to \infty} \frac{\beta^2}{4(N-1)^2}   \mathbb{E}  \sum_{k:k \neq i} \sum_{\ell : \ell \neq i, k} \left \langle h_1(Z^i - Z^k) ,  h_1(Z^k - Z^\ell) \right \rangle  \\
%& = \frac{|\nabla V(\gamma)|^2}{4} - \lim_{N \to \infty} \frac{\beta^2}{4(N-1)^2}   \mathbb{E}  \sum_{k:k \neq i} \sum_{\ell : \ell \neq i, k} \left \langle h_1(Z^i - Z^k) ,  h_1(Z^k - Z^\ell) \right \rangle
%\end{aligned} 
%$$
%\begin{equation}
%	
\begin{align} \label{2dmaincalculation}
& \frac{\beta^2}{4} \left [ |\mc{H}m_{\beta,V}(\gamma)|^2 - \mc{H}[m_{\beta,V} \mc{H}m_{\beta,V}](\gamma)   \right ] =  \frac{\beta}{4} \int_{\R^2} \left \langle \nabla V(\gamma) - \nabla V(w) ,  \frac{\gamma - w}{|\gamma - w|^2}  \right \rangle m_{\beta,V}(dw) \nonumber \\
 &= \lim_{N \to \infty} \mathbb{E} \left [ \frac{ | \nabla_{z^i} V(Z^i) |^2}{4} + \frac{\beta^2}{8} | (h_1 * \mu^{N,i}_\bZ )(Z^i) |^2 - \frac{\beta}{4}  \langle \nabla V(Z^i) , (h_1 * \mu^{N,i}_\ZZ )(Z^i) \rangle  \right ] \nonumber \\
 & + \lim_{N \to \infty} \mathbb{E} \left [ \frac{3\beta^2}{8} \frac{(h_2 * \mu^{N,i}_\bZ)(Z^i)}{N-1}  + \frac{\beta^2}{4} \int \int_{\substack{w,u \in \RR^2 \\ w \neq u}} \frac{\mu^{N,i}_\bZ(dw) \mu^{N,i}_\bZ(du)}{D^2(Z^i-w,Z^i-u)}   \right ] \nonumber \\
& = \frac{|\nabla V(\gamma)|^2}{4} + \frac{\beta^2}{8} \cdot \frac{|\nabla V(\gamma)|^2}{\beta^2}  - \frac{\beta}{4} \cdot   \frac{|\nabla V(\gamma)|^2}{\beta} \nonumber \\
 & + \lim_{N \to \infty} \mathbb{E} \left [ \frac{3\beta^2}{8} \frac{(h_2 * \mu^{N,i}_\bZ)(Z^i)}{N-1}  + \frac{\beta^2}{4} \int \int_{\substack{w,u \in \RR^2 \\ w \neq u}} \frac{\mu^{N,i}_\bZ(dw) \mu^{N,i}_\bZ(du)}{D^2(Z^i-w,Z^i-u)}    \right ] \nonumber\\
 & = \frac{\beta^2}{8} |\mc{H}m_{\beta,V}(\gamma)|^2 + \lim_{N \to \infty} \frac{\beta^2}{4} \mathbb{E} \left [ \frac{3}{2} \frac{(h_2 * \mu^{N,i}_\bZ)(Z^i)}{N-1}  +  \int \int_{\substack{w,u \in \RR^2 \\ w \neq u}} \frac{\mu^{N,i}_\bZ(dw) \mu^{N,i}_\bZ(du)}{D^2(Z^i-w,Z^i-u)}   \right ]. \\
 \nonumber
\end{align}
%\end{equation}

Up to this point, the proof parallels the one dimensional case closely, but now we need a
new ingredient to complete the proof in the two dimensional case.
Subtracting the first term ``$\frac{\beta^2}{8} |\mc{H}m_{\beta,V}(\gamma)|^2$'' of \eqref{2dmaincalculation} from both sides and recalling Lemma \ref{circumcircle}, 
the lefthandside of \eqref{2dmaincalculation} becomes
$$
\frac{\beta^2}{4} \int \int_{ \substack{w,u \in \RR^2 \\ w \neq u}}  \frac{\mu_{\beta,V}(dw) \mu_{\beta,V}(du)  }{D^2(\gamma-w,\gamma-u)}.
$$
But this implies the estimate
$$
\begin{aligned}
 \limsup_{N \to \infty} & \ \mathbb{E} \int \int_{\substack{w,u \in \RR^2 \\ w \neq u}} \frac{\mu^{N,i}_\bZ(dw) \mu^{N,i}_\bZ(du)}{D^2(Z^i-w,Z^i-u)} \\
 & \leq  
\limsup_{N \to \infty} \mathbb{E} \left [ \frac{3}{2} \frac{(h_2 * \mu^{N,i}_\bZ)(Z^i)}{N-1}
+   \int \int_{\substack{w,u \in \RR^2 \\ w \neq u}} \frac{\mu^{N,i}_\bZ(dw) \mu^{N,i}_\bZ(du)}{D^2(Z^i-w,Z^i-u)}  \right ] \\ & = \int \int_{ \substack{w,u \in \RR^2 \\ w \neq u}}  \frac{\mu_{\beta,V}(dw) \mu_{\beta,V}(du)  }{D^2(\gamma-w,\gamma-u)},
\end{aligned}
$$
By Lemma 5.1.7 of Ambrosio-Gigli-Savar\'{e} \cite{ambrosio1},
we have shown that the function $$f(z,w,u) := 1/D^2(z-w,z-u)$$ is uniformly integrable
with respect to the sequence of measures 
$$
\nu^N(dz, dw, du) := \mb{E} [\delta_{Z^i} \otimes \mu^{N,i}_\bZ \otimes \mu^{N,i}_\bZ ](dz, dw, du).
$$
Proposition 5.1.10 of the same reference \cite{ambrosio1} addresses how to circumvent the lone singularity of $1/D^2(\xi,\eta)$
at $(\xi,\eta) = ((0,0),(0,0))$, hence giving us the convergence
$$
\lim_{N \to \infty} \mathbb{E} \int \int_{ \substack{w,u \in \RR^2 \\ w \neq u}}\frac{\mu^{N,i}_\bZ(dw) \mu^{N,i}_{\bZ}(du)}{D^2(Z^i-w,Z^i-u)}  = 
\int \int_{ \substack{w,u \in \RR^2 \\ w \neq u}}\frac{\mu_{\beta,V}(dw) \mu_{\beta,V}(du)}{D^2(\gamma-w,\gamma-u)}.
$$
This in turn implies
$$
\lim_{N \to \infty} \frac{3}{2} \frac{\mathbb{E}(h_2 * \mu^{N,i}_\bZ)(Z^i)}{N-1} \\
= \int \int_{ \substack{w,u \in \RR^2 \\ w \neq u}}  \frac{\mu_{\beta,V}(dw) \mu_{\beta,V}(du)  }{D^2(\gamma-w,\gamma-u)}  - 
\lim_{N \to \infty} \mathbb{E} \int \int_{\substack{w,u \in \RR^2 \\ w \neq u}} \frac{\mu^{N,i}_\bZ(dw) \mu^{N,i}_\bZ(du)}{D^2(Z^i-w,Z^i-u)}  = 0,
$$
which completes the proof. 
\end{proof}

\begin{proof}[Proof of Corollary \ref{2depsilonnash}]
	For each $1 \leq i \leq N$, let $J^{N,i}_{C_1, C_2}(\psi | \bx, \boldsymbol{\phi}^{-i})$ denote the 
	cost functional associated with the state cost $F^{N,i}(\bz)$ of \eqref{2dstatecost}	with coefficients $C_1, C_2 >0$.
	Now one can follow a similar line of argument as for Theorem  \ref{verificationthm} to verify that $\boldsymbol{\phi}_\beta(\bz) := (-\nabla_{z^k}v^{N,k}_\beta(\bz) )_{k=1}^N$ is 
	a closed loop Nash equilibrium in the sense that for all $1 \leq i \leq N$
\begin{equation}
	\label{coulomboptimality}
	J^{N,i}_{\frac{\beta^2}{8},\frac{3\beta^2}{8}}(\phi_\beta^i | \bx, \boldsymbol{\phi}_\beta^{-i}) = \inf_{\psi  \in \mc{A}^i} J^{N,i}_{\frac{\beta^2}{8},\frac{3\beta^2}{8}}(\psi | \bx, \boldsymbol{\phi}_\beta^{-i}) = \frac{\beta}{4} + \frac{\sigma^2}{2(N-1)},
\end{equation}
where the classes
	$\mc{A}^i \subset \mc{A}^i(\boldsymbol{\phi}^{-i}_\beta)$, $1 \leq i \leq N$, of feedback controls are defined just as in Theorem \ref{verificationthm}.
	To sketch a proof of this required verification theorem in two dimensions,
	one can establish the analog of \eqref{lyapunovsde}, \eqref{finiteh2} by using Lemma \ref{2dehjb} to compute (recall the definition \eqref{2dglobalcostbetter} of $F^N(\bz)$)
		\begin{equation}  \label{2dlyapunovsde}
	\begin{aligned}
	 & dW_\beta(\bZ_t^*) =  \left ( - F^N(\bZ^*_t) + N \lambda^N_\beta 
	  -  \frac{1}{2} \Vert \nabla_\bz W_\beta(\bZ^*_t) \Vert^2 \right )  dt + \frac{\sigma}{\sqrt{N-1}}\nabla_\bz W_\beta(\bZ^*_t) \cdot d\mathbf{B}_t \\
	  & = \left ( - \frac{\Vert \bZ_t^* \Vert^2}{4} - \frac{\beta^2}{6} \sum_{i=1}^N \int \int_{ \substack{w,u \in \RR^2 \\ w \neq u}} \frac{\mu^{N,i}_{\bZ^*_t}(dw) \mu^{N,i}_{\bZ^*_t}(du)}{D^2(Z^{*i}_t-w, Z^{*i}_t-u)}  - \frac{\beta^2}{4} \sum_{i=1}^N \frac{(h_2 * \mu^{N,i}_{\bZ^*_t})(Z^i_t)}{N-1}  \right )  dt \\
	  & + N \left (\frac{\beta}{8} + \frac{\sigma^2}{2(N-1)} \right )  dt + \frac{\sigma}{\sqrt{N-1}}\nabla_\bz W_\beta(\bZ^*_t) \cdot d\mathbf{B}_t 
	  \end{aligned}
	  \end{equation}
which gives 
\begin{equation} \label{2dfiniteh2}
	\mb{E} \sum_{i=1}^N \int_0^{T \wedge T_M} (h_2 * \mu^{N,i}_{\bZ_t^*})(Z_t^{*i}) dt 
	\leq \frac{4(N-1)}{\beta^2} \cdot \left [N \left (\frac{\beta}{8} + \frac{\sigma^2}{2(N-1)} \right )T + W_\beta(\bz_0) + M'  \right ] < \infty.
\end{equation}
where $T_M$, $M'$ are defined as in the proof of Theorem \ref{verificationthm}.
Notice here we no longer require $\beta > \sigma^2$, merely $\beta >0$.  
The ergodicity statement requires a little more effort because one cannot simply exploit convexity, which is special to the one dimensional case; instead, one can proceed in a similar manner as Bolley-Chafa\"{i}-Fontbona \cite{bolley2018dynamics}
or Lu-Mattingly \cite{lu2019geometric} by relying on Lyapunov techniques. 
	
	Returning to the proof, to confirm that $\boldsymbol{\phi}_\beta$ is an approximate closed loop Nash equilibrium, fix an arbitrary $C_2>0$ and $\psi \in \mc{A}^i$ for some $1 \leq i \leq N$. 
	Let
	$$
	\epsilon_N := \left ( C_2 - \frac{3\beta^2}{8} \right )  \frac{\mathbb{E}(h_2 * \mu^{N,i}_\bZ)(Z^i)}{N-1} \geq 0.  
	$$
	Then we have that 
	$$
	\begin{aligned}
	J^{N,i}_{\frac{\beta^2}{8},C_2}(\phi_\beta^i | \bx, \boldsymbol{\phi}_\beta^{-i}) 
	&= J^{N,i}_{\frac{\beta^2}{8},\frac{3\beta^2}{8}}(\phi_\beta^i | \bx, \boldsymbol{\phi}_\beta^{-i}) + \epsilon_N \\
	& \leq  J^{N,i}_{\frac{\beta^2}{8},\frac{3\beta^2}{8}}(\psi | \bx, \boldsymbol{\phi}_\beta^{-i}) + \epsilon_N
	\\
	&\leq 
	J^{N,i}_{\frac{\beta^2}{8},C_2}(\psi | \bx, \boldsymbol{\phi}_\beta^{-i}) + \epsilon_N,
	\end{aligned}
	$$
	where the first inequality uses \eqref{coulomboptimality} and
	the second inequality follows simply by the fact $C_2 \geq 3\beta^2/8$. 
	Hence, $\boldsymbol{\phi}^{-i}_\beta$ is a $\epsilon_N$-closed loop Nash equilibrium with cost functional $J^{N,i}_{\frac{\beta^2}{8},C_2}$. 
	But by Theorem \ref{big2dtheorem}, $\epsilon_N$ goes to $0$ as $N \to \infty$, so $\boldsymbol{\phi}^{-i}_\beta$ is in fact an approximate closed loop Nash equilibrium with limiting optimal cost $\beta/4$, as required.
	
	This completes the proof for the closed loop case;
	the open loop case is similar, but relies instead on the weaker
	assumption $C_2 \geq \beta/8$ and 
	requires exploiting the potential 
	structure to reduce consideration to optimality in the auxiliary global problem.
\end{proof}

\subsection*{Acknowledgments} 

The second author would like to thank many people: Ramon Van Handel, for discussing ergodic theory and a toy version of the open loop model; 
Mykhaylo Shkolnikov, for introducing him to Section 3 in \cite{vgmsedge} and for important suggested edits; and Daniel Lacker, for helpful comments on early drafts.
The first author was partially supported by NSF \#DMS--1716673,
and the first and second author by ARO \#W911NF--17--1--0578.
The second and third authors also thank IPAM for hosting them during final edits of the initial version of this paper.

\bibliography{buses}

\def\cprime{$'$} \def\Dbar{\leavevmode\lower.6ex\hbox to 0pt{\hskip-.23ex
  \accent"16\hss}D} \def\cprime{$'$}
\begin{thebibliography}{10}

\bibitem{parkedcars}
A.Y. Abul-Magd.
\newblock Modelling gap-size distribution of parked cars using random-matrix
  theory.
\newblock {\em Physica A: Statistical Mechanics and its Applications},
  368(2):536 -- 540, 2006.

\bibitem{ambrosio1}
Luigi Ambrosio, Nicola Gigli, and Giuseppe Savar\'{e}.
\newblock {\em Gradient flows in metric spaces and in the space of probability
  measures}.
\newblock Lectures in Mathematics ETH Z\"urich. Birkh\"auser Verlag, 2005.

\bibitem{agz1}
Greg~W. Anderson, Alice Guionnet, and Ofer Zeitouni.
\newblock {\em An introduction to random matrices}, volume 118 of {\em
  Cambridge Studies in Advanced Mathematics}.
\newblock Cambridge University Press, Cambridge, 2010.

\bibitem{arapostathis2012ergodic}
Ari Arapostathis, Vivek~S Borkar, and Mrinal~K Ghosh.
\newblock {\em Ergodic control of diffusion processes}, volume 143.
\newblock Cambridge University Press, 2012.

\bibitem{umich2}
Jinho Baik.
\newblock Circular unitary ensemble with highly oscillatory potential.
\newblock {\em arXiv:1306.0216}, 2013.

\bibitem{buses2}
Jinho Baik, Alexei Borodin, Percy Deift, and Toufic Suidan.
\newblock A model for the bus system in {C}uernavaca ({M}exico).
\newblock {\em J. Phys. A}, 39(28):8965--8975, 2006.

\bibitem{bardi1}
Martino Bardi.
\newblock Explicit solutions of some linear-quadratic mean field games.
\newblock {\em Networks and heterogeneous media}, 7(2):243--261, 2012.

\bibitem{gbjm1}
Guy Barles and Joao Meireles.
\newblock On unbounded solutions of ergodic problems in $\mathbb{R}^m$ for
  viscous {H}amilton--{J}acobi equations.
\newblock {\em Comm. in Partial Differential Equations}, 41(12):1985--2003,
  2016.

\bibitem{bayrakter3}
Erhan Bayraktar, Jaksa Cvitanic, and Yuchong Zhang.
\newblock {Large Tournament Games}.
\newblock {\em SSRN}, 2018.

\bibitem{bayrakter1}
Erhan Bayraktar and Yuchong Zhang.
\newblock A rank-based mean field game in the strong formulation.
\newblock {\em Electronic Communications in Probability}, 21, 2016.

\bibitem{arous1}
G\'{e}rard Ben~Arous and Alice Guionnet.
\newblock Large deviations for {W}igner's law and {V}oiculescu's
  non-commutative entropy.
\newblock {\em Probab. Theory Related Fields}, 108(4):517--542, 1997.

\bibitem{berman2016propagation}
Robert~J Berman and Magnus {\"O}nnheim.
\newblock Propagation of chaos for a class of first order models with singular
  mean field interactions.
\newblock {\em arXiv:1610.04327}, 2016.

\bibitem{pbrs1}
Philippe Biane and Roland Speicher.
\newblock Free diffusions, free entropy and free {F}isher information.
\newblock {\em Ann. Inst. H. Poincar\'e Probab. Statist.}, 37(5):581--606,
  2001.

\bibitem{regularityofflow}
Piotr Biler, Grzegorz Karch, and R\'egis Monneau.
\newblock Nonlinear diffusion of dislocation density and self-similar
  solutions.
\newblock {\em Comm. Math. Phys.}, 294(1):145--168, 2010.

\bibitem{gb1}
Gordon Blower.
\newblock {\em Random matrices: high dimensional phenomena}, volume 367 of {\em
  London Mathematical Society Lecture Note Series}.
\newblock Cambridge University Press, Cambridge, 2009.

\bibitem{bolley2018dynamics}
Fran{\c{c}}ois Bolley, Djalil Chafa{\"\i}, and Joaqu{\'\i}n Fontbona.
\newblock Dynamics of a planar {C}oulomb gas.
\newblock {\em The Annals of Applied Probability}, 28(5):3152--3183, 2018.

\bibitem{nonlocaltolocal}
Pierre Cardaliaguet.
\newblock The convergence problem in mean field games with local coupling.
\newblock {\em Applied Mathematics \& Optimization}, 76(1):177--215, 2017.

\bibitem{cdll1}
Pierre Cardaliaguet, Fran{\c{c}}ois Delarue, Jean-Michel Lasry, and
  Pierre-Louis Lions.
\newblock The master equation and the convergence problem in mean field games.
\newblock {\em arXiv:1509.02505}, 2015.

\bibitem{pcap1}
Pierre Cardaliaguet and Alessio Porretta.
\newblock Long time behavior of the master equation in mean-field game theory.
\newblock {\em arXiv:1709.04215}, 2017.

\bibitem{bible1}
Ren\'e Carmona and Fran\c{c}ois Delarue.
\newblock {\em Probabilistic theory of mean field games with applications.
  {I}}, volume~83 of {\em Probability Theory and Stochastic Modelling}.
\newblock Springer, 2018.

\bibitem{lacker1}
Ren\'e Carmona and Daniel Lacker.
\newblock A probabilistic weak formulation of mean field games and
  applications.
\newblock {\em Ann. Appl. Probab.}, 25(3):1189--1231, 2015.

\bibitem{carrillo1}
Jos\'e~A. Carrillo, Lucas C.~F. Ferreira, and Juliana~C. Precioso.
\newblock A mass-transportation approach to a one dimensional fluid mechanics
  model with nonlocal velocity.
\newblock {\em Adv. Math.}, 231(1):306--327, 2012.

\bibitem{cl1}
C.~Carton-Lebrun.
\newblock Product properties of {H}ilbert transforms.
\newblock {\em J. Approximation Theory}, 21(4):356--360, 1977.

\bibitem{ecdl2}
Emmanuel C\'epa and Dominique L\'epingle.
\newblock Diffusing particles with electrostatic repulsion.
\newblock {\em Probab. Theory Related Fields}, 107(4):429--449, 1997.

\bibitem{ecdl1}
Emmanuel C\'epa and Dominique L\'epingle.
\newblock Brownian particles with electrostatic repulsion on the circle:
  {D}yson's model for unitary random matrices revisited.
\newblock {\em ESAIM Probab. Statist.}, 5:203--224, 2001.

\bibitem{chafai2018concentration}
Djalil Chafai, Adrien Hardy, and Myl{\`e}ne Ma{\"\i}da.
\newblock Concentration for coulomb gases and coulomb transport inequalities.
\newblock {\em Journal of Functional Analysis}, 275(6):1447--1483, 2018.

\bibitem{deiftpotential}
P.~Deift, T.~Kriecherbauer, and K.~T.-R. McLaughlin.
\newblock New results on the equilibrium measure for logarithmic potentials in
  the presence of an external field.
\newblock {\em J. Approx. Theory}, 95(3):388--475, 1998.

\bibitem{deift1}
Percy Deift.
\newblock Some open problems in random matrix theory and the theory of
  integrable systems.
\newblock {\em Contemporary Mathematics}, 458:419, 2008.

\bibitem{theOG}
Freeman~J Dyson.
\newblock A {B}rownian-motion model for the eigenvalues of a random matrix.
\newblock {\em Journal of Mathematical Physics}, 3(6):1191--1198, 1962.

\bibitem{umich}
Anthony {Fader}.
\newblock {The Gap Size Distribution of Parked Cars and the Coulomb Gas Model}.
\newblock {\em
  \texttt{\emph{http://dept.math.lsa.umich.edu/undergrad/REU/ArchivedREUpapers/Faderpaper06.pdf}}},
  2006.

\bibitem{fischer1}
Markus Fischer.
\newblock On the connection between symmetric {$N$}-player games and mean field
  games.
\newblock {\em Ann. Appl. Probab.}, 27(2):757--810, 2017.

\bibitem{pf1}
Peter~J. Forrester.
\newblock {\em Log-Gases and Random Matrices}.
\newblock Princeton University Press, 2010.

\bibitem{gangbo1}
Wilfrid Gangbo and Andrzej \'{S}wi\c{e}ch.
\newblock Existence of a solution to an equation arising from the theory of
  mean field games.
\newblock {\em J. Differential Equations}, 259(11):6573--6643, 2015.

\bibitem{gangbo2019differentiability}
Wilfrid Gangbo and Adrian Tudorascu.
\newblock On differentiability in the wasserstein space and well-posedness for
  hamilton--jacobi equations.
\newblock {\em Journal de Math{\'e}matiques Pures et Appliqu{\'e}es},
  125:119--174, 2019.

\bibitem{vgmsedge}
Vadim Gorin and Mykhaylo Shkolnikov.
\newblock Interacting particle systems at the edge of multilevel {D}yson
  {B}rownian motions.
\newblock {\em Adv. Math.}, 304:90--130, 2017.

\bibitem{pgjm1}
Piotr Graczyk and Jacek Ma{\l}ecki.
\newblock Strong solutions of non-colliding particle systems.
\newblock {\em Electronic Journal of Probability}, 19, 2014.

\bibitem{guionnet2004first}
Alice Guionnet.
\newblock First order asymptotics of matrix integrals; a rigorous approach
  towards the understanding of matrix models.
\newblock {\em Communications in mathematical physics}, 244(3):527--569, 2004.

\bibitem{guionnet2002large}
Alice Guionnet and Ofer Zeitouni.
\newblock Large deviations asymptotics for spherical integrals.
\newblock {\em Journal of functional analysis}, 188(2):461--515, 2002.

\bibitem{nycsubway}
Aukosh Jagannath and Thomas Trogdon.
\newblock Random matrices and the {N}ew {Y}ork {C}ity subway system.
\newblock {\em Physical Review E}, 96(3):030101, 2017.

\bibitem{pedestrian}
Daniel Jezbera, David Kordek, Jan K{\v{r}}{\'\i}{\v{z}}, Petr {\v{S}}eba, and
  Petr {\v{S}}roll.
\newblock Walkers on the circle.
\newblock {\em Journal of Statistical Mechanics: Theory and Experiment},
  2010(01):L01001, 2010.

\bibitem{buses1}
Milan Krb{\'a}lek and Petr {\v{S}}eba.
\newblock The statistical properties of the city transport in {C}uernavaca
  ({M}exico) and random matrix ensembles.
\newblock {\em Journal of Physics A: Mathematical and General}, 33(26):L229,
  2000.

\bibitem{lacker3}
Daniel Lacker.
\newblock A general characterization of the mean field limit for stochastic
  differential games.
\newblock {\em Probab. Theory Related Fields}, 165(3-4):581--648, 2016.

\bibitem{lackerinprep}
Daniel Lacker.
\newblock On the convergence of closed-loop {N}ash equilibria to the mean field
  game limit.
\newblock {\em arXiv:1808.02745}, 2018.

\bibitem{thaleia1}
Daniel Lacker and Thaleia Zariphopoulou.
\newblock Mean field and n-agent games for optimal investment under relative
  performance criteria.
\newblock {\em arXiv:1703.07685}, 2017.

\bibitem{jlpl1}
Jean-Marie Lasry and Pierre-Louis Lions.
\newblock Nonlinear elliptic equations with singular boundary conditions and
  stochastic control with state constraints.
\newblock {\em Mathematische Annalen}, 283(4):583--630, 1989.

\bibitem{mfg}
Jean-Michel Lasry and Pierre-Louis Lions.
\newblock Mean field games.
\newblock {\em Jpn. J. Math.}, 2(1):229--260, 2007.

\bibitem{france}
G{\'e}rard Le~Ca{\"e}r and Renaud Delannay.
\newblock The administrative divisions of mainland {F}rance as 2{D} random
  cellular structures.
\newblock {\em Journal de Physique I}, 3(8):1777--1800, 1993.

\bibitem{lionslectures}
Pierre-Louis Lions.
\newblock Cours au coll\`ege de france.

\bibitem{liu2016propagation}
Jian-Guo Liu and Rong Yang.
\newblock Propagation of chaos for large brownian particle system with coulomb
  interaction.
\newblock {\em Research in the Mathematical Sciences}, 3(1):40, 2016.

\bibitem{lu2019geometric}
Yulong Lu and Jonathan~C Mattingly.
\newblock Geometric ergodicity of langevin dynamics with coulomb interactions.
\newblock {\em arXiv preprint arXiv:1902.00602}, 2019.

\bibitem{genes}
Feng Luo, Jianxin Zhong, Yunfeng Yang, and Jizhong Zhou.
\newblock Application of random matrix theory to microarray data for
  discovering functional gene modules.
\newblock {\em Phys. Rev. E}, 73:031924, Mar 2006.

\bibitem{matytsin1994large}
A~Matytsin.
\newblock On the large-n limit of the itzykson-zuber integral.
\newblock {\em Nuclear Physics B}, 411(2-3):805--820, 1994.

\bibitem{meckes2015rate}
Elizabeth~S Meckes and Mark~W Meckes.
\newblock A rate of convergence for the circular law for the complex ginibre
  ensemble.
\newblock In {\em Annales de la Facult{\'e} des sciences de Toulouse:
  Math{\'e}matiques}, volume~24, pages 93--117, 2015.

\bibitem{mehta}
Madan~Lal Mehta.
\newblock {\em Random matrices}, volume 142.
\newblock Elsevier, 2004.

\bibitem{menon2017complex}
Govind Menon.
\newblock The complex {B}urgers’ equation, the {HCIZ} integral and the
  {C}alogero-{M}oser system.

\bibitem{nutz2}
Marcel Nutz and Yuchong Zhang.
\newblock A mean field competition.
\newblock {\em arXiv:1708.01308}, 2017.

\bibitem{primes}
Andrew~M Odlyzko.
\newblock On the distribution of spacings between zeros of the zeta function.
\newblock {\em Math. Comp.}, 48(177):273--308, 1987.

\bibitem{parkedcars2}
S.~Rawal and G.J. Rodgers.
\newblock Modelling the gap size distribution of parked cars.
\newblock {\em Physica A: Statistical Mechanics and its Applications},
  346(3-4):621 -- 630, 2005.

\bibitem{lrzs1}
L.~C.~G. Rogers and Z.~Shi.
\newblock Interacting {B}rownian particles and the {W}igner law.
\newblock {\em Probab. Theory Related Fields}, 95(4):555--570, 1993.

\bibitem{goegue}
Petr {\v{S}}eba.
\newblock Parking in the city.
\newblock {\em Acta Physica Polonica A}, 112(4):681 -- 690, 2007.

\bibitem{perchedbirds}
Petr {\v{S}}eba.
\newblock Parking and the visual perception of space.
\newblock {\em Journal of Statistical Mechanics: Theory and Experiment},
  2009(10):L10002, 2009.

\bibitem{serfaty2015coulomb}
Sylvia Serfaty.
\newblock {\em Coulomb gases and Ginzburg--Landau vortices}.
\newblock 2015.

\bibitem{villani2008optimal}
C{\'e}dric Villani.
\newblock {\em Optimal transport: old and new}, volume 338.
\newblock Springer Science \& Business Media, 2008.

\bibitem{voic1}
Dan Voiculescu.
\newblock The analogues of entropy and of {F}isher's information measure in
  free probability theory. {I}.
\newblock {\em Comm. Math. Phys.}, 155(1):71--92, 1993.

\bibitem{voic6}
Dan Voiculescu.
\newblock The analogues of entropy and of {F}isher's information measure in
  free probability theory. {VI}. {L}iberation and mutual free information.
\newblock {\em Adv. Math.}, 146(2):101--166, 1999.

\bibitem{pw1}
Piotr Warcho{\l}.
\newblock Buses of {C}uernavaca -- an agent-based model for universal random
  matrix behavior minimizing mutual information.
\newblock {\em Journal of Physics A: Mathematical and Theoretical},
  51(26):265101, 2018.

\end{thebibliography}
\bibliographystyle{plain}

\end{document}